%% file: Article.tex
\documentclass[12pt]{article}

\input{preamble}

\title{\textbf{Weak well-posedness for a class of degenerate L\'evy-driven SDEs with H\"older continuous coefficients}}

{\author[1,2,*]{\textbf{L. Marino}}}
\author[1,3,$\dagger$]{\textbf{S. Menozzi}}

\affil[1]{\footnotesize{Laboratoire de Mod\'elisation Math\'ematique d'Evry (LaMME), UMR CNRS 8071,
 Universit\'e Paris-Saclay, Universit\'e d'Evry Val d'Essonne, $23$
Boulevard de France $91037$ Evry, France.}}
\affil[2]{\footnotesize Dipartimento di Matematica, Universit\`a di Pavia, Via Adolfo Ferrata $5$, $27100$ Pavia,
Italy.}
\affil[3]{\footnotesize Laboratory of Stochastic Analysis, Higher School of Economics (HSE), Pokrovsky Boulevard, 11, Moscow, Russian Federation.}
\affil[*]{\footnotesize{lorenzo.marino@univ-evry.fr}}
\affil[$\dagger$]{\footnotesize{stephane.menozzi@univ-evry.fr}}

\begin{document}
\maketitle
\begin{abstract}
In this article, we study the effects of the propagation of a non-degenerate L\'evy noise through a chain of deterministic differential equations whose coefficients are H\"older continuous and satisfy a weak H\"ormander-like condition. In particular, we assume some non-degeneracy with respect to the components which transmit the noise.
Moreover,  we characterize, for some specific dynamics, through suitable counter-examples, the almost sharp regularity exponents that ensure the weak well-posedness for the associated SDE. As a by-product of our approach, we also derive some Krylov-type estimates for the density of the weak solutions of the considered SDE.
\end{abstract}

{\small{\textbf{Keywords:} degenerate L\'evy driven SDEs; Well-posedness of martingale problem, Peano counter-example.}}

{\small{\textbf{MSC:} Primary: $60$H$10$, $34$F$05$; Secondary: $35$K$65$, $35$R$09$.}}

\section{Introduction}

We investigate the effects of the propagation of a $d$-dimensional L\'evy noise through a chain of $n\ge 2$ differential equations.
Namely, we are interested in a degenerate,  L\'evy-driven stochastic differential equation (SDE in short) of the following form:
\begin{equation}
\label{eq:Intro}
  \begin{cases}
    dX^1_t \, = \,\left[[A_t]_{1,1}X_t^1+\cdots+[A_t]_{1,n}X_t^n+ F_1(t,X^1_t,\dots,X^n_t)\right]dt + \sigma(t,X^1_{t-},\dots,X^n_{t-})dZ_t,\\
    dX^2_t \, = \, \left[ [A_t]_{2,1}X^1_t+\cdots+[A_t]_{2,n}X^n_t +F_2(t,X^2_t,\dots,X^n_t)\right] dt,\\
    dX^3_t \, = \, \left[ [A_t]_{3,2}X^2_t+\cdots+[A_t]_{\textcolor{black}{3,n}}X^n_t+F_3(t,X^3_t,\dots,X^n_t)\right]dt,\\
    \vdots\\
    dX^n_t \, = \, \left[ [A_t]_{\textcolor{black}{n,n-1}} X^{n-1}_t+[A_t]_{n,n} X^{n}_t+F_n(t,X^n_t)\right]dt,
  \end{cases}
\end{equation}
where for $i\in \llbracket 1,n\rrbracket$ ($\llbracket \cdot, \cdot \rrbracket$ denotes the set of all the integers in the interval),  $X_t^i$ is $\R^{d_i} $ valued, with $d_1=d $ and $ d_i\ge 1,\ i\in \llbracket 2,n\rrbracket$, \textcolor{black}{with $d_i\le d_{i-1} $}. Set $N=\sum_{i=1}^n d_i $. We suppose that the
$F_i\colon [0,+\infty)\times\R^{{\scriptscriptstyle{\sum_{j=i}^n}} d_j}\to \R^d$,  $\sigma\colon [0,+\infty)\times\R^N\to \R^d\otimes \R^d$ are  Borel \textcolor{black}{measurable} and respectively locally bounded
and uniformly elliptic and bounded.

We also assume the entries $([A_t]_{ij})_{1\le i\le n,\ i-1\le j\le n} $ are Borel \textcolor{black}{measurable} bounded and such that
 the blocks $[A_t]_{i,i-1}$ in $\R^{d_{i}}\otimes \R^{d_{i-1}}$, $2\le i \le n $ have rank $d_i$, uniformly in time.
 This is a kind of non-degeneracy assumption which can be viewed as weak H\"ormander-like condition. It actually precisely allows the noise to propagate into the system.

Eventually, the noise $\{Z_t\}_{t\ge 0}$ belongs to a class of $d$-dimensional, symmetric, L\'evy processes with suitable properties.
\textcolor{black}{In particular, we will assume that its L\'evy measure is absolutely continuous with respect to the L\'evy measure of a non-degenerate $\alpha$-stable process in the sub-critical regime ($\alpha$ in $(1,2]$) and its Radon-Nikodym derivative enjoys some natural properties.
When working with non-trivial diffusion coefficients, we will additionally suppose that the corresponding $\alpha$-stable process is rotationally invariant.} The class of processes $\{Z_t\}_{t\ge0}$ we can consider includes, for example, the tempered, the layered or the relativistic $\alpha$-stable processes. In the case of an additive noise, cylindrical stable processes could be handled as well.

Here, the major issue is linked with the specific degenerate framework we consider. Indeed, the noise only acts on the first component of the dynamics \eqref{eq:Intro} and it implies, in particular, that the random perturbation on the $i$-th
line of SDE \eqref{eq:Intro} only comes from the previous $(i-1)$-th one, through the non-degeneracy of the matrixes $[A_t]_{i,i-1}$. Hence, the smoothing effect associated with the L\'evy noise decreases along the chain,  making thus more and more difficult to regularize by noise the furthest lines of Equation \eqref{eq:Intro}. \newline
We nevertheless prove the weak well-posedness, i.e.\ the existence and uniqueness in law, for the above SDE \eqref{eq:Intro} when the drift $F=(F_1,\dots,F_n)$ and $ \sigma$ lie  in a suitable anisotropic H\"older space with multi-indices of regularity. We assume that $F_1$ and $\sigma $ have spatial H\"older regularity $\beta^1>0 $ with respect to the $j$-th variable. We highlight already that we could have considered different regularity indexes $\beta_j^1 $ for the regularity of $F_1$ with respect to the $j$-th variable. We keep only one common index for notational simplicity. We also suppose that for fixed $j\in \llbracket2,n \rrbracket  $, $(F_2,\cdots,F_j) $ has H\"older regularity $\beta^j$  with respect to the $j$-th variable, where:
\[ \beta^j\,\, \in \,\, \Bigl(\frac{1+\alpha(j-2)}{1+\alpha(j-1)};1\Bigr].\]
We indeed recall that from the dynamics \eqref{eq:Intro} the variable $x_j$ does not appear in the chain after level $j$.

Furthermore, we will show through suitable counter-examples that the above threshold for the regularity exponents $\beta^j$ is ``almost'' sharp for a perturbation of the $j^{{\rm th}} $ level of the chain.
Such counter-examples are based on Peano-type dynamics adapted to our degenerate, fractional framework.

Models of the form \eqref{eq:Intro} naturally appear in various scientific contexts: in physics, for the
analysis of anomalous diffusions phenomena or for Hamiltonian models in turbulent regimes (see e.g.\ \cite{Baeumer:Benson:Meerschaert01}, \cite{Cushman:Park:Kleinfelter:Moroni05}, \cite{Eckmann:Pillet:Rey-Bellet99}); in mathematical
finance and econometrics, for example in the pricing of Asian options (see e.g.\ \cite{Jeamblanc:Yor:Chesney09}, \cite{Brockwell01}, \cite{Barndorff-Nielsen:Shephard01}). In particular, models that consider L\'evy noises, such as SDE \eqref{eq:Intro}, seem more natural and realistic for many
applications since they allow the presence of jumps.

\paragraph{Weak well posedness for non-degenerate stable SDEs.}
The topic of weak well-posedness for non-degenerate (i.e.\ $d=N$) SDEs of the form:
\begin{equation}
\label{eq:non-deg_SDE}
X_t \, = \, x +\int_0^tF(X_s) ds +Z_t, \quad t\ge 0,
\end{equation}
where $\{Z_t\}_{t\ge0}$ is a symmetric $\alpha$-stable process on $\R^N$, has been widely studied in the last decades, especially in the diffusive, local setting, i.e.\ when $\alpha=2$ and $\{Z_t\}_{t\ge0}$ is a Brownian Motion, and it is now
quite well-understood. We can first refer to the seminal work \cite{book:Stroock:Varadhan79} where the Authors considered additionally a multiplicative noise with bounded drift and  non-degenerate, continuous in space diffusion coefficient. We recall moreover that in the framework of \eqref{eq:non-deg_SDE} with bounded drift, strong uniqueness also holds (cf. \cite{Veretennikov80}).\newline
SDEs like \eqref{eq:non-deg_SDE} with a proper $\alpha$-stable process ($\alpha<2$) were firstly investigated in \cite{Tanaka:Tsuchiya:Watanabe74} where the weak well-posedness was obtained for the one-dimensional case when the drift $F$ is bounded, continuous and the L\'evy exponent $\Phi$ of $\{Z_t\}_{t\ge 0}$ satisfies $\Re \Phi(\xi)^{-1}= \textcolor{black}{O(1/\vert \xi \vert)}$ if $\vert \xi \vert \to \infty$. The multidimensional case ($d>1$) can be similarly obtained  following \cite{Komatsu83} if the drift is bounded, continuous and the law of $\{Z_t\}_{t\ge 0}$ admits a density with respect to the Lebesgue measure on $\R^d$. Equations as \eqref{eq:non-deg_SDE} with drift in some suitable $L^p$-spaces and non-degenerate noise were also considered in \cite{Jin18} (see also the references therein). We can eventually quote the recent work by Krylov who obtained even strong uniqueness for Brownian SDEs with drifts in critical $L^p$-spaces, see \cite{Krylov21}.

In recent years, SDEs driven by singular (distributional) drift have gained a lot of interest, especially in the Brownian setting, where they arise as a model for diffusions in
random media (see e.g. \cite{Flandoli:Russo:Wolf03},\cite{Flandoli:Russo:Wolf04},\cite{Flandoli:Issoglio:Russo17}, \cite{Delarue:Diel16},
\cite{Cannizzaro:Chouk18}).\newline
In the non-local
$\alpha$-stable framework, \textcolor{black}{we can refer to} \cite{Athreya:Butkovsky:Mytnik18} where the
authors \textcolor{black}{derived, in the
one-dimensional case, strong well-posedness} for
a time-homogeneous drift
of \textcolor{black}{\textit{negative}} H\"older
regularity strictly
greater than
$(1-\alpha)/2$. \textcolor{black}{More precisely, the drifts therein belong to the Besov space $B_{\infty,\infty}^{\theta}, \theta\in \big((1-\alpha)/2,0\big) $ which consists in the distributional derivatives of functions being $1+\theta $ H\"older continuous}.  On the same side, the almost simultaneous works \cite{Ling:Zhao19} and \cite{Chaudru:Menozzi20} take into account time-homogeneous and time-inhomogeneous, respectively, \textcolor{black}{multidimensional} distributional drift\textcolor{black}{s} in general Besov spaces with suitable
conditions on the parameters. \textcolor{black}{Both paper address the weak well-posedness of the corresponding SDE}.
 Let us also remark that in the one-dimensional framework, \textcolor{black}{the same} regularity \textcolor{black}{thresholds appear} on the drift for the strong and the weak well-posedness, since it is possible to exploit \textcolor{black}{\textit{local time} type} arguments (see \textcolor{black}{\cite{Chaudru:Menozzi20}, \cite{Athreya:Butkovsky:Mytnik18} and also \cite{Flandoli:Issoglio:Russo17}, \cite{Bass:Chen01} in the diffusive setting}).

These results rely on Young integrals in order to give a meaningful sense to the dynamics. Beyond the Young regime, we instead
refer to \cite{kremp:Perkowski20} \textcolor{black}{where techniques such as paracontrolled products (which have also been popular in the recent developments in the SPDE theory)} are exploited to analyze the martingale problem associated \textcolor{black}{with} a
time-inhomogeneous drift of regularity index  strictly greater than $(2-2\alpha)/3$.

Moreover, we would like to remark that the above works concerned the so-called \emph{sub-critical} case, i.e.\ when $\alpha>1$. Indeed, SDEs like
\eqref{eq:non-deg_SDE} are much more difficult to handle if $\alpha\le 1$ since in this case, the noise does not dominate the system \textcolor{black}{for small time scales}. Two recent
works along this path are \cite{Zhao19} and \cite{Chaudru:Menozzi:Priola20} where \textcolor{black}{the authors} consider $\alpha<1$, $(1-\alpha)$-H\"older drift $F$ and $\alpha=1$,
continuous drift, respectively. We also mention that for H\"older drifts, the well-posedness of the associated martingale problem can be obtained following
\cite{Mikulevicius:Pragarauskas14} if $F$ is bounded or through the Schauder estimates given in \cite{Chaudru:Menozzi:Priola19} when $F$ is unbounded.

\paragraph{Regularization by noise in a degenerate setting.}
All the above results present a common phenomenon that, following the terminology in \cite{book:Flandoli11}, is usually called \emph{regularization by
noise}. This occurs when a deterministic ODE is ill-posed (for example if the drift is less than Lipschitz) but its stochastic counterpart
(SDE) is well posed in a strong or a weak sense. \newline
To obtain such phenomenon, the noise plays a fundamental role. A usual assumption is that the noise should act on every line of the dynamics, regularizing the
coefficients. It is then clear that in our degenerate framework, when the noise acts only on the first component of the chain \eqref{eq:Intro}, the situation is
even more delicate. In order to obtain some kind of regularization effect in this case, we need that the noise propagates through the system, reaching all the
lines of Equation \eqref{eq:Intro}. A typical assumption ensuring such type of behaviour is the so-called H\"ormander condition for hypoellipticity (cf.
\cite{Hormander67}).\newline
From the structure of the equation \eqref{eq:Intro} at hand, we will consider a \textit{weak} type H\"ormander condition, i.e. up to some regularization of the diffusion coefficient, the drift is needed to span the space through
Lie bracketing.

In the Hamiltonian setting $n=2$, when $\{Z_t\}_{t\ge0}$ is a Brownian Motion and for a more general, non-linear, drift than in \eqref{eq:Intro} which still satisfies a weak H\"ormander type condition,  Chaudru de Raynal showed in \cite{Chaudru18} that the associated SDE  is weakly well-posed as soon as the drift is H\"older continuous in the degenerate variable with regularity index strictly greater than $1/3$. It was also established through an appropriate counter-example,  that the $1/3$-threshold is (almost) sharp for the second component of the drift. Such a result has been then extended in \cite{Chaudru:Menozzi17} in order to consider the more general case of $n$ oscillators. Therein, the  regularity thresholds that ensure weak uniqueness also depend on the variable and the level of the chain. This seems intuitively clear, the further the variable in the oscillator chain, the larger its typical time scale, the weaker the regularity needed to regularize components which are above that variable in the chain. Also, some corresponding Krylov type estimates, giving existence and integrability properties of the density of the SDE are derived.
We can mention as well the recent work by Gerencs\'er \cite{gerencser20} who \textcolor{black}{obtains} similar regularization properties for the iterated time integrals of a fractional Brownian motion.

In the jump case, the situation is much more delicate. Within the proper regularization by noise framework (when the coefficients are less than Lipschitz continuous),  we cite \cite{Huang:Menozzi16} where the
Authors showed the weak well-posedness for \eqref{eq:Intro} with $F=0$ and a H\"older continuous diffusion coefficient, under  some constraints on the dimensions
$d$,$n$. In that framework, the Authors obtained as well same point-wise density estimates. The driving noises considered were stable and tempered stable processes.
\newline
Finally, we mention that it is possible to derive the weak well-posedness of dynamics \eqref{eq:Intro} via the martingale formulation, exploiting the Schauder
estimates given in \cite{Hao:Wu:Zhang19} for the kinetic model ($n=2$). In that work, the Authors actually characterized conditions for strong uniqueness, using Littlewood-Paley decomposition techniques.

We will here proceed through a perturbative approach. Namely, we will expand the formal generator associated with \eqref{eq:Intro} around the one of a well understood process, with possibly time inhomogeneous coefficients which are anyhow frozen in space. We will call such a process a \textit{proxy}. The most natural candidate to be a \textit{proxy} for \eqref{eq:Intro} is a degenerate Ornstein-Uhlenbeck process. In the case of time homogeneous coefficients, Priola and Zabczyk  established in \cite{Priola:Zabczyk09} existence of the density for such processes under the same previously indicated non-degeneracy conditions on the matrix $A$ (which turn out to be equivalent in that setting to the well known Kalman condition).
\paragraph{Intrinsic difficulties associated with large jumps.}
When $Z$ is a strict stable process, the  density of the corresponding degenerate Ornstein-Uhlenbeck process can somehow be related to the one of a multi-scale stable process which has however a very singular associated \textit{spectral measure} (spherical part of the $\alpha $-stable L\'evy measure) on $\mathbb S^{N-1}$, see e.g. \cite{Huang:Menozzi16}, \cite{Huang:Menozzi:Priola19} and Proposition \ref{prop:Decomposition_Process_X} below.
From Watanabe \cite{Watanabe07}, it is known that the tails of stable densities are highly related to the nature of this spectral measure. Specifically, the concentration properties worsen when the measure becomes singular. This renders delicate the characterization of the smoothing properties for the proxy, especially when it depends on parameters and that one would like to obtain  estimates which are uniform w.r.t. those parameters (see Proposition \ref{prop:Smoothing_effect} and Section \ref{SEZIONE_SPIEGAZIONE_CONGELAMENTO_E_ALTRI} below).

Even for smooth coefficients, the stable like jump setting is much more delicate to establish the existence of the density for (degenerate) SDEs. For multiplicative noises, we cannot indeed rely on the flow techniques considered in \cite{book:bichteler:Gravereaux:Jacod87} or \cite{book:kunita19} in the non-degenerate case, and L\'eandre in the degenerate one, see \cite{Leandre85},\cite{Leandre88}.
Still for smooth coefficients, we can refer to the work of Zhang \cite{Zhang14} who obtained existence and smoothness results for the density of equations of type \eqref{eq:Intro} in arbitrary dimension, for a possibly more general non linear drift, still satisfying a weak H\"ormander type condition when the driving process is a rotationally invariant stable process. The strategy therein is based on the  \textit{subordinated} Malliavin calculus, which consists in applying the \textit{usual} Malliavin calculus techniques on a Brownian motion observed along the path of an independent $\alpha $-stable subordinator. In whole generality a \textit{complete} version of the H\"ormander theorem in the jump case seems to lack.
We can refer to the work by Cass \cite{Cass09} who gets smoothness of the density in the weak H\"ormander framework under technical restrictions.


\paragraph{Complete model and assumptions.}
Let us now specify the assumptions on equation \eqref{eq:Intro} that we rewrite in the shortened form:
\begin{equation}\label{eq:SDE}
dX_t \, = \, G(t,X_t)dt + B\sigma(t,X_{t-})dZ_t, \quad t\ge 0,
\end{equation}
where  $B$ is the embedding from $\R^d$ to $\R^N$ given in matricial form as
\[B:=\begin{pmatrix}
            I_{d\times d}, & 0_{d\times (N-d)},
\end{pmatrix}^t\]
and $G(t,x)=A_tx+F(t,x) $ with:
      \begin{equation}\label{eq:def_matrix_A}
A_t \, := \, \begin{pmatrix}
               [A_t]_{1,1}& \dots         & \dots         & \dots     & [A_t]_{1,n} \\
               [A_t]_{2,1}       & [A_t]_{2,2} & \dots         & \dots     & [A_t]_{2,n}\\
               0 & [A_t]_{3,2}       & \ddots & \ddots     & \vdots \\
               \vdots        & \ddots        & \ddots        & \ddots    & \vdots        \\
               0 & \dots         & 0 & [A_t]_{n,n-1} & [A_t]_{n,n}
             \end{pmatrix}.
\end{equation}

A classical assumption in this degenerate framework (cf.\ \cite{book:Stroock:Varadhan79}, \cite{Krylov04}, \cite{Chaudru:Menozzi17}) is the \emph{uniform ellipticity} of the underlying non-degenerate component of the diffusion matrix at any fixed space-time point. Namely,
\begin{description}
  \item[{[UE]}] There exists a constant $\eta>1$ such that for any $t\ge 0$ and any $x$ in $\R^N$, it holds that
  \[\eta^{-1}\vert \xi \vert^2 \, \le \, \sigma(t,x)\xi\cdot \xi  \, \le\,\eta
\vert \xi \vert^2, \quad \xi \in \R^d,
\]
\end{description}
where ``$\cdot$'' stands for the inner product on the smaller space $\R^d$.\newline
We will suppose that the drift $G(t,x)=A_tx+F(t,x)$ has a particular ``upper diagonal'' structure and its sub-diagonal elements are linear and non-degenerate, i.e.\
\begin{description}
  \item[{[H]}]
  \begin{itemize}
      \item $F=(F_1,\dots,F_n)\colon [0,\infty)\times\R^N\to \R^N$ is such that $F_i$ depends only on time and on the last $n-(i-1)$ components, i.e.\ $F_i(t, x_i,\dots,x_n)$,  for any $i$ in $\llbracket 1, n \rrbracket$;
      \item $A\colon [0,\infty)\to \R^N\otimes \R^N$ \textcolor{black}{is bounded} and the blocks $ [A_t]_{i,j}\in \R^{d_i}\otimes \R^{d_j}$ in \eqref{eq:def_matrix_A} are such that \[[A_t]_{i,j} = \begin{cases}
      \text{is non-singular (i.e.\ it has rank  $d_i$) uniformly in } t, \mbox{ if } j= i-1;\\
      0, \mbox{ if } j< i-1.
      \end{cases}\]
  \end{itemize}
\end{description}
Clearly, $n$ is in $\llbracket 1,N\rrbracket$ and $n=1$ if and
only if $d=N$, i.e.\ if the dynamics is non-degenerate.\newline
In the linear framework ($F=0$) and for constant diffusion coefficients ($\sigma(t,x)=\sigma$), this last assumption can be seen as a H\"ormander-type condition, ensuring the hypoellipticity of the infinitesimal generator associated with the process
$\{X_t\}_{t\ge0}$, which is in this setting equivalent to the Kalman condition, see e.g. \cite{Priola:Zabczyk09}. We highlight however that in our framework, the ``classic'' H\"ormander assumption (cf.\ \cite{Hormander67}) cannot be considered, due to the low
regularity of the coefficients we will consider in \eqref{eq:SDE} (see Theorem \ref{thm:main_result}) \textcolor{black}{and the non-local setting}.

\textcolor{black}{One can wonder why we here restrict to a nonlinear perturbation of a linear drift and do not consider a fully non-linear drift as in the Gaussian setting, see e.g. \cite{Delarue:Menozzi10}, \cite{Chaudru:Menozzi17} and then linearize around this drift to consider a suitable \textit{proxy} process. The point is that the degenerate non-local setting induces additional difficulties which are mainly due to the highly singular behavior of the L\'evy measure associated with the proxy process on the \textit{big} space $\R^N $, see e.g. \cite{Huang:Menozzi16} for related discussions and Proposition \ref{prop:Decomposition_Process_X} and Remark \ref{DA_SCRIVERE_VINCOLO_SUL_MODELLO2} below}.\\
\newline
In Equation \eqref{eq:SDE} above, $\{Z_t\}_{t\ge 0}$ is a $d$-dimensional, symmetric and adapted L\'evy
process with respect to some stochastic basis $(\Omega, \mathcal{F},\{ \mathcal{F}_t\}_{t\ge 0},\mathbb{P})$.
We recall that a $d$-valued L\'evy process is a stochastically continuous process on $\R^d$ starting from zero and such that its
increments are independent and stationary. Moreover, it is well-known (see e.g.\ \cite{book:Sato99}) that any L\'evy process admits a
c\`adl\`ag modification, i.e.\ a right continuous modification having left limits $\mathbb{P}$-almost surely. We will always assume to
have chosen such a version.\newline
A fundamental tool in the analysis of L\'evy processes is given by the L\'evy-Kitchine formula (see for instance \cite{book:Jacob05}) that allows us to represent
the L\'evy symbol $\Phi(\xi)$ of $\{Z_t\}_{t\ge0}$, given by
\[\mathbb{E}[e^{i\xi\cdot Z_t}] \, = \, e^{t\Phi(\xi)}, \quad \xi \in \R^d\]
in terms of the generating triplet $(b,\Sigma,\nu)$ as:
\[\Phi(\xi) \, = \, i b \cdot \xi -\frac{1}{2} \Sigma\xi\cdot \xi+\int_{\R^d_0}\bigl(e^{i \xi\cdot z}-1-i \xi\cdot z \mathds{1}_{B(0,1)}(z)\bigr)
\,\nu(dy), \quad \xi \in \R^d,\]
where $b$ is a vector in $\R^d$, $\Sigma$ is a symmetric, non-negative definite matrix in $\R^d\otimes \R^d$ and $\nu$ is a L\'evy measure on
$\R^d_0:=\R^d\smallsetminus \{0\}$, i.e.\ a $\sigma$-finite measure on $\mathcal{B}(\R^d_0)$, the Borel $\sigma$-algebra on $\R^d_0$, such
that $\int(1\wedge \vert z\vert^2) \, \nu(dz)$ is finite.
In particular, any L\'evy process is completely determined by its generating triplet $(b,\Sigma,\nu)$. \newline
Importantly, we point out already that a change on the truncation set $B(0,1)$ for the L\'evy-Kitchine formula does not affect the formulation of the L\'evy symbol $\Phi$, since we assumed \textcolor{black}{$Z $, and therefore $\nu$}, to be symmetric. Namely, given a threshold $c>0$, the L\'evy symbol $\Phi(\xi)$ of $\{Z_t\}_{t\ge0}$ could be also represented as
\begin{equation}
\label{eq:change_of_truncation}
\Phi(\xi) \, = \, i b \cdot \xi -\frac{1}{2} \Sigma\xi\cdot \xi+\int_{\R^d_0}\bigl(e^{i \xi\cdot z}-1-i \xi\cdot z \mathds{1}_{B(0,c)}(z)\bigr)
\,\nu(dz), \quad \xi \in \R^d,
\end{equation}
where $b$, $\Sigma$ and $\nu$ are as above.
Here, we only consider pure jump processes, i.e.\ $\Sigma=0$. Indeed, the more general case, where a Gaussian component is considered, can be obtained from already existing results (cf. \cite{Chaudru:Menozzi17}).\newline
We will suppose moreover that, additionally to the symmetry, the L\'evy measure $\nu$ of $\{Z_t\}_{t\ge 0}$ satisfies the following
\emph{non-degeneracy condition}:
\begin{description}
  \item[{[ND]}] there exist a Borel function $Q\colon\R^d\to [0,\infty)$ such that
  \begin{itemize}
      \item $\text{ess-sup}\{Q(z)\colon z \in \R^d\}<+\infty$;
      \item there exist $r_0>0$ and $c>0$ such that $Q(z)\ge c$ and \textcolor{black}{is} Lipschitz continuous in $B(0,r_0)$;
      \item there exists $\alpha\in (1,2) $ and a finite, non-degenerate measure $\mu$ on $\mathbb{S}^{d-1}$ such that
      \[
\nu(\mathcal{A}) \,=\,\int_{0}^{\infty}\int_{\mathbb{S}^{d-1}}\mathds{1}_{\mathcal{A}}(rs)Q(rs)\,\mu(ds) \frac{dr}{r^{1+\alpha}},\quad \mathcal{A} \in\mathcal{B}(\R^d_0),
\]
\end{itemize}
\end{description}
where $\mathcal{B}(\R^d_0)$ stands for the Borelian
$\sigma$-field on $\R^d_0$.
We recall moreover that a spherical measure $\mu$ on $\mathbb{S}^{d-1}$ is non-degenerate (in the sense of Kolokoltsov \cite{Kolokoltsov00}) if there exists a constant $\tilde{\eta} \ge 1$
such that
\begin{equation}\label{eq:non_deg_measure}
\tilde{\eta}^{-1}\vert \xi \vert^\alpha \, \le \, \int_{\mathbb{S}^{d-1}}\vert \xi\cdot s \vert^\alpha \, \mu(ds) \, \le\,\tilde{\eta}\vert \xi \vert^\alpha,
\quad \xi \in \R^d.
\end{equation}
Since any $\alpha$-stable L\'evy measure can be decomposed into a spherical part $\mu$ on $\mathbb{S}^{d-1}$ and a radial part $r^{-(1+\alpha)}dr$ (see e.g.\ Theorem $14.3$ in \cite{book:Sato99}), assumption [\textbf{ND}] roughly states that the L\'evy measure of $\{Z_t\}_{t\ge0}$ is absolutely continuous with respect to the \textcolor{black}{non-degenerate} (in the sense of \eqref{eq:non_deg_measure}), L\'evy measure of a $\alpha$-stable process and that their Radon-Nikodym derivative is given by the function $Q$. From this point further, we will denote such a L\'evy measure by $\nu_\alpha(dz):=\mu(ds)r^{-(1+\alpha)}dr$ with $z=rs$.\newline
\textcolor{black}{We insist that the whole paper is concerned with the the so-called \textit{sub-critical} regime, i.e. $\alpha\in (1,2) $. Some hints concerning possible extensions to super-critical cases are given in Remark \ref{RK_SC} below.}

In order to deal with a possibly multiplicative noise, i.e.\ in
the presence of a space-inhomogeneous diffusion coefficient $\sigma$ in Equation \eqref{eq:SDE}, we will need the following:
\begin{description}
\item[{[AC]}] If $x\to \sigma(t,x)$ is non-constant for some
$t\ge0$, then the measure $\nu_\alpha$
is absolutely continuous with respect to the Lebesgue measure on $\R^N$ with Lipschitz Radon-Nykodim derivative.
\end{description}
Assumptions [\textbf{ND}] and [\textbf{AC}] together imply in particular that in the multiplicative case, the L\'evy measure $\nu$ of $\{Z_t\}_{t\ge 0}$ can be decomposed as
\begin{equation}
\label{eq:decomposition_nu_AC}
    \nu(dz) \, = \, Q(z)\frac{g(\frac{z}{|z|})}{|z|^{d+\alpha}}dz,
\end{equation}
for some Lipschitz function $g\colon \mathbb{S}^{d-1}\to \R$.\newline
At this point, we would like to remark that no regularity is assumed for the L\'evy measure $\nu$ of $\{Z_t\}_{t\ge0}$ in the additive framework (or more generally, for a space-independent $\sigma$). In particular, the measure $\mu$ in condition [\textbf{ND}] may not be absolutely continuous with respect to the Lebesgue measure on $\mathbb{S}^{d-1}$. Indeed, our model can also include very singular (with respect to the Lebesgue measure) examples such as the cylindrical $\alpha$-stable process associated with $\mu=\sum_{i=1}^d\frac{1}{2}(\delta_{e_i}+\delta_{-e_i})$. See e.g.\ \cite{Bass:Chen06}  for more details.\newline
From this point further, we always assume that the above hypotheses on the coefficients are satisfied.

We would like to conclude the introduction with some comments concerning our assumptions with particular reference with our previous works.\newline
In \cite{Marino21}, the Schauder estimates, an important analytical first step for proving \textcolor{black}{the} well-posedness of SDEs, has been showed for degenerate Ornstein-Uhlenbeck operators driven by a more general class of L\'evy noises. It also includes, for example, the asymmetric version of the stable-like noises we consider in this work. We start highlighting that a similar family of noises could not have been introduced here, as in \cite{Marino20}, due to the non-linear structure of our problem and, especially, our technique of proof through a perturbative approach. Indeed, it requires more delicate regularizing properties for the involved operators and, in particular, a compatibility between some proxy and the original operator, seen as a perturbation of the first one. \newline
Here, we have followed a backward perturbative approach  as firstly introduced by McKean-Singer in \cite{Mckean:Singer67}. This terminology comes from the fact that the underlying proxy process will be associated with a backward in time flow. This method  appears  more natural  for proving weak uniqueness in a degenerate $L^p-L^q$ framework (cf. \cite{Chaudru:Menozzi17} in the diffusive case). Roughly speaking, it only requires controls on the gradients (in a weak sense) for the solutions of the associated PDE in order to apply the inversion technique on the infinitesimal generator.
However, we are confident that the Schauder estimates presented in \cite{Marino20} could be extended to the class of noises we consider here. Relying on them, we could have then proven the uniqueness in law for dynamics such as \eqref{eq:SDE}. This method appears really involved and long since it structurally requires to establish pointwise estimates for the first order derivatives with respect to the degenerate components of the dynamics.
Another useful advantage of the backward perturbative approach  is that it allows us to show Krylov-type estimates on the solution process $X_t$ of SDE \eqref{eq:SDE}. \textcolor{black}{This} kind of controls seems of independent interest and new for random dynamics involving degenerate stable-like noises.

The drawback of our approach is that it leads to a specific structure in Equation \eqref{eq:SDE}, \textcolor{black}{given by assumption} \textbf{[H]}. Namely, we cannot consider  drift of the form $F_i(x)=F_i(x_{i-1},\cdots,x_n) $ with non-linear dependence w.r.t. $x_{i-1}$, variable which transmits the noise. This case is often investigated for Brownian noises (see e.g. \cite{Delarue:Menozzi10}, \cite{Chaudru:Menozzi17}).
This feature is specifically linked to the structure of the joint law of a  stable process and its iterated integrals which generate a multi-scale stable process with highly singular associated spectral measure, see e.g. Proposition \ref{prop:Decomposition_Process_X} and Remark \ref{DA_SCRIVERE_VINCOLO_SUL_MODELLO} below or \cite{Huang:Menozzi16}.
Similar issues \textcolor{black}{force} us to assume in the multiplicative noise case that the driving process has an absolutely continuous spectral measure with respect to the Lebesgue measure on $\mathbb{S}^{d-1}$. This precisely allows us to get estimates which will be uniform with respect to the parameters for the considered class of proxys.

\paragraph{Main Driving Processes Considered.} Here, we highlight that assumption [\textbf{ND}] applies to a large class of L\'evy processes on $\R^d$. As already
pointed out in \cite{Schilling:Sztonyk:Wang12}, it holds for the following families of stable-like examples with $\textcolor{black}{\alpha \in (1,2)}$:
\begin{enumerate}
  \item Stable process \cite{book:Sato99}:
  \[Q(z) \, = \, 1;\]
  \item Truncated stable process with $r_0>0$ \cite{Kim:Song08}:
  \[Q(z) \, = \, \mathds{1}_{(0,r_0]}(|z|);\]
  \item Layered stable process with $\beta>\alpha$ and $r_0>0$ \cite{Houdre:Kawai07}:
  \[Q(z) \, = \, \mathds{1}_{(0,r_0]}(|z|)+\mathds{1}_{(r_0,\infty)}(|z|)|z|^{\alpha-\beta};\]
  \item Tempered stable process \cite{Rosinski09} with $Q(z)=Q(rs)$ such that for all $s$ in $\mathbb{S}^{d-1}$,
  \[r\to Q(rs) \text{ is completely monotone, $Q(0)>0$ and $\lim_{r\to +\infty}Q(rs)=0$}.\]
  \item Relativistic stable process \cite{Carmona:Masters:Simon90}, \cite{Byczkowski:Malecki:Ryznar09}:
  \[Q(z) \, = \, (1+|z|)^{(d+\alpha-1)/2}e^{-|z|};\]
  \item Lamperti process with $f\colon \mathbb{S}^{d-1}\to \R$ even such that
  $\sup f(s)<1+\alpha$ \cite{Caballero:Pardo:Perez10}:
  \[Q(z) \, = \,
  \exp\bigl(|z|f(\frac{z}{|z|})\bigr)\Bigl(\frac{|z|}{e^{|z|}-1}\Bigr)^{1+\alpha}, \quad z \in \R^d_0.\]
\end{enumerate}

\paragraph{Organization of Paper.}
The article is organized as follows. In Section $2$, we introduce some useful notations and we present the associated martingale
problem. In particular, we state there our main results.
Section $3$ contains all the associated analytical tools that allow to derive our results. Namely, we follow there a perturbative approach, considering a suitable linearization of our dynamics \eqref{eq:SDE} around a Cauchy-Peano flow which takes into account the deterministic part of our model (corresponding to \eqref{eq:SDE} with $\sigma =0$). Section $4$ is then dedicated to prove the well-posedness of the associated martingale problem, exploiting the analytical results given in Section $3$. In Section $5$, we finally construct an ``ad hoc'' Peano counter-example to the uniqueness in law for SDE \eqref{eq:SDE}.

\setcounter{equation}{0}
\section{Basic Notations and Main Results}

We start recalling some useful notations we will need below. In the following, $C$ will denote a generic \emph{positive} constant. It may change from line to line and it will depend only on the parameters appearing in the previously stated assumptions, as for instance: $d,N,\alpha,\eta,b,g,r_0,\mu$. We will explicitly specify any other dependence that may occur.\newline
Given a function $f\colon \R^N\to \R$, we denote by $Df(x)$, and $D^2f(x)$ the first and second Fr\'echet derivative of $f$ at a point $x$ in $\R^N$ respectively,
when they exist. We denote by $B_b(\R^N)$ the family of all the Borel and bounded functions $f\colon \R^N\to \R$. It is a Banach space endowed with the supremum
norm $\Vert \cdot \Vert_\infty$. We also consider its closed subspace $C_b(\R^N)$ consisting of all the continuous functions. Moreover, $C^\infty_c(\R^N)\subseteq
C_b(\R^N)$ denotes the space of smooth functions with compact support.

We now recall two correlated definitions of solution associated with SDE \eqref{eq:SDE}. Let us consider fixed $\mu$ in $\mathcal{P}(\R^N)$, the family of
the probability measures on $\R^N$ and an initial time $t\ge0$.

\begin{definition}
A weak solution of SDE \eqref{eq:SDE} with starting condition $(t,\mu)$ is a $N$-dimensional, c\`adl\`ag, adapted process $\{X_s\}_{s\ge 0}$
on some filtered probability space $(\Omega, \mathcal{F},\{ \mathcal{F}_s\}_{s \ge 0}, \mathbb{P})$ such that
\begin{itemize}
  \item the law of $X_t$ is $\mu$;
  \item there exists a $d$-dimensional, adapted L\'evy process $\{Z_s\}_{s\ge t}$ satisfying [\textbf{ND}] and [\textbf{AC}] such that
    \begin{equation}
    \label{eq:SDE_Integral}
    X_s \, = \, X_t + \int_t^sG(u,X_u)\,du+\int_t^s \sigma(u,X_{u-})B\, dZ_u, \quad s\ge t,\, \, \mathbb{P}\text{-a.s.}
    \end{equation}
    \end{itemize}
\end{definition}
To state our second definition, we need to consider the infinitesimal generator  $\partial_s+L_s$ (formally) associated with the solutions of SDE \eqref{eq:SDE}. Noticing that the term involving the constant drift $b$ can be absorbed in the expression for $G$ without loss of generality, the operator $L_s$ can be represented for any $\phi$ in $C^\infty_c(\R^N)$ as
\begin{multline}
   \label{eq:def_generator}
L_s\phi(s,x) \, := \, \langle G(s,x),D_x\phi(x)\rangle +\mathcal{L}_s\phi(s,x)
\\
:= \,\langle G(s,x),D_x\phi(x)\rangle
+ \int_{\R^d_0}\bigl[\phi(x+B(s,x)z)-\phi(x) \bigr] \,\nu(dz),
\end{multline}
where $\langle \cdot, \cdot \rangle$ denotes the inner product on the bigger space $\R^N$ and, for brevity, $B(s,x):=B\sigma(s,x)$. As done in \cite{Priola15}, we introduce the following definition:
\begin{definition}
A solution of the martingale problem for $\partial_s+L_s$ with initial condition $(t,\mu)$ is an $N$-dimensional, c\`adl\`ag process $\{X_s\}_{s\ge t}$ on some
probability space $(\Omega, \mathcal{F},\mathbb{P})$ such that
\begin{itemize}
    \item the law of $X_t$ is $\mu$;
    \item for any $\phi$ in $\dom\bigl(\partial_s+L_s\bigr)$, the process
  \[\Bigl{\{}\phi(s,X_s)-\phi(t,X_t)- \int_{t}^{s}\bigl(\partial_u+L_u\bigr)\phi(u,X_u) \, du\Bigr{\}}_{s \ge t}\]
  is a $\mathbb{P}$-martingale with respect to the natural filtration $\{ \mathcal{F}^X_s\}_{s\ge 0}$ of the process $\{X_s\}_{s\ge 0}$.
\end{itemize}
\end{definition}

We can now recall some known results that enlighten the link between the two definitions presented above.
For a more thorough analysis on the topic of martingale problems in a rather abstract and general framework, we refer to Chapter $4$ in \cite{book:Ethier:Kurtz86}.\newline
Given a solution $\{X_s\}_{s\ge0}$ of SDE \eqref{eq:SDE},
an application of the It\^o formula immediately shows that the process $\{X_s\}_{s\ge0}$ is a solution of the martingale problem for
$\partial_s+L_s$ with initial condition $(t,\mu)$, too. \newline
On the other hand, if there exists a solution $\{X_s\}_{s\ge0}$ of the martingale problem for $\partial_t+L_t$ with initial condition $(t,\mu)$, it is possible to
construct an ``enhanced'' filtered probability space $(\tilde{\Omega},\tilde{ \mathcal{F}}, \{\tilde{ \mathcal{F}}\}_{s\ge0},\tilde{\mathbb{P}})$ on which there exists
a solution $\{\tilde{X}_s\}_{s\ge 0}$ of the SDE \eqref{eq:SDE}. Moreover, the two processes $\{X_s\}_{s\ge t}$ and $\{\tilde{X}_s\}_{s\ge t}$ have the same law (See, for more details, \cite{Kurtz11}). Thus, it holds that:

\begin{prop}
\label{thm:equivalence_existence}
Let $\mu$ be in $\mathcal{P}(\R^N)$ and $t\ge0$. The existence of a weak solution for SDE \eqref{eq:SDE} with initial condition $(t,\mu)$ is equivalent to the existence of a
solution to the martingale problem for $\partial_s+L_s$ with initial condition $(t,\mu)$.
\end{prop}

We can now move on the notion of uniqueness associated with our problem.

\begin{definition}
We say that weak uniqueness holds for the SDE \eqref{eq:SDE} with initial condition $(t,\mu)$ if any two solutions $\{X_s\}_{s\ge 0}$, $\{Y_s\}_{s\ge 0}$ of SDE
\eqref{eq:SDE} with initial condition $(t,\mu)$ have same finite dimensional distributions. In particular, we say that  SDE \eqref{eq:SDE} is weakly well-posed if for any
$\mu$ in $\mathcal{P}(\R^N)$ and any $t\ge 0$, there exists a unique weak solution of SDE \eqref{eq:SDE} with initial condition $(t,\mu)$.
\end{definition}

Since the definition above takes into account only the law of the solutions $\{X_s\}_{s\ge t}$, $\{Y_s\}_{s\ge t}$, they may, in general, have been defined on
different stochastic bases or with respect to two different underlying L\'evy processes. The definition of uniqueness for a solution of the martingale problem for
$\partial_s+L_s$ can be stated similarly.

It is not difficult to check that the uniqueness of the martingale problem for $\partial_s+L_s$ with initial condition $(t,\mu)$ implies the weak uniqueness of the SDE
\eqref{eq:SDE}. Furthermore, it has been shown in \cite{Kurtz11}, Corollary $2.5$ that the converse is also true.

\begin{prop}
\label{thm:equivalence_uniqueness}
Let $\mu$ be in $\mathcal{P}(\R^N)$ and $t\ge 0$. Then, weak uniqueness holds for SDE \eqref{eq:SDE} with initial condition $(t,\mu)$ if and only if uniqueness holds for the
martingale problem for $\partial_s+L_s$ with initial condition $(t,\mu)$.
\end{prop}

Thanks to Propositions \ref{thm:equivalence_existence} and \ref{thm:equivalence_uniqueness}, we can conclude that the two approaches, i.e.\ the martingale formulation
and the dynamics given in \eqref{eq:SDE}, are equivalent in specifying a L\'evy diffusion process on $\R^N$. We recall however that a third, yet equivalent, method
is given by the forward Fokker-Plank equation governing the law of the process. We will not  explicitly define it since we will not exploit it afterwards (see,
for more details, e.g.\ \cite{Figalli08}, \cite{LeBris:Lions08}). \newline
From now on, we write $x$ in $\R^N$ as $x=(x_1,\dots,x_n)$ where $x_i=(x^1_i,\dots,x^{d_i}_i)$ is in $\R^{d_i}$ for any $i$ in $\llbracket 1,n\rrbracket$. We can now state our main theorem.

\begin{theorem}
\label{thm:main_result}
\textcolor{black}{Let $\{\beta^j\colon j \in \llbracket 1,n\rrbracket \}$ be a family of indexes in $(0,1]$} such that
\begin{itemize}
    \item $x_j \to \sigma(t,x_1,\dots,x_j,\dots,x_n)$ is $\beta^1$-H\"older continuous, uniformly in $t$ and in $x_i$ for $i\neq j$;
    \item $x_j \to F_1(t,x_1,\dots,x_j,\dots,x_n)$ is $\beta^1$-H\"older continuous, uniformly in $t$ and in $x_i$ for $i\neq j$;

    \item $x_j \to F_i(t,x_i,\dots,x_j,\dots,x_n)$ is $\beta^j$-H\"older continuous, uniformly in $t$ and in $x_k$, for $k\neq j$ and $2\le i\le j$.
\end{itemize}
Additionally, we suppose that there exists $K\ge 1$ such that $|F_i(t,0)|\le K$ for any $i$ in $\llbracket 1,n\rrbracket$ and any $t\ge 0$. Then, the SDE \eqref{eq:SDE} is weakly well-posed if
\begin{equation}
\label{eq:thresholds_beta}
\beta^j\, > \, \frac{1+\alpha(j-2)}{1+\alpha(j-1)},\ j\ge 2.
\end{equation}
\end{theorem}
Theorem \ref{thm:main_result} above will follow from Propositions \ref{thm:equivalence_existence} and \ref{thm:equivalence_uniqueness}, once we have shown
that under the same assumptions, there exists a unique weak solution to the martingale problem for $(\partial_s+L_s,\delta_x)$ at any $x$ in $\R^N$.\newline
\textcolor{black}{
The fact that all the coordinates of the drift present the same regularity with respect to $x_i$ may appear a bit strange at first sight. Indeed, one could guess that the needed regularity on the drift depends on the direction since the effective intensity of the noise is not the same on two different components. Such an assumption is actually given by the specific current framework, which involves as a proxy a stochastic integral with respect to a stable-like jump process and its associated iterated integrals that leads to additional constraints on the regularity indexes needed for our method to work. For a more thorough (and technical) explanation of this particular assumption, we suggest the interested reader to see Remark \ref{DA_SCRIVERE_VINCOLO_SUL_MODELLO2}.}

\begin{remark}[About possible extensions to some \textit{super-critical} regimes]\label{RK_SC}
\textcolor{black}{As already emphasized, we focused here on the \textit{sub-critical} case $\alpha\in (1,2) $. Anyhow, in our current specific model, which can be viewed as a degenerate Ornstein-Uhlenbeck process perturbed by a suitable non-linear drift, we truly feel that some super-critical cases could be addressed as well. Indeed, there are actually few points below where we use that $\alpha>1 $. We can e.g. mention Section \ref{SEZIONE_SPIEGAZIONE_CONGELAMENTO_E_ALTRI} for the existence of a solution to the martingale problem and the proof of Lemmas \ref{convergence_dirac} and \ref{lemma:Control_p_M}. This is furthermore mainly to ease some technical points and we are pretty confident that the arguments could be adapted to handle super-critical indexes $\alpha $.}

\textcolor{black}{There is anyhow \textit{a priori} a lower threshold for the super-critical indexes that could be handled. Namely, from the analysis below, we require that $\beta^j<\alpha, \ j\in \llbracket 1,n \rrbracket $. Since $\beta^j\le 1 $, this condition is automatically fulfilled in the sub-critical regime $\alpha\in (1,2) $. Anyhow, we think that super-critical cases could be considered under the additional condition that $\max_{j\in \llbracket 1,n \rrbracket }\beta^j<\alpha $. Since the minimal thresholds go to $1$ with $n$ in \eqref{eq:thresholds_beta}, this means that the larger the chain the higher the threshold for the super-critical regime}.
\end{remark}

As a by-product of our method of proof, we have been able to show a Krylov-type estimates for the solutions of SDE \eqref{eq:SDE}. For notational convenience, we will say that two real numbers $p>1$, $q>1$ satisfy Condition $(\mathscr{C})$ when the following inequality holds:
\[ \tag{$\mathscr{C}$}
\bigl(\frac{1-\alpha}{\alpha}N+\sum_{i=1}^n
id_i\bigr)\frac{1}{q}+\frac{1}{p} \, < \, 1.
\]

Roughly speaking, such a threshold guarantees the necessary integrability in time with respect to the associated intrinsic scale of the system when considering the $L^p_t-L^q_x$ theory (see Equation \eqref{eq:control_det_T2} for more details). Furthermore, when considering the homogeneous case, i.e.\ when all the components of the system has the same dimension ($d_i=d$ and $N=nd$), condition $(\mathscr{C})$ can be rewritten in the following, clearer, way:
\[\left(\frac{2+\alpha(n-1)}{\alpha}\right)\frac{nd}{q}+\frac{2}{p} \, < \, 2.\]
In particular, taking $\alpha=2$ above, we find the same threshold appearing in \cite{Chaudru:Menozzi17} for the diffusive setting. We highlight moreover that our thresholds can be seen as a natural extension of the ones appearing in \cite{Krylov:Rockner05} in the non-degenerate, Brownian setting.

\begin{corollary}\label{coroll:Krylov_Estimates}
Under the same assumptions of Theorem \ref{thm:main_result}, let $T>0$ and $p>1$, $q>1$ such that Condition $(\mathscr{C})$ holds. Then, there exists a constant $C:=C(T,p,q)$ such that for any  $f$ in $L^p\bigl(0,T;L^q(\R^N)\bigr)$, it holds
\begin{equation}
\label{eq:Krylov_Estimates}
\Bigl{|}\mathbb{E}^{\mathbb P_{t,x}}\bigl[\int_t^Tf(s,X_s) \, ds\bigr] \Bigr{|} \, \le \, C\Vert f \Vert_{L^p_tL^q_x}, \quad (t,x) \in [0,T]\times \R^N,
\end{equation}
where $\{X_s\}_{s\ge0}$ is the canonical process associated with $\mathbb P_{t,x}[\cdot]:=\mathbb E[\cdot|X_t=x] $ which is also the unique weak solution of SDE \eqref{eq:SDE} with initial condition $(t,x)$. In particular, the random variable $X_s$ admits a density $p(t,s,x,\cdot)$ for any $t<s$ and any $x$ in $\R^N$.
\end{corollary}

Additionally, we have been able to show the following non uniqueness result.

\begin{theorem}
\label{thm:counterexample}
\textcolor{black}{Fixed $i$, $j$ in $\llbracket 2,n\rrbracket$ with $j\ge i$ and  $\beta^j_i$ in $(0,1]$ such that
\[
\beta^j_i \, < \, \frac{1+\alpha(i-2)}{1+\alpha(j-1)},
\]
there exists a drift $F$ such that $x_j \to F_i(t,x_i,\dots,x_j,\dots,x_n)$ is $\beta_i^j$-H\"older continuous, uniformly in $t$ and in $x_k$, for $k\neq j$,
for which  weak uniqueness fails for the SDE \eqref{eq:SDE}.}
\end{theorem}

The above result will be proven in Section $5$, showing a suitable,
explicit Peano-type counter-example.

\begin{remark}
As opposed to the Gaussian driven case, we did not succeed to obtain regularity indexes which are \textit{sharp} at any level of the chain (cf.\ \cite{Chaudru:Menozzi17}). However, we point out that for diagonal systems of the form:
\begin{equation}
\label{eq:DIAG}
  \begin{cases}
    dX^1_t \, = \, F_1(t,X^1_t,\dots,X^n_t)dt + \sigma(t,X^1_{t-},\dots,X^n_{t-})dZ_t,\\
    dX^2_t \, = \, \left[A^2_tX^1_t +F_2(t,X^2_t)\right] dt,\\
    dX^3_t \, = \, \left[A^3_tX^2_t+F_3(t,X^3_t)\right]dt,\\
    \vdots\\
    dX^n_t \, = \, \left[A^n_t X^{n-1}_t+F_n(t,X^n_t)\right]dt,
  \end{cases}
\end{equation}
i.e. the degenerate components are perturbed by a function which only depends of the current level on the chain, we  have that the previous thresholds are \textit{almost} sharp. Indeed, in this case, we are led to consider $\beta^j >\frac{1+\alpha(j-2)}{1+\alpha(j-1)}$ which gives the well-posedness from the conditions in Theorem \ref{thm:main_result}  while Theorem \ref{thm:counterexample} shows that uniqueness fails as soon as $ \beta_j^j<\frac{1+\alpha(j-2)}{1+\alpha(j-1)} $.
For this diagonal system, Theorems \ref{thm:main_result} and
\ref{thm:counterexample} together then provide an ``almost'' complete
understanding of the weak well-posedness for degenerate SDEs of type
\eqref{eq:DIAG} with H\"older coefficients. Indeed,
the problem for the critical exponents
\[\overline{\beta}^j_j \, = \, \frac{1+\alpha(j-2)}{1+\alpha(j-1)}, \quad
 j \in \llbracket 1,n\rrbracket, \]
remains to be investigated and, up to our best knowledge, there are no general
available results even in the diffusive case. We can only mention \cite{Zhang18} in the kinetic case.
\end{remark}

\setcounter{equation}{0}
We present in this section the analytical tools we will need to show the well-posedness of the associated martingale problem. In particular, they will be fundamental in the derivation of our main Theorem \ref{thm:main_result}, thanks to Propositions \ref{thm:equivalence_existence} and \ref{thm:equivalence_uniqueness}. For this reason, we will assume in this section to be under the same conditions of Theorem \ref{thm:main_result}. Moreover, we will suppose that the final time horizon $T$ is small enough for our scopes. Indeed, we could always exploit the Markov property of the involved processes and standard chaining in time arguments to extend the results to arbitrary (but finite) time intervals.

\subsection{The ``Frozen'' Dynamics}

The crucial element in our approach consists in choosing wisely a suitable proxy operator with well-known properties and controls, along which we can expand the infinitesimal generator $L_s$, with an additional negligible error.\newline
In order to deal with potentially unbounded perturbations $F$, it is natural to use a proxy involving a non-zero first order term  associated with a flow associated with $G(t,x):=Ax+F(t,x)$, the transport part of SDE \eqref{eq:SDE} (see e.g. \cite{Krylov:Priola10} or \cite{Chaudru:Menozzi:Priola19}).\newline
Remembering that we assume $F$ to be H\"older continuous, we know from the classical Peano-Lipschitz Theorem that there exists a solution of
\begin{equation}
\label{eq:def_Cauchy_Peano_flow}
    \begin{cases}
  d \theta_{t,\tau}(\xi) \, = \, \bigl[A_t \theta_{t,\tau}( \xi)+F(t, \theta_{t,\tau}( \xi))\bigr]\, dt \quad \mbox{ on } [0,\tau];\\
   \theta_{\tau,\tau}( \xi) \, = \,  \xi,
\end{cases}
\end{equation}
even if it may be not unique. For this reason, we are going to choose one particular flow, denoted by $\theta_{t,\tau}( \xi)$, and consider it fixed throughout the work. As it will be shown below in Lemma \ref{lemma:measurability_flow}, it is always possible to take a measurable version of such a flow.\newline
More precisely, given a freezing couple $(\tau, \xi)$ in $(0,T]\times\R^N$, the backward flow will be defined on $[0,\tau]$ as
\[\theta_{t,\tau}( \xi) \, = \,  \xi - \int_{t}^{\tau}\bigl[ A_u \theta_{u,\tau}( \xi)+F(u, \theta_{u,\tau}( \xi))\bigr] \, du.\]
Fixed the reference flow, the next step is to consider the stochastic dynamics linearized along the backward flow $\theta_{t,\tau}(\xi)$. Namely, for any fixed starting point $(t,x)$ in $[0,\tau]\times \R^N$, we consider  $\{\tilde{X}^{\tau,\xi,t,x}_s\}_{s\in [t,T]}$ solving the following SDE:
\begin{equation}
\label{eq:SDE_frozen_SDE}
\begin{cases}
d\tilde{X}^{\tau,\xi,t,x}_u\, = \, \bigl[A_u\tilde{X}^{\tau,\xi,t,x}_u+\tilde{F}^{\tau, \xi}_u\bigr]\, du +B\tilde{\sigma}^{\tau, \xi}_u\,  dZ_u, \quad u\in  [t,T],\\
 \tilde{X}^{\tau,\xi,t,x}_t\, = \, x,
\end{cases}
\end{equation}
where $\tilde{\sigma}^{\tau,\xi}_s:=\sigma(s,\theta_{s,\tau}(\xi))$ and $\tilde{F}^{\tau, \xi}_s:=F(s,\theta_{s,\tau}(\xi))$.\newline
In order to obtain an integral representation of the process $\{\tilde{X}^{\tau,\xi,t,x}_s\}_{s\in [t,T]}$, we now introduce the time-ordered resolvent $\mathcal{R}_{s,t}$ of the matrix $A_s$ starting at time $t$. Namely, $\mathcal{R}_{s,t}$ is a time-dependent matrix in $\R^N\otimes \R^N$ that is solution of the following ODE:
\[\begin{cases}
  \partial_s \mathcal{R}_{s,t} \, = \, A_s \mathcal{R}_{s,t} ,\quad s \in [0,T]; \\
  \mathcal{R}_{t,t} \, = \, \Id_{N\times N}.
\end{cases}\]
By the variation of constants method, it is now easy to check that the solution $\tilde{X}^{\tau,\xi,t,x}_s$ of SDE \eqref{eq:SDE_frozen_SDE} satisfies that
\begin{equation}\label{eq:integral_represent_frozen_SDE}
\tilde{X}^{\tau,\xi,t,x}_s \, = \, \tilde{m}^{\tau, \xi}_{s,t}(x) +\int_t^s \mathcal{R}_{s,u}B\tilde{\sigma}^{\tau,\xi}_u dZ_u,
\end{equation}
where the ``frozen shift'' $\tilde{m}^{\tau, \xi}_{s,t}(x)$ is given by:
\begin{equation}\label{eq:def_tilde_m}
\tilde{{m}}^{\tau, \xi}_{s,t}( x) \, = \, \mathcal{R}_{s,t} x + \int_{t}^{s}\mathcal{R}_{s,u} \tilde{F}^{\tau,\xi}_u \, du.
\end{equation}

We point out already two important properties of the shift $\tilde{m}^{\tau,\xi}_{s,t}(x)$.

\begin{lemma}
\label{lemma:identification_theta_m}
Let $s$ in $[0,T]$ and $x,y$ two points in $\R^N$. Then, for any $t<s$, it holds that
\begin{align}
\label{eq:identification_theta_m}
\tilde{m}^{t,x}_{s,t}(x) \, &= \, \theta_{s,t}(x)\\
y-\tilde{m}^{s,y}_{s,t}(x) \, &= \, \theta_{t,s}(y)-x\label{eq:identification_theta_m1}
\end{align}
\end{lemma}
\begin{proof}
We start noticing that by construction in \eqref{eq:def_tilde_m}, $\tilde{m}^{\tau,\xi}_{s,t}(x)$ satisfies
\begin{equation}
\label{eq:dif_differential_m_tilde}
\partial_s\tilde{m}^{\tau,\xi}_{s,t}(x) \, = \,  A_s\tilde{m}^{\tau,\xi}_{s,t}(x)+F(s, \theta_{s,\tau}(\xi)),
\end{equation}
for any freezing parameters $(\tau,\xi)$.
Choosing $\tau=t$, $\xi=x$ above, it then holds that
\[\partial_s\left[\tilde{m}^{s,x}_{s,t}(x) - \theta_{s,t}(x)\right] \, = \, A_s\left[
\tilde{m}^{t,x}_{s,t}(x)-\theta_{s,t}(x) \right].\]
Since, $\tilde{m}^{t,x}_{t,t}(x)=\theta_{t,t}(x)=x$, Equation \eqref{eq:identification_theta_m} then follows immediately applying the Gr\"onwall lemma.\newline
The second identity in \eqref{eq:identification_theta_m1} follows in a similar manner.
\end{proof}

We are now interested in investigating the analytical properties of the ``frozen'' solution process $\tilde{X}^{\tau,\xi,t,x}_s$. In particular, we will show in the next results the existence of a density for such a process and its anisotropic regularizing effect, at least for small times. Further on, we will consider fixed a time-dependent matrix $\mathbb{M}_t$ on $\R^N\otimes \R^N$ given by
\begin{equation}
\label{eq:def_matrix_M}
    \mathbb{M}_t := \text{diag}(I_{d_1\times d_1},tI_{d_2\times d_2},\dots,t^{n-1}I_{d_{n}\times d_{n}}), \quad t\ge0,
\end{equation}
which reflects the multi-scale nature of the underlying dynamics in \eqref{eq:SDE_frozen_SDE}.

\begin{prop}[Decomposition]
\label{prop:Decomposition_Process_X}
Let the freezing couple $(\tau,\xi)$ be in $[0,T]\times \R^N$, $t<s$ in $[0,T]$ and $x$ in $\R^N$. Then, there exists a L\'evy process $\{\tilde{S}^{\tau,\xi,t,s}_{u}\}_{u\ge 0}$ such that
\begin{equation}\label{eq:decomposition_measure_1}
\tilde{X}^{\tau,\xi,t,x}_s \, \overset{(\text{law})}{=} \, \tilde{m}^{\tau, \xi}_{s,t}(x) + \mathbb{M}_{s-t} \tilde{S}^{\tau,\xi,t,s}_{s-t}.
\end{equation}
In particular, the random variable $\tilde{X}^{\tau,\xi,t,x}_s$ admits a continuous density $\tilde{p}^{\tau,\xi}(t,s,x,\cdot)$ given by
\begin{align}
\label{eq:representation_density}
&\tilde{p}^{\tau,\xi}(t,s,x,y) \, = \,  \frac{1}{\det \mathbb{M}_{s-t}}p_{\tilde{S}^{\tau,\xi,t,s}}\left(t-s, \mathbb{M}^{-1}_{s-t}(y-\tilde{m}^{\tau,\xi}_{s,t}(x))\right) \\
&:= \, \frac{\det \mathbb{M}^{-1}_{s-t}}{(2\pi)^N}\int_{\R^N}e^{-i\langle \mathbb{M}^{-1}_{s-t}(y-\tilde{m}^{\tau,\xi}_{s,t}(x)),z\rangle}\exp\left((s-t)\int_{\R^N}\left[\cos(\langle z, p \rangle )-1\right]\nu_{\tilde{S}^{\tau,\xi,t,s}}(dp)\right)\, dz, \notag
\end{align}
where $\nu_{\tilde{S}^{\tau,\xi,t,s}}$ and $p_{\tilde{S}^{\tau,\xi,t,s}}(u,\cdot)$ are  the L\'evy measure and the density associated with the process $\{\tilde{S}^{\tau,\xi,t,s}_{u}\}_{u\ge0}$, respectively.
\end{prop}
\begin{proof}
For simplicity, we start denoting
\[
\tilde{\Lambda}^{\tau,\xi,t,s} \,:= \, \int_t^s \mathcal{R}_{s,u}B\tilde{\sigma}^{\tau,\xi}_u dZ_u,\quad s\ge t,
\]
so that we have from Equation \eqref{eq:integral_represent_frozen_SDE} that $\tilde{X}^{\tau,\xi,t,x}_s \, = \, \tilde{m}^{\tau, \xi}_{s,t}(x) +\tilde{\Lambda}^{\tau,\xi,t,s}$. To conclude, we need to construct a L\'evy process $\{\tilde{S}^{\tau,\xi,t,s}_u\}_{u\ge 0}$ on $\R^N$  such that
\begin{equation}\label{eq:identity_in_law}
\tilde{\Lambda}^{\tau,\xi,t,s} \, \overset{(\text{law})}{=} \, \mathbb{M}_{s-t}\tilde{S}^{\tau,\xi,t,s}_{s-t}.
\end{equation}
To show the identity in law, we are going to reason in terms of the characteristic functions. We start recalling that
the L\'evy process $\{Z_t\}_{t\ge 0}$ on $\R^d$ is characterized by the L\'evy symbol
\[\Phi(p) \, = \,  \int_{\R^d_0}\left[\cos(p\cdot q)-1\right]Q(q) \, \nu_\alpha(dq),
\quad p \in \R^d,\]
where $\nu_\alpha(dq)=\mu(d\theta)\frac{dr}{r^{1+\alpha}}$ is the L\'evy measure of an $\alpha$-stable process.
It is well-known (see e.g.\ Lemma $2.2$ in \cite{Schilling:Wang12}) that at any fixed $t\le s$ in $[0,1]$, $\tilde{\Lambda}^{\tau,\xi,t,s}$ is an infinitely divisible random variable with associated L\'evy symbol
\[\Phi_{\tilde{\Lambda}^{\tau,\xi,t,s}}(z) \, := \, \int_{t}^{s}\Phi\bigl((\mathcal{R}_{s,u}B\tilde{\sigma}^{\tau,\xi}_u)^*z\bigr)\,du, \quad z \in \R^N,\]
where, we recall, we have denoted $\tilde{\sigma}^{\tau,\xi}_u=\sigma(u,\theta_{u,\tau}(\xi))$.\newline
Setting $v:=(u-t)/(s-t)$ and noticing that $u=u(v):=t+v(s-t)$, we can now
rewrite the L\'evy symbol of $\tilde{\Lambda}^{\tau,\xi,t,s}$ as
\begin{equation}
\label{proof:ref_notations}
  \Phi_{\tilde{\Lambda}^{\tau,\xi,t,s}}(z) \, := \, (s-t)\int_{0}^{1}\Phi\bigl((\mathcal{R}_{s,u(v)}B\tilde{\sigma}^{\tau,\xi}_{u(v)})^*z\bigr)\,dv.
\end{equation}
From the analysis performed in \cite{Huang:Menozzi16}, Lemmas $5.1$ and $5.2$ (see also \cite{Delarue:Menozzi10} Proposition $3.7$), we then know that we can decompose the first column of the resolvent $\mathcal{R}_{s,u(v)}$ in the following way:
\[\mathcal{R}_{s,u(v)}B \, = \, \mathbb{M}_{s-t}\widehat{\mathcal{R}}_vB,\]
where $\{\widehat{\mathcal{R}}_v\colon v \in [0,T]\}$ are non-degenerate and bounded matrixes in $\R^N\otimes \R^N$ and the multi-scale matrix $\mathbb{M}_t$ is given in \eqref{eq:def_matrix_M}. We can now rewrite the L\'evy symbol of $\tilde{\Lambda}^{\tau,\xi,t,s}$ as
\[\Phi_{\tilde{\Lambda}^{\tau,\xi,t,s}}(z) \, = \, (s-t)\int_{0}^{1}\Phi\bigl((\widehat{\mathcal{R}}_vB\tilde{\sigma}^{\tau,\xi}_{u(v)})^*\mathbb{M}_{s-t}z\bigr)\,dv, \quad z \in \R^N.\]
The above equality suggests us to define, for any fixed $t\le s$ in $(0,1]$, the (unique in law) L\'evy process  $\{\tilde{S}^{\tau,\xi,t,s}_u\}_{u\ge0}$ associated with the L\'evy symbol
\begin{equation}
\label{proof:eq:def_Levy_symbol_S}
\begin{split}
\Phi_{\tilde{S}^{\tau,\xi,t,s}}(z) \, &:= \, \int_{0}^{1}\Phi\bigl((\widehat{\mathcal{R}}_{v}B\tilde{\sigma}^{\tau,\xi}_{u(v)})^*z\bigr)\,dv \\
&= \, \int_{0}^{1}\int_{\R^d}\left[\cos\left(\langle z,\widehat{\mathcal{R}}_{v}B\tilde{\sigma}^{\tau,\xi}_{u(v)}p\rangle\right)-1\right]\,\nu(dp)dv.
\end{split}
\end{equation}
Since we have that
\begin{equation}
\label{proof:eq:inversion}
\mathbb{E}\bigl[e^{i\langle z,\tilde{\Lambda}^{\tau,\xi,t,s}\rangle}\bigr] \, = \, e^{\Phi_{\tilde{\Lambda}^{\tau,\xi,t,s}}(z)}\, = \, e^{(s-t)\Phi_{\tilde{S}^{\tau,\xi,t,s}}(\mathbb{M}_tz)} \,
= \,\mathbb{E}\bigl[e^{i\langle z,\mathbb{M}_t\tilde{S}^{\tau,\xi,t,s}_{s-t}\rangle}\bigr],
\end{equation}
it follows immediately that Equation \eqref{eq:identity_in_law} holds.\newline
To show the existence of a density for $\tilde{X}^{\tau,\xi,t,x}_s$, we want to exploit the Fourier inversion formula in \eqref{proof:eq:inversion}. To do it, we firstly need to prove that $\text{exp}(\Phi_{\tilde{\Lambda}^{\tau,\xi,t,s}}(z))$ is integrable. From \eqref{proof:eq:def_Levy_symbol_S}, we notice that
\[
\begin{split}
\Phi_{\tilde{S}^{\tau,\xi,t,s}}(z) \,
&= \, \int_{0}^{1}\int_{\R^d}\left[\cos\left(\langle z,\widehat{\mathcal{R}}_{v}B\tilde{\sigma}^{\tau,\xi}_{u(v)}p\rangle\right)-1\right]\,\nu(dp)dv \\
&= \,\int_{0}^{1}\int_{\R^d}\left[\cos\left(\langle z,\widehat{\mathcal{R}}_{v}B\tilde{\sigma}^{\tau,\xi}_{u(v)}p\rangle\right)-1\right]Q(p)\,\nu_\alpha(dp)dv,
\end{split}
\]
where in the last step we used hypothesis [\textbf{ND}]. Exploiting now that the quantities above are non-positive and $Q(p)\ge c>0$ for $p$ in $B(0,r_0)$, we write that
\[
\begin{split}
&\Phi_{\tilde{S}^{\tau,\xi,t,s}}(z)\, \le C \,\int_{0}^{1}\int_{B(0,r_0)}\left[\cos\left(\langle z,\widehat{\mathcal{R}}_{v}B\tilde{\sigma}^{\tau,\xi}_{u(v)}p\rangle\right)-1\right]\,\nu_\alpha(dp)dv \\
&\qquad= \, C\Bigl{\{}-\int_0^1\left|(\widehat{\mathcal{R}}_{v}B\tilde{\sigma}^{\tau,\xi}_{u(v)})^\ast z\right|^\alpha dv+\int_0^1\int_{B^c(0,r_0)}\left[1-\cos\left(\langle z,\widehat{\mathcal{R}}_{v}B\tilde{\sigma}^{\tau,\xi}_{u(v)}p\rangle\right)\right]\nu_\alpha(dp)dv\Bigr{\}} \\
&\qquad\le \, C\Bigl{\{}-\int_0^1\left|(\widehat{\mathcal{R}}_{v}B\tilde{\sigma}^{\tau,\xi}_{u(v)})^\ast z\right|^\alpha\, dv+1\Bigr{\}}.
\end{split}
\]
To conclude, we recall that Lemma $5.4$ in \cite{Huang:Menozzi16}  states that
\[\int_0^1\left|(\widehat{\mathcal{R}}_{v}B\tilde{\sigma}^{\tau,\xi}_{u(v)})^\ast z\right|^\alpha\, dv \, \ge \, C|z|^\alpha,\]
for some positive constant $C$ independent from $t,s,\tau,\xi$. It then follows in particular that
\begin{equation}
\label{eq:control_Levy_symbol_S}
\Phi_{\tilde{S}^{\tau,\xi,t,s}}(z) \, \le \,C\left[1-|z|^\alpha\right], \quad z \in \R^N.
\end{equation}
Since $\text{exp}(\Phi_{\tilde{\Lambda}^{\tau,\xi,t,s}}(z))$ is integrable, it implies that there exists a density $p_{\tilde{\Lambda}^{\tau,\xi,t,s}}(t,s,\cdot)$ of the random variable $\tilde{\Lambda}^{\tau,\xi,t,s}$.  We can now apply the Fourier inversion formula in Equation \eqref{proof:eq:inversion} showing that $p_{\tilde{\Lambda}^{\tau,\xi,t,s}}(t,s,\cdot)$ is given by
\begin{equation}
 \label{eq:definition_density_q}
p_{\tilde{\Lambda}}^{\tau,\xi}(t,s,y) \, := \, \frac{1}{(2\pi)^N}\int_{\R^N}e^{-i\langle y,z\rangle}\text{exp}\left((s-t)\Phi_{\tilde{S}^{\tau,\xi,t,s}}(z)\right)\, dz.
\end{equation}
From Decomposition \eqref{eq:decomposition_measure_1} and Equation \eqref{eq:definition_density_q}, the representation for $\tilde{p}^{\tau,\xi}(t,s,x,\cdot)$ follows immediately.
\end{proof}

Once we have proven the existence of a density $\tilde{p}^{\tau,\xi}(t,s,x,\cdot)$ for the ``frozen'' stochastic dynamics $\tilde{X}^{\tau,\xi,t,x}_s$, we move now on determining its associated smoothing effects. In particular, we show in the following proposition that the derivatives of the ``frozen'' density are controlled by another density at the price of an additional time singularity of the order corresponding to the intrinsic time scale of the considered component in the stable regime. Importantly, such a control holds uniformly in the freezing parameters $(\tau, \xi)$. \newline
Let us introduce for simplicity the following time-dependent scale matrix:
\begin{equation}
\label{eq:def_matrix_T}
    \mathbb{T}_{t} \, := \, t^{\frac{1}{\alpha}}\mathbb{M}_t,\quad t\ge 0,
\end{equation}
\textcolor{black}{where, we recall, the intrinsic time scale matrix $\mathbb{M}_t$ was given in \eqref{eq:def_matrix_M}.}

\begin{prop}
\label{prop:Smoothing_effect}
There exists a family $\{\overline{p}(u,\cdot)\colon u\ge 0\}$ of densities on $\R^N$ and a positive constant $C:=C(N,\alpha)$ such that
\begin{itemize}
  \item for any $u\ge0$ and any $z$ in $\R^N$,  $\overline{p}(u,z)=u^{-N/\alpha}\overline{p}(1,u^{-1/\alpha}z)$; (stable scaling property)
  \item for any $\gamma$ in $[0,\alpha)$,
  \begin{equation}\label{equation:integration_prop_of_q}
    \int_{\R^N} \overline{p}(u,z) \, \vert z \vert^\gamma \, dz \, \le \, Cu^{\gamma/\alpha}, \quad u> 0;
  \end{equation}
  \item  for any $k$ in $\llbracket 0,2 \rrbracket$, any $i$ in $\llbracket 1,n \rrbracket$, any $t< s$ in $[0,T]$ and any $x,y$ in $\R^N$,
  \begin{equation}
  \label{eq:derivative_prop_of_q}
      \vert D^k_{x_i}\tilde{p}^{\tau,\xi}(t,s,x,y) \vert \le C\frac{(s-t)^{-k\frac{1+\alpha(i-1)}{\alpha}}}{\det \mathbb{T}_{s-t}}\overline{p}\left(1,\mathbb{T}^{-1}_{s-t}(y-\tilde{m}^{\tau, \xi}_{s,t}(x))\right).
  \end{equation}
  where we denoted, coherently with the notations introduced before Theorem \ref{thm:main_result}, $D_{x_i}=\left(D_{x^1_i},\dots, D_{x^{d_i}_i}\right)$.
\end{itemize}
\end{prop}
\begin{remark}[About the freezing parameters]
\label{DA_SCRIVERE_VINCOLO_SUL_MODELLO}
We carefully point out that since we will later on choose as parameters $(\tau,\xi)=(s,y)$, it is particularly important that we manage to obtain an upper bound by a density which is independent from those parameters. \textcolor{black}{Indeed, the freezing parameter $y$ will also be an integration variables} (see Section \ref{SEZIONE_SPIEGAZIONE_CONGELAMENTO_E_ALTRI} below \textcolor{black}{and the definition of the \textit{pseudo} Green kernel in \eqref{eq:def_Green_Kernel}}). This is precisely why we actually impose the specific semi-linear drift structure in SDE \eqref{eq:SDE} (cf.\ assumption \textbf{[H]}), as opposed to the more general one that can be handled  in the Gaussian case \cite{Chaudru:Menozzi17}. This is a framework which naturally gives the independence of the large jumps of the \textit{proxy} process $\tilde{X}^{\tau,\xi,t,x}_s$ as used in \eqref{proof_control_N_segnato} below. The more general case for the first  order dynamics considered in \cite{Chaudru:Menozzi17} would actually lead to linearize around a matrix which would depend on the freezing parameters. For such models, we did not succeed in proving that the corresponding densities can be bounded independently of the parameters (see also the proof of Lemma \ref{lemma:Control_P_N} below for a similar issue regarding the  diffusion coefficient).
\end{remark}

\begin{proof}
Fixed the freezing parameters $(\tau,\xi)$ in $[0,T]\times \R^N$, and the times $t<s$ in $[0,T]$, we start applying the It\^o-L\'evy decomposition to the process $\{\tilde{S}^{\tau,\xi,t,s}_u\}_{u\ge 0}$ introduced in Proposition \ref{prop:Decomposition_Process_X} at the associated characteristic stable time, i.e.\ we choose to truncate at threshold $u^{1/\alpha}$. Thus, we can write
\begin{equation}
\label{eq:decomposition_S}
\tilde{S}^{\tau,\xi,t,s}_u \,= \, \tilde{M}^{\tau,\xi,t,s}_u+\tilde{N}^{\tau,\xi,t,s}_u
\end{equation}
for some $\tilde{M}^{\tau,\xi,t,s}_u,\tilde{N}^{\tau,\xi,t,s}_u$ independent random
variables corresponding \textcolor{black}{respectively} to the small jumps part and the large jumps part. Namely, we denote for any $v>0$,
\begin{equation}\label{DEF_TILDE_N}
\tilde{N}^{\tau,\xi,t,s}_v \, := \, \int_{0}^{v}\int_{\vert z \vert >u^{1/\alpha}}zP_{\tilde{S}^{\tau,\xi,t,s}}(dr,dz) \,\, \text{ and } \,\, \tilde{M}^{\tau,\xi,t,s}_v: \,= \, \tilde{S}^{\tau,\xi,t,s}_v-\tilde{N}^{\tau,\xi,t,s}_v,
\end{equation}
where $P_{\tilde{S}^{\tau,\xi,t,s}}(dr,dz)$ is the Poisson random measure associated with the process $\tilde{S}^{\tau,\xi,t,s}$. It can be shown, similarly to Proposition \ref{prop:Decomposition_Process_X}, that the process $\{\tilde{M}^{\tau,\xi,t,s}_u\}_{u\ge 0}$ admits a density $p_{\tilde{M}^{\tau,\xi,t,s}}(u,\cdot)$. Indeed, it is well-known that the small jump part leads to a density which is in the Schwartz class $\mathcal{S}(\R^N)$ (see Lemma \ref{lemma:Control_p_M} below). We can then rewrite the density $p_{\tilde{S}^{\tau,\xi,t,s}}$ of $\tilde{S}^{\tau,\xi,t,s}$ in the following way:
\begin{equation}
\label{proof:decomposition_density_S}
    p_{\tilde{S}^{\tau,\xi,t,s}}(u,z) \, = \, \int_{\R^N}p_{\tilde{M}^{\tau,\xi,t,s}}(u,z-y)P_{\tilde{N}^{\tau,\xi,t,s}_u}(dy)
\end{equation}
where $P_{\tilde{N}^{\tau,\xi,t,s}_u}$ is the law of $\tilde{N}^{\tau,\xi,t,s}_u$.\newline
We need now to control the modulus of the density $p_{\tilde{S}^{\tau,\xi,t,s}}$ with another density, independently from the parameters $\tau$, $\xi$.
From Lemma \ref{lemma:Control_p_M} in the Appendix (see also Lemma B.$2$ in \cite{Huang:Menozzi16})  with $m =N+1$, we know that there exists a positive constant $C$, independent from $\tau,\xi$ such that
\begin{equation}
\label{proof:control_pM_segnato}
\left| D^k_{z} p_{\tilde{M}^{\tau,\xi,t,s}}(u,z) \right| \, \le \, Cu^{-(N+k)/\alpha}\left(\frac{u^{1/\alpha}}{u^{1/\alpha}+|z|}\right)^{N+2} \, =: \, Cu^{-k/\alpha}p_{\overline{M}}(u,z),
\end{equation}
for any $k$ in $\llbracket 0, 2 \rrbracket$, any $u>0$, and any $z$ in $\R^N$.\newline
Moreover, denoting by $\overline{M}_u$ the random variable with density $p_{\overline{M}}(u,\cdot)$ that is independent from $\tilde{N}^{\tau,\xi,t,s}_u$, we can easily check that $p_{\overline{M}}(u,z)=u^{-N/\alpha}p_{\overline{M}}(1,u^{-1/\alpha}z)$ and thus, that $\overline{M}$ is $\alpha$-selfsimilar:
\[\overline{M}_u \, \overset{{\rm law}}{=} \,
u^{1/\alpha}\overline{M}_1.\]
On the other hand, Lemma \ref{lemma:Control_P_N} in the Appendix (see also Lemma A.$2$ in \cite{Frikha:Konakov:Menozzi21}) ensures the existence of a family $\{\overline{P}_u\}_{u\ge 0}$ of probability measures  such that
\begin{equation}
\label{proof_control_N_segnato}
P_{\tilde{N}^{\tau,\xi,t,s}_u}(\mathcal{A}) \, \le \, C\overline{P}_u(\mathcal{A}), \quad \mathcal{A} \in \mathcal{B}(\R^N),
\end{equation}
for some positive constant $C$ independent from the parameters $\tau,\xi,t,s$.\newline
For any fixed $u\ge0$, let us now denote by $\overline{N}_u$ \textcolor{black}{a} random variable with law $\overline{P}_u$ that is independent from $\overline{M}_u$. Thanks to the representation of the measure $\overline{P}_u$ in \eqref{eq:representation_P_N_segnato}, it is then immediate to check that
\[\overline{N}_u \, \overset{(\text{law})}{=} \,
u^{1/\alpha}\overline{N}_1.\]
We can finally define the family $\{\overline{p}(u,\cdot)\}_{u\ge0}$ of densities as
\begin{equation}
\label{proof_decomposition_p_segnato}
\overline{p}(u,z) \,:=\, \int_{\R^N}p_{\overline{M}}(u,z-w)\overline{P}_{u}(dw),
\end{equation}
which corresponds to the density of the following random variable:
\[\overline{S}_u \, := \, \overline{M}_u +\overline{N}_u\]
for any fixed $u\ge 0$.
Using Fourier transform and the already proven $\alpha$-selfsimilarity of $\overline{M}$ and $\overline{N}$, we now show that
\[\overline{S}_u \, \overset{(\text{law})}{=} \, u^{1/\alpha}\overline{S}_1,\]
or equivalently, that
\[\overline{p}(u,z)\, = \, u^{-N/\alpha}\overline{p}(1,u^{-1/\alpha}z)\]
for any $u \ge 0$ and any $z$ in $\R^N$. Moreover,
\[\mathbb{E}[\vert \overline{S}_u\vert^\gamma] \, = \, \mathbb{E}[\vert \overline{M}_u+\overline{N}_u\vert^\gamma] \, = \,
Cu^{\gamma/\alpha}\bigl(\mathbb{E}[\vert \overline{M}_1\vert^\gamma]+\mathbb{E}[\vert \overline{N}_1\vert^\gamma]\bigr) \, \le \,  Cu^{\gamma/\alpha},\]
for any $\gamma <\alpha$. In particular, Equation \eqref{equation:integration_prop_of_q} holds. We emphasize that the integrability constraints precisely come from the Poisson measure $\overline{P}_u$ which behaves as the one associated with the large jumps of an $\alpha$-stable density.\newline
Equation \eqref{eq:derivative_prop_of_q} now follows easily from the previous arguments. From Equation \eqref{proof:decomposition_density_S}, we start noticing that Controls  \eqref{proof:control_pM_segnato}, \eqref{proof_control_N_segnato} and \eqref{proof_decomposition_p_segnato} imply that for any $k$ in $\llbracket 0,2\rrbracket$,
\[\left| D^k_{z} p_{\tilde{S}^{\tau,\xi,t,s}}(u,z) \right| \, \le \, Ct^{-k/\alpha}\overline{p}(u,z), \quad u\ge 0, \, z \in \R^N,\]
for some constant $C>0$, independent from the parameters $\tau,\xi,t,s$. Recalling the decomposition in \eqref{eq:decomposition_measure_1}, Equation \eqref{eq:derivative_prop_of_q} for $k=0$ already follows.\newline
To show instead the case $k=1$, we can write that
\[
\begin{split}
\bigl{\vert} D_{x_i}\tilde{p}^{\tau,\xi}(t,s,\,&x,y) \bigr{\vert} \, = \, \left| \frac{1}{\det (\mathbb{M}_{s-t})} D_{x_i} \left[p_{\tilde{S}^{\tau,\xi,t,s}}\bigl(s-t,\mathbb{M}^{-1}_{s-t} (y-\tilde{m}^{\tau,\xi}_{s,t}(x)))\right] \right| \\
&= \,\left| \frac{1}{\det (\mathbb{M}_{s-t})} \langle D_z p_{\tilde{S}^{\tau,\xi,t,s}}\left(s-t,\cdot\right)(\mathbb{M}^{-1}_{s-t}
(y-\tilde{m}^{\tau,\xi}_{s,t}(x))),D_{x_i}\mathbb{M}^{-1}_{s-t}\tilde{m}^{\tau,\xi}_{s,t}(x)\rangle \right| \\
&= \, \frac{(s-t)^{-1/\alpha}}{\det (\mathbb{T}_{s-t})}  \overline{p}\left(1,\mathbb{T}^{-1}_{s-t}
(y-\tilde{m}^{\tau,\xi}_{s,t}(x))\right)\left|D_{x_i}\mathbb{M}^{-1}_{s-t}\tilde{m}^{\tau,\xi}_{s,t}(x) \right|,
\end{split}
\]
where in the last step we exploited the $\alpha$-scaling property of $\overline{p}$. From Equation \eqref{eq:dif_differential_m_tilde}, we now notice that the function $x\to \tilde{m}^{\tau,\xi}_{s,t}(x)$ is affine, so that
\[\left| D_{x_i}\mathbb{M}^{-1}_{s-t}\tilde{m}^{\tau,\xi}_{s,t} (x) \right|  \, \le \,  C(s-t)^{-(i-1)}.\]
Hence, it follows that
\[\vert D_{x_i}\tilde{p}^{\tau,\xi}(t,s,x,y) \vert \, \le \, C\frac{(s-t)^{-\frac{1+\alpha(i-1)}{\alpha}}}{\det
(\mathbb{T}_{s-t})} \overline{p}\left(1,\mathbb{T}^{-1}_{s-t}
(y-\tilde{m}^{\tau,\xi}_{s,t}(x))\right).\]
The other case ($k=2$) can be derived in an analogous way.
\end{proof}

We conclude this section with a useful control on the powers of the density $\overline{p}(u,z)$.

\begin{corollary}
\label{coroll:control_power_q_density}
Let $q\ge 1$. Then, there exists a positive constant $C:=C(q)$ such that
\begin{equation}
\label{eq:control_power_q_density}
[\overline{p}(u,z)]^q \, \le \, u^{(1-q)\frac{N}{\alpha}}C\overline{p}(u,z), \quad (u,z) \in (0,T]\times \R^N.
\end{equation}
\end{corollary}
\begin{proof}
We start noticing that we can assume without loss of generality that $u=1$, thanks to the scaling property of $\overline{p}(u,z)$ in Proposition \ref{prop:Smoothing_effect}. Moreover, we know that there exists a constant $K$ such that $\overline{p}(1,z)\le 1$ for any $z$ in $B^c(0,K)$, since $\overline{p}(1,\cdot)$ is a density. It then clearly follows that
\[[\overline{p}(1,z)]^q \, \le \, \overline{p}(1,z), \quad z \in B^c(0,K).\]
On the other hand, we recall that $\overline{p}(1,\cdot)$ is continuous. For any $z$ in $B(0,K)$, it then holds that
\[[\overline{p}(1,z)]^q \, = \, \overline{p}(1,z) [\overline{p}(1,z)]^{q-1} \, \le \, C\overline{p}(1,z),\]
where $C$ is the maximum of $[\overline{p}(1,\cdot)]^q$ on $B(0,K)$.
\end{proof}

\subsection{Regularity of Density along the Terminal Condition} \label{SEZIONE_SPIEGAZIONE_CONGELAMENTO_E_ALTRI}
We briefly explain here how we want to prove the well-posedness of the martingale formulation associated with $\partial_s+L_s$ at some starting point $(t,x)$. We will mainly focus on the problem of uniqueness since the existence of a solution can be easily handled from already known results. Indeed, we recall that under the assumptions we consider, the main part of the operator $L_s$ is of order $\alpha>1$ while the perturbation is sub-linear. Thus, the existence of a solution  can be obtained, for example, from Theorem $2.2$ in \cite{Stroock75}.
\newline
In particular,  uniqueness for the martingale problem will follow  once the Krylov-like estimates \eqref{eq:Krylov_Estimates} have been shown.\newline
Starting from a solution $\{X^{t,x}_s\}_{s\in[0,T]}$ of the martingale problem with starting point $(t,x)$, the idea is to exploit the properties of the frozen dynamics $\{\tilde{X}^{\tau,\xi,t,x}_s\}_{s\in[0,T]}$ in \eqref{eq:integral_represent_frozen_SDE}. For this reason, let us denote by $\tilde{L}^{\tau,\xi}_s$ its infinitesimal generator and define for $f$ in $ C_c^{1,2}([0,T)\times \R^{N})$ the associated Green kernel:
$$\tilde{G}^{\tau,\xi}(t,x)=\int_t^T ds \int_{\R^N}\tilde{p}^{\tau,\xi}(t,s,x,y) f(s,y)dy.$$
Standard results now give that
\begin{equation}
\label{eq:intro_MP}
    \left(\partial_t+\tilde{L}^{\tau,\xi}_t\right)  \tilde{G}^{\tau,\xi}f(t,x) \, = \, -f(t,x), \quad (t,x) \in [0,T)\times \R^N,
\end{equation}
for any (fixed) freezing parameters $(\tau,\xi)$.\newline
The first step of our method then consists in applying the It\^o formula on the function $\tilde{G}^{\tau,\xi} f$, which is indeed smooth enough, and the solution process $\{X^{t,x}_s\}_{s\in[0,T]}$:
\[
\tilde{G}^{\tau,\xi} f(t,x)+ \mathbb{E}\left[\int_t^{T} (\partial_s+L_s)\tilde{G}^{\tau,\xi} f(s,X^{t,x}_s)   \,ds\right] \, =\, 0.
\]
Exploiting \eqref{eq:intro_MP}, we can then write
\[
\tilde{G}^{\tau,\xi} f (t,x)- \mathbb{E}\Bigl[\int_t^{T} f(s,X^{t,x}_s)\,  ds \Bigr]+\mathbb{E}\Bigl[\int_{t}^T\bigl(L_s-\tilde{L}^{\tau,\xi}_s\bigr)                          \tilde{G}^{\tau,\xi} f(s,X^{t,x}_s) \, ds\Bigr] \, = \, 0
\]
or, equivalently,
\[
\mathbb{E}\Bigl[\int_t^{T} f(s,X^{t,x}_s) \, ds \Bigr] \, = \, \tilde{G}^{\tau,\xi} f
(t,x)+\mathbb{E}\Bigl[\int_{t}^T \bigl(L_s-\tilde{L}^{\tau,\xi}_s\bigr) \tilde{G}^{\tau,\xi} f(s,X^{t,x}_s)\, ds\Bigr].\]
While an estimate of the frozen Green kernel $\tilde{G}^{\tau,\xi} f$ can be obtained from Proposition \ref{prop:Smoothing_effect}, the main difficulty of our approach will be to control, uniformly in $(t,x)$, the following quantity:
\[\int_{t}^T\int_{\R^N}\bigl(L_s-\tilde{L}^{\tau,\xi}_s\bigr) \tilde{G}^{\tau,\xi} f(s,x)\, ds.\]
Focusing for example only on the component associated with the deterministic drift $F$, i.e.
\[\int_t^T\int_{\R^N}\langle F(t,x)-F(t,\theta_{t,\tau}(\xi)),D_x \tilde{p}^{\tau,\xi}(t,s,x,y)\rangle f(s,y)\, dyds,\]
it is clear that we need some kind of compatibility between the arguments of the drift $F$ and those of the frozen density $\tilde{p}^{\tau,\xi}(t,s,x,\cdot)$, in order to exploit the associated smoothing effect (Proposition \ref{prop:Smoothing_effect}). Namely, we need to compare the quantities $(x-\theta_{t,\tau}(\xi))$ and $(y-\tilde{m}^{\tau,\xi}_{s,t}(x))$.\newline
Noticing that for $\tau=s$ and $\xi=y$, $(y-\tilde{m}^{\tau,\xi}_{s,t}(x))=\theta_{t,s}(y)-x$, it follows from Proposition \ref{prop:Smoothing_effect} that this choice of freezing parameters gives the natural compatibility between the difference of the generators and the upper-bounds of the derivatives of the corresponding proxy.

The above reasoning requires however a more thorough analysis on the ``density'' $\tilde{p}^{s,y}(t,s,x,\cdot)$ frozen along the terminal condition $(\tau,\xi)$. Indeed, the freezing parameter $y$ appears also as the integration variable. In other words, with this approach, the freezing parameter cannot be fixed once for all. The present section is precisely dedicated to the handling of such a choice. This will lead us to introduce a \textit{pseudo} Green kernel, see \eqref{eq:def_Green_Kernel} below, from which will then derive uniqueness to the martingale problem following the Stroock and Varadhan approach (see Chapter 7 in \cite{book:Stroock:Varadhan79}), through appropriate inversion in $L_t^q-L_x^p $ spaces,  proving that the remainder has a small corresponding norm.

We start with a lemma showing the existence of at least one version of the flow $\theta_{t,s}(y)$ which is measurable in $s$ and $y$. This result will be fundamental to make licit any integration of this flow along the terminal condition $y$.

\begin{lemma}\label{lemma:measurability_flow}
There exists a measurable mapping $\theta \colon [0,T]^2\times \R^N \to\R^N$ such that
\begin{equation}
\label{eq:measurability_flow}
\theta(t,s,z)\, := \, \theta_{t,s}(z)\, =    z + \int_{t}^{s}\bigl[ A_u \theta_{u,s}( z)+F(u, \theta_{u,s}( z))\bigr] \, du.
\end{equation}
\end{lemma}
\begin{proof}
The result can be obtained from \cite{Zubelevich12} and a standard compactness argument.
\end{proof}
From this point further, we assume without loss of generality to have chosen such a measurable version $\theta_{t,s}(x)$ of the reference flow.

The next Lemma \ref{lemma:bilip_control_flow} (Approximate Lipschitz condition of the flows) will be a key technical tool for our method. Roughly speaking, it says that a kind of equivalence between the rescaled forward and backward flows appears even in our framework (where the drift $F$ is not regular enough), up to an additional constant contribution, for any two measurable flows satisfying Equation \eqref{eq:def_Cauchy_Peano_flow}. We only remark that similar results has been thoroughly exploited in \cite{Delarue:Menozzi10, Menozzi11, Menozzi18} when considering Lipschitz drifts or \cite{Chaudru:Menozzi17} in the degenerate diffusive setting with H\"older coefficients. \newline
This \emph{approximated} Lipschitz property will be fundamental later on in the proof of Lemma \ref{convergence_dirac} (Dirac Convergence of frozen density) below. It will be proved in Appendix $A.1$, adapting the lines of \cite{Chaudru:Menozzi17}.

\begin{lemma}\label{lemma:bilip_control_flow}
Let $\theta \colon [0,T]^2\times \R^N \to\R^N$, $\check{\theta} \colon [0,T]^2\times \R^N \to\R^N$ be two measurable flows satisfying Equation \eqref{eq:measurability_flow}. Then, there exist two positive constants $(C,C'):=(C,C')(T)$ such that for any $t<s$ in $[0,T]$ and any $x,y$ in $\R^N$,
\begin{equation}
\label{eq:bilip_control_flow}
C^{-1} |\mathbb{T}_{s-t}^{-1}(\check{\theta}_{s,t}(x)-y)| - C' \, \le \, |\mathbb{T}_{s-t}^{-1}(x-\theta_{t,s}(y))|\,\le\, C\left[|\mathbb{T}_{s-t}^{-1}(\check{\theta}_{s,t}(x)-y)| +1\right].
\end{equation}
\end{lemma}
From the above lemma, we also derive the following important estimate for the rescaled difference between the forward flow $\theta_{s,t}(x)$ and the linearized forward dynamics $\tilde{m}^{s,y}_{s,t}(x)$ (defined in \eqref{eq:def_tilde_m}) where the linearization is considered along any backward flow.

\begin{corollary}
\label{coroll:control_error_flow}
Let $\theta \colon [0,T]^2\times \R^N \to\R^N$ be a measurable flow satisfying Equation \eqref{eq:measurability_flow}. Then, there exist a positive constant $C:=C(T)$ and $\zeta$ in $(0,1)$ such that for any $t<s$ in $[0,T]$ and any $x,y$ in $\R^N$,
\begin{equation}
\label{eq:control_error_flow}
 |\mathbb{T}_{s-t}^{-1}(\theta_{s,t}(x)-\tilde{m}^{s,y}_{s,t}(x))| \,\le \, C (s-t)^{\frac{1}{\alpha} \wedge \zeta}\left(1+|\mathbb{T}_{s-t}^{-1}(\theta_{s,t}(x)-y)|\right).
\end{equation}
\end{corollary}
\begin{proof}
We start exploiting the differential dynamics given in Equation \eqref{eq:dif_differential_m_tilde} to write that
\begin{equation}
\label{proof:eq:decomposition1}
\begin{split}
   \mathbb{T}_{s-t}^{-1}\left(\theta_{s,t}(x)-\tilde{m}^{s,y}_{s,t}(x)\right)
\, &= \, \mathbb{T}_{s-t}^{-1}
\int_{t}^{s}  \Bigl{\{}\bigl[ F(u,\theta_{u,t}(x))-F(u,\theta_{u,s}(y)) \bigr] \\
&\qquad \quad\qquad \qquad+\left[A_u(\theta_{u,t}(x)-\tilde{m}^{s,y}_{u,t}(x))\right] \Bigr{\}}\, du  \\
&:= \, (\mathcal{I}_{s,t}^{1}+\mathcal{I}_{s,t}^{2})(x,y).
\end{split}
\end{equation}
We start dealing with $\mathcal{I}_{s,t}^{1}(x,y)$. The key idea is to use the sub-linearity of $F$ and the appropriate H\"older exponents.
Namely, using the Young inequality, we derive that
\[
\begin{split}
|&\mathcal{I}_{s,t}^{1}(x,y)| \, \le \, C\sum_{i=1}^n (s-t)^{-\frac{1+\alpha(i-1)}{\alpha}}\sum_{j=i}^n \int_t^{s}  |(\theta_{u,t}(x)-\theta_{u,s}(y) )_{j}|^{\beta^j} \, du\\
&\,\,\le \, C\biggl{\{}(s-t)^{-\frac{1}{\alpha}}\int_t^{s} \left[|\theta_{u,t}(x)-\theta_{u,s}(y) |+1\right] \, du \\
&\quad\quad +\sum_{i=2}^n (s-t)^{-\frac{1+\alpha(i-1)}{\alpha}} \sum_{j=i}^n \int_t^{s} \Bigl[ (s-t)^{-\gamma^j}|((\theta_{u,t}(x)-\theta_{u,s}(y) )_{j})|+(s-t)^{\gamma^j \frac{\beta^j}{1-\beta^j}} \Bigr] \, du\biggr{\}},
\end{split}
\]
for some parameters $\gamma^j>0$ to be specified below.
Denoting now for simplicity, \[\Gamma_j\,:= \, -\frac{1+\alpha(i-1)}{\alpha}+\gamma^j \frac{\beta^j}{1-\beta^j},\]
we get that
\[
\begin{split}
|\mathcal{I}_{s,t}^{1}(x,
y)|\, &\le \, C\biggl{\{}(s-t)^{\frac{\alpha-1}{\alpha}}+\int_t^{s}|\mathbb{T}_{s-t}^{-1}(\theta_{u,t}(x)-\theta_{u,s}(y) )|\, du \\
& \,\,\quad+\sum_{i=2}^n  \sum_{j=i}^n \int_t^{s} \Bigl[(s-t)^{-i+j -\gamma^j}  \Big(\frac{|((\theta_{u,t}(x)-\theta_{u,s}(y) )_{j})|}{(s-t)^{\frac{1+\alpha(j-1)}{\alpha}}}\Big)+(s-t)^{-\Gamma_j} \Bigr]\, du\biggr{\}}\\
&\le \, C\biggl{\{}(s-t)^{\frac{\alpha-1}{\alpha}}+\int_t^{s}  |\mathbb{T}_(s-t)^{-1}(\theta_{u,t}(x)-\theta_{u,s}(y))|\, du\\
&\quad\,\, +\sum_{i=2}^n  \sum_{j=i}^n \int_t^{s}  \Bigl[(s-t)^{-i+j -\gamma^j}  |\mathbb{T}_{s-t}^{-1}(\theta_{u,t}(x)-\theta_{u,s}(y) )|+(s-t)^{\Gamma_j} \Bigr]\, du\biggr{\}}.
\end{split}
\]
We now use Lemma \ref{lemma:bilip_control_flow} (Approximate Lipschitz condition of the flows) to derive that
\[ |\mathbb{T}_{s-t}^{-1}(\theta_{u,t}(x)-\theta_{u,s}(y) )| \, \le \, C(|\mathbb{T}_{s-t}^{-1}(\theta_{s,t}(x)-y )|+1).\]
We emphasize here that in our current framework we should \textit{a priori} write $\theta_{s,u}(\theta_{u,t}(x)) $ in the above equation since we do not have the flow property. Anyhow, since Lemma \ref{lemma:bilip_control_flow} (Approximate Lipschitz condition of the flows) is valid for any flow starting from $\theta_{u,t}(x) $ at time $u$ associated with the ODE (see Equation \eqref{eq:bilip_control_flow}) we can proceed along the previous one, i.e. $(\theta_{v,t}(x))_{v\in [u,s] }$. The previous reasoning yields that
\begin{equation}
\begin{split}
  \label{PREAL_CTR_I1_EPS}
|\mathcal{I}_{s,t}^{1}(x,y) |\, &\le \,C\biggl{\{}(s-t)^{\frac{\alpha-1}{\alpha}}+(s-t) \bigl[|\mathbb{T}_{s-t}^{-1}(\theta_{s,t}(x)-y)|+1\bigr]\\
& \qquad \qquad \times \Bigl[1+\sum_{i=2}^n  \sum_{j=i}^n  \Bigl((s-t)^{-i+j -\gamma^j}+(s-t)^{-\frac{1+\alpha(i-1)}{\alpha}+\gamma^j \frac{\beta^j}{1-\beta^j}} \Bigr) \Bigr] \biggr{\}}.
\end{split}
\end{equation}
We now choose for $j$ in $\llbracket i,n\rrbracket$,
\[-i+j -\gamma^j \, = \, -\frac{1+\alpha(i-1)}{\alpha}+\gamma^j \frac{\beta^j}{1-\beta^j}\, \Leftrightarrow\, \gamma^j\, = \, \left(j-\frac{\alpha-1}{\alpha}\right)\left(1-\beta^j\right),\]
to balance the two previous contributions associated with the indexes $i,j$. \newline
To obtain a global smoothing effect with respect to ${s-t}$ in \eqref{PREAL_CTR_I1_EPS} we need to impose:
\begin{equation}
\label{eq:natural_threshold}
-i+j-\gamma^j\, >\,-1 \, \Leftrightarrow \, \beta^j\, >\, \frac{1+\alpha(i-2)}{1+\alpha(j-1)}, \quad \forall i \le j.
\end{equation}
Hence, under our assumptions, we have that there exists $\zeta$ in $(0,1)$ depending on $\beta^j$ for any $j\in \llbracket i,n\rrbracket$ such that
\begin{equation}
\label{CTR_I1_EPS}
|\mathcal{I}_{s,t}^{1}(x,y)| \,  \le \, C(s-t)^\zeta\left[1+|\mathbb{T}_{s-t}^{-1}(\theta_{s,t}(x)-y)|\right].
\end{equation}
Recalling from the structure of $A$ that
\[|\mathbb{T}^{-1}_{s-t}A_u\mathbb{T}_{s-t}| \le \, C(s-t)^{-1},\]
Control \eqref{eq:control_error_flow} now follows from
\eqref{proof:eq:decomposition1}, \eqref{CTR_I1_EPS} and the Gr\"onwall lemma.
\end{proof}

Thanks to the Approximate Lipschitz property of the flow presented in Lemma \ref{lemma:bilip_control_flow} above and Corollary \ref{coroll:control_error_flow}, we can now adapt the controls on the derivatives of the frozen density (Proposition \ref{prop:Smoothing_effect}) to the ``density'' $\tilde{p}^{s,y}(t,s,x,y)$. Indeed, we recall again that the function $\tilde{p}^{s,y}(t,s,x,y)$ is not a proper density in $y$ since the integration variable $y$ stands also as freezing parameter. This is one of the main difficulties of the approach.

The following result is the key to our analysis since it precisely quantifies the smoothing effect in time of the proxy we chose.
\begin{corollary}
\label{coroll:Smoothing_effect}
There exists a positive constant $C:=C(N,\alpha)$ such that for any $\gamma$ in $[0,\alpha)$, any $t< s$ in $[0,T]$ and any $x,y$ in $\R^N$,
  \begin{equation}
\label{eq:smoothing_effect_frozen_y}
\int_{\R^N} \frac{|\mathbb{T}^{-1}_{s-t}(\theta_{t,s}(y)-x)|^\gamma}{\det \mathbb{T}_{s-t}}\overline{p}(1,\mathbb{T}^{-1}_{s-t}(\theta_{t,s}(y)-x)) \, dy \, \le \, C.
  \end{equation}
Moreover, if $K>0$ is large enough, it holds that
  \begin{equation}
\label{eq:smoothing_effect_frozen_y_GENERIC_FUNCTION}
\textcolor{black}{\int_{\R^N} \mathds{1}_{|{\mathbb{T}^{-1}_{s-t}(\theta_{t,s}(y)-x)}|\ge K} \frac 1{\det \mathbb{T}_{s-t}}\overline{p}(1,\mathbb{T}^{-1}_{s-t}(\theta_{t,s}(y)-x)) \, dy \,
\le \, C \int_{\R^N} \mathds{1}_{|z|\ge \frac K2}\check p(1,z) dz},
  \end{equation}
  where $\check p $ enjoys the same integrability properties as $\overline{p}$ (stated in Proposition \ref{prop:Smoothing_effect}).
\end{corollary}
For the sake of clarity the proof of Corollary \ref{coroll:Smoothing_effect} is postponed to the Appendix.

\begin{remark}[\textcolor{black}{About the uniformity in $x_i$ for the H\"older regularity threshold}]
\label{DA_SCRIVERE_VINCOLO_SUL_MODELLO2}
The strengthened assumptions concerning the integrability thresholds in Theorem \ref{thm:main_result} with respect to the natural ones appearing in \eqref{eq:natural_threshold} might seem awkward at first sight.\newline
Indeed, the natural approach to get rid of the flow involving the integration variable in \eqref{eq:smoothing_effect_frozen_y} would have been to use the approximate Lipschitz property of the flow established in Lemma \ref{lemma:bilip_control_flow}. This indeed readily yields that:
$$ |\mathbb{T}^{-1}_{s-t}(\theta_{t,s}(y)-x)|^\gamma\le C(1+ |\mathbb{T}^{-1}_{s-t}(y-\theta_{s,t}(x))|^\gamma).$$
The main difficulty is that we  do not actually succeed in establishing in whole generality that:
\begin{equation}\label{equiv_density_flows}
\overline{p}(1,\mathbb{T}^{-1}_{s-t}(\theta_{t,s}(y)-x))\le C \check{p}(1,\mathbb{T}^{-1}_{s-t}(y-\theta_{s,t}( x)),
\end{equation}
for a density $\check p $ which shares the same integrability properties as $\bar p$.\newline
Equation \eqref{equiv_density_flows} is absolutely direct in the diffusive setting from the explicit form of the Gaussian density and it has been thoroughly used in \cite{Chaudru:Menozzi17} to derive sharp thresholds for weak uniqueness. It is clear that the above control has to be considered point-wise and one of the huge difficulties with stable type processes consists in describing precisely their tail behavior which is actually very much related to the geometry of their corresponding spectral measure on the sphere. We refer to the seminal work of Watanabe \cite{Watanabe07} for a precise description of the tails in terms of the dimension of the support of the spectral measure, in the stable case, and to the extension by Sztonyk \cite{Sztonyk10} for the tempered stable case. The delicate point comes of course from the behavior of the Poisson measure (large jumps) as illustrated in  the following computation. From \eqref{proof:control_pM_segnato} and \eqref{proof_decomposition_p_segnato}, we write that
\begin{align*}
\overline{p}(1,\mathbb{T}^{-1}_{s-t}(\theta_{t,s}(y)-x))\, &= \, \int_{\R ^N} p_{\bar M}(1,\mathbb{T}^{-1}_{s-t}(\theta_{t,s}(y)-x)-w) \, \overline{P}_1(dw)\notag\\
&\le \, C\int_{\R^N} \frac{1}{(C+|\mathbb{T}_{s-t}^{-1}(\theta_{t,s}(y)-x)-w|)^M} \, \overline{P}_1(dw).
\end{align*}
Let us first emphasize that, when $|\mathbb{T}^{-1}_{s-t}(\theta_{t,s}(y)-x)|\le K $ (\textit{diagonal type regime}) for some fixed $K$, then Control \eqref{equiv_density_flows} holds. Indeed, since from Corollary \ref{coroll:control_error_flow},
\[|\mathbb{T}_{s-t}^{-1}(\tilde m_{s,t}^{t,y}(x)-\theta_{s,t}(x))|\, \le \, \tilde C(s-t)^{\frac{1}{\alpha} \wedge \zeta}\left(1+|\mathbb{T}_{s-t}^{-1}(\theta_{s,t}(x)-y)|\right),\]
we would get, recalling from Lemma \ref{lemma:identification_theta_m}, Equation \eqref{eq:identification_theta_m1} that $\theta_{t,s}(y)-x =y-\tilde{m}^{s,y}_{s,t}(x)$, that
\begin{align*}
\overline{p}(&1,\mathbb{T}^{-1}_{s-t}(\theta_{t,s}(y)-x))\, \le \, C\int_{\R ^N} \frac{1}{(C+|\mathbb{T}_{s-t}^{-1}(\theta_{t,s}(y)-x)-w|)^M} \,\overline{P}_1(dw)\\
 &\le \,C\int_{\R^N}\frac{1}{([C+ |\mathbb{T}_{s-t}^{-1}(y-\theta_{s,t}(x))-w|-(s-t)^{\frac{1}\alpha\wedge \zeta}|\mathbb{T}_{s-t}^{-1}(\theta_{s,t}(x)-y)|]\vee 1)^M} \,\overline{P}_1(dw)\\
 &\le \,C\int_{\R^N}\frac{1}{([\check C+ |\mathbb{T}_{s-t}^{-1}(y-\theta_{s,t}(x))-w|)^M} \, \overline{P}_1(dw)\\
 &=: \, \check p(1,\mathbb{T}^{-1}_{s-t}(y-\theta_{s,t}(x))),
\end{align*}
and $\check p $  plainly satisfies the required integrability conditions.
These computations actually emphasize that \eqref{equiv_density_flows} holds, up to a modification of $\check C $ above, up to the threshold \[|\mathbb{T}^{-1}_{s-t}(\theta_{t,s}(y)-x)|\, \le\, c_0(s-t)^{-(\frac{1}{\alpha}\wedge \zeta)},\]
for some $c_0>0$ small enough with respect to $C$. It would therefore remain to investigate the complementary \textit{very off-diagonal regime}.\newline
Let us now concentrate on the \textit{off-diagonal regime} $|\mathbb{T}^{-1}_{s-t}(\theta_{t,s}(y)-x)|> K  $. In that case, we write:
\begin{align}
\overline{p}(1,\mathbb{T}^{-1}_{s-t}(\theta_{t,s}(y)-x))
&\, \le \, C\int_{\R ^N} \frac{1}{(C+|\mathbb{T}_{s-t}^{-1}(\theta_{t,s}(y)-x)-w|)^M} \overline{P}_{1}(dw)\notag\\
&\le \, C\int_{0}^1\overline{P}_{1}(\{w\in \R^N\colon(1+|\mathbb{T}_{s-t}^{-1}(\theta_{t,s}(y)-x)-w|)^{-M} |>u  \})du\notag\\
&\le \, C\int_{0}^1\overline{P}_{1}(B(\mathbb{T}_{s-t}^{-1}(\theta_{t,s}(y)-x), u^{-1/M})du. \label{eq:spiegazione1}
 \end{align}
 It now follows from the proof of Proposition \ref{prop:Decomposition_Process_X} that the support of the spectral measure on $\mathbb S^{N-1} $ associated with  $\{\tilde{S}^{\tau,\xi,t,s}_{u}\}_{u\ge 0}$ has dimension $d$. The related concentration properties also transmit to $\bar N_1 $ (see the proofs of Proposition \ref{prop:Smoothing_effect}
 and Lemma \ref{lemma:Control_P_N}). Thus, we get from \cite{Watanabe07}, \cite{Sztonyk10} (see respectively Lemma 3.1 and Corollary 6 in those references) that there exists a constant $C>0$ such that for all $z$ $\R^{N}$ and $r>0$:
\begin{equation}
\label{EST_POISSON}
\overline{P}_{1}(B(z,r))\le Cr^{d+1} (1+r^\alpha)|z|^{-(d+1+\alpha)}.
\end{equation}
In other words, the global bound is given by the worst decay deriving from the dimension of the support of the spectral measure.
In the current case $|z|\ge K $, this bound is clearly of interest for \textit{large} values of $z$. Hence, from \eqref{eq:spiegazione1} and \eqref{EST_POISSON}, it holds that
 \[
 \begin{split}
\overline{p}(1,\mathbb{T}^{-1}_{s-t}(\theta_{t,s}(y)&-x)) \, \le \, C\int_0^1u^{-(d+1)/M}
 (1+
 u^{-\alpha/M})\, du
|\mathbb{T}_{s-t}^{-1}(\theta_{t,s}(y)-x)|^{-(d+1+\alpha)}
\\
&\le \, C (1+|\mathbb{T}_{s-t}^{-1}(\theta_{t,s}(y)-x)|)^{-(d+1+\alpha)}
\int_0^1 [u^{-(d+1)/M}+u^{-(d+1+\alpha)/M} du],
\end{split}
\]
Choosing  $M>d+1+\alpha $ then gives that there exists $C\ge 1$ such that
$$
\overline{p}(1,\mathbb{T}^{-1}_{s-t}(\theta_{t,s}(y)-x)) \le  C (1+|\mathbb{T}_{s-t}^{-1}(\theta_{t,s}(y)-x)|)^{-(d+1+\alpha)}
.$$
We thus get from Lemma \ref{lemma:bilip_control_flow}, up to a modification of $C$, that:
\begin{align}\label{BAD_BOUND}
\overline{p}(1,\mathbb{T}^{-1}_{s-t}(\theta_{t,s}(y)-x)) \le  C (1+|\mathbb{T}_{s-t}^{-1}(y-\theta_{s,t}(x))|)^{-(d+1+\alpha)}.
\end{align}
This actually leads to strong dimension constraints for this bound to be integrable. This phenomenon already appeared e.g. in  \cite{Huang:Menozzi16}
and \textcolor{black}{required} therein to consider $d=1,n=3$ at most to address the well posedness of the martingale problem associated with a linear drift and a multiplicative isotropic stable noise. Those thresholds and dimension constraints remain with this approach. Actually, from the threshold appearing in \eqref{eq:thresholds_beta}, we would like to consider the left-hand side of \eqref{eq:smoothing_effect_frozen_y} with $\gamma>\frac{1+\alpha}{1+2\alpha}$ corresponding to $j=3=n$ therein. From Control \eqref{BAD_BOUND}, this would require $-\frac{1+\alpha}{1+2\alpha}+(d+1+\alpha)>3, d=1 \iff \alpha^2-\alpha-1>0$, which in our framework imposes that $\alpha\in (\frac{1+\sqrt 5}2,2)$.\newline
Another possibility would have been, in the tempered case, to keep track of the tempering function, instead of bounding $\tilde p^{\tau,\xi} $ by a self-similar density $\bar p $, in order to benefit from the tempering at infinity to compensate the bad concentration rate in \eqref{BAD_BOUND}. However, see \cite{Huang:Menozzi16} and \cite{Sztonyk10}, we would have obtained bounds of the form
$$\tilde p^{\tau,\xi} (t,s,x,y)\le  C (1+|\mathbb{T}_{s-t}^{-1}(y-\theta_{s,t}(x))|)^{-(d+1+\alpha)}Q\left( |\mathbb M_{s-t}^{-1}(y- \theta_{s,t}(x))|\right).
$$
Such a bound will give space integrability but deteriorates as well the time-integrability.
This difficulty would occur even in the truncated case, thoroughly studied in the non-degenerate case by Chen \textit{et al.} \cite{Chen:Kim:Kumagai08}.\newline
Thus, we will develop here another approach. Namely, we would like to change variable to $\bar y:=\mathbb{T}^{-1}_{s-t}(\theta_{t,s}(y)-x) $ in the left-hand side of Equation \eqref{eq:smoothing_effect_frozen_y}. Of course, this is not bluntly possible since the coefficients at hand are not smooth enough. The point is then to introduce a flow $\theta_{t,s}^\delta(y) $ associated with mollified coefficients (for which the difference with respect to the initial flow will be controlled similarly to what is done to establish the approximate Lipschitz property of the flows in Lemma \ref{lemma:bilip_control_flow}) and then, to control  $\det(\nabla \theta_{t,s}^\delta(y)) $ (see Lemma \ref{LEMMA_FOR_DET} below). Since we do not have here a summation with respect to the single rescaled components as in the previous Lemma \ref{lemma:bilip_control_flow} above or as in Corollary \ref{coroll:control_error_flow}, this will conduct to reinforce our assumptions and suppose that $(F_i)_{i\in \llbracket 2,n\rrbracket} $ has the same regularity with respect to the variable $x_j$, $j\in \llbracket 2,n\rrbracket$, whatever the level of the chain.  This is precisely  what leads to consider the condition
\begin{itemize}
\item $x_j \to F_i(t,x_i,\dots,x_j,\dots,x_n)$ is $\beta^j$-H\"older continuous, uniformly in $t$ and in $x_k$ for $k\neq j$, with
 \begin{equation*}
\beta^j \, > \, \frac{1+\alpha(j-2)}{1+\alpha(j-1)}.
\end{equation*}
\end{itemize}
\end{remark}

Let us introduce now some useful tools for the study of the martingale problem for $\partial_s+L_s$. The first step is to consider a suitable Green-type kernel associated with the frozen density $\tilde{p}^{s,y}$ and establish which Cauchy-like problem it solves. Namely, we define for any function $f\colon [0,T]\times \R^N \to \R$ regular enough, the \textit{pseudo} Green kernel $\tilde{G}_\epsilon$ given by:
\begin{equation}\label{eq:def_Green_Kernel}
\tilde{G}_\epsilon f(t,x) \, := \, \int_{(t+\epsilon)\wedge T}^T \int_{\R^N}  \tilde{p}^{s,y}(t,s,x,y) f(s,y) \, dy ds, \quad (t,x) \in [0,T)\times \R^N,
\end{equation}
where $\epsilon$ is meant to be small.\newline
We only remark that the above Green kernel $\tilde{G}_\epsilon$ is well-defined, since the frozen density $\tilde{p}^{s,y}(t,s,
x,y)$ is measurable in $(s,y)$ thanks to Lemma \ref{lemma:measurability_flow} (measurability of the flow in these parameters).

\begin{prop}
\label{prop:pointwise_green_kernel}
Let $p$, $q$ in $(1,+\infty)$ such that the integrability Condition $(\mathscr{C})$ holds. Then, there exists a positive constant $C:=C(T,p,q)$  such that for any $f$ in
$L^p\big(0,T;L^q(\R^N)\bigr)$,
\[
\Vert \tilde{G}_{\epsilon}f\Vert_\infty \, \le \, C \Vert f \Vert_{L^p_tL^q_x}.\]
Moreover, it holds that $\lim_{T\to 0}C(T,p,q) = 0$.
\end{prop}
\begin{proof}
We start using the H\"older inequality in order to split the component with $f$ and the part with the density $\tilde{p}(t,s,x,y)$:
\[
\begin{split}
|\tilde{G}_\epsilon f(t,x)| \,
&\le \, C \Vert f \Vert_{L^p_tL^q_x} \left(\int_{(t+\epsilon)\wedge T}^T \left (\int_{\R^N} \left|\tilde{p}^{s,y}(t,s,x,y)\right|^{q'} \,dy \right)^{\frac{p'}{q'}} ds \right)^\frac{1}{p'} \\
&=: \, C \Vert f \Vert_{L^p_tL^q_x}|I_\epsilon(t,x)|,
\end{split}
\]
where we have denoted by $p'$, $q'$ the conjugate of $p$ and $q$, respectively.\newline
In order to control the remainder term $I_\epsilon(t,x)$, we now apply \eqref{eq:derivative_prop_of_q} from Proposition \ref{prop:Smoothing_effect} with $k=0$ and $(\tau,\xi)=(s,y)$ to write that
\[
|I_\epsilon(t,x)|^{p'} \, \le \, C\int_{(t+\epsilon)\wedge T}^T \left(\int_{\R^N} \left(\frac{1}{\det \mathbb{T}_{s-t}} \overline{p}\left(1, \mathbb{T}^{-1}_{s-t}(y-\tilde{m}^{s,y}_{s,t}(x))\right)\right)^{q'} \, dy \right)^{\frac{p'}{q'}} ds,
\]
where we recall that $\mathbb{T}_t=t^{1/\alpha}\mathbb{M}_t$ (see \eqref{eq:def_matrix_T} and \eqref{eq:def_matrix_M}).\newline
From Corollaries \ref{coroll:control_power_q_density} and \ref{coroll:Smoothing_effect}, we then write that
\[\begin{split}
|I_\epsilon(t,x)|^{p'} \, &\le \,
C\int_{(t+\epsilon)\wedge T}^T \left(\int_{\R^N} \frac{1}{(\det \mathbb{T}_{s-t})^{q'}} \overline{p}\left(1, \mathbb{T}^{-1}_{s-t}(y-\tilde{m}^{s,y}_{s,t}(x))\right) \, dy \right)^{\frac{p'}{q'}} ds\\
&\le \, C \int_{(t+\epsilon)\wedge T}^T  (\det \mathbb{T}_{s-t})^{\frac{p'}{q'}-p'} ds\, =  \, C \int_{(t+\epsilon)\wedge T}^T   \frac{1}{(\det \mathbb{T}_{s-t})^{\frac{p'}{q}}} ds.
\end{split}
\]
Since by definition of matrix $\mathbb{T}_t$, it holds that
\begin{equation}
\label{eq:control_det_T1}
\det \mathbb{T}_{s-t} \, = \, (s-t)^{\sum_{i=1}^nd_i\frac{1+\alpha(i-1)}{\alpha}},
\end{equation}
we can conclude that under the integrability assumption $(\mathscr{C})$, we have that
\begin{equation}
\label{eq:control_det_T2}
\bigl(\sum_{i=1}^nd_i\frac{1+\alpha(i-1)}{\alpha}\bigr)\frac{p'}{q} \, < \, 1  \, \Leftrightarrow \, \bigl(\sum_{i=1}^nd_i\frac{1+\alpha(i-1)}{\alpha}\bigr)\frac{1}{q}+\frac{1}{p} \, <\,  1.
\end{equation}
The proof is complete.
\end{proof}

We now want to understand which Cauchy-like problem is solved by the \textit{density} $\tilde{p}^{s,y}(t,s,x,y)$ frozen at the terminal point $(s,y)$. We start denoting by $\tilde{L}^{s,y}_t$ the infinitesimal generator of the proxy process $\{\tilde X_{s}^{s,y,t,x}\}_{s\in [t,T]}$. For any smooth function $\phi\colon \R^N\to\R$, it writes:
\begin{equation}
\label{eq:def_frozen_generator}
\begin{split}
 \tilde{L}^{s,y}_t\phi(x)\, &:= \,
\langle A_t x+\tilde{F}^{s, y}_t, D_x \phi(x)\rangle + \tilde{\mathcal{L}}^{s,y}_t \\
&:= \langle A_t x+\tilde{F}^{s, y}_t, D_x \phi(x)\rangle + \int_{\R^d_0}\bigl[\phi(x+B\tilde{\sigma}^{s, y}_t w)-\phi(x) \bigr] \,\nu(dw),
\end{split}
\end{equation}
where, we recall, $\tilde{F}^{s,y}_t:=F(t,\theta_{t,s}(y))$ and $\tilde{\sigma}^{s,y}_t:=\sigma(t,\theta_{t,s}(y))$.\newline
By direct calculation, it is not difficult to check now that for any $(s,x,y)$ in $[0,T]\times \R^{2N}$ it holds that
\begin{equation}\label{eq:differential-eq}
 \left(\partial_t + \tilde{L}^{s,y}_t\right) \tilde{p}^{s,y} (t,s,x,z) \,  = \, 0, \quad (t,z) \in [0,s)\times \R^N.
\end{equation}
However, we carefully point out that some attention is requested to establish the following lemma, which is crucial
to derive which Cauchy-type problem the function $\tilde{G}f := \lim_{\epsilon\to 0}\tilde{G}_\epsilon f$ actually solves. In particular, it is important to highlight that Lemma \ref{convergence_dirac} (Dirac Convergence of frozen density) below cannot be obtained directly from the convergence in law of the frozen process $\tilde{X}_{s}^{s,y,t,x}$ towards the Dirac mass (cf.\ Equation
\eqref{eq:differential-eq}). Indeed, the integration variable $y$ also appears as a freezing parameter which makes the argument more complicated.\newline
The proofs of the following two lemmas is quite involved and technical. For this reason, we decided to postpone them to the Appendix, Section \ref{SEC_TEC_LEMMA_APP}.

\begin{lemma}
\label{convergence_dirac}
Let $(t,x)$ be in $[0,T)\times \R^N$ and $f\colon \R^N\to \R$ a bounded continuous function. Then,
\[
\lim_{\epsilon \to 0}\left| \int_{\R^N} f(y) \tilde{p}^{t+\epsilon,y}(t,t+\epsilon,x,y)\, dy
-f(x) \right| \, = \, 0.
\]
Moreover, the above limit is uniform with respect to $t$ in $[0,T]$.
\end{lemma}

A similar result involving the $L^p_tL^q_x$-norm can also be obtained. For notational simplicity, let us set
\begin{equation}
\label{eq:def_f_epsilon}
I_\epsilon f(t,x) \, := \, \int_{\R^N} f(t+\epsilon,y)\mathds{1}_{[0,T-\epsilon]}(t)
\tilde{p}^{t+\epsilon,y}(t,t+\epsilon,x,y) \, dy
\end{equation}
for any sufficiently regular function $f\colon [0,T]\times \R^N \to \R$.

\begin{lemma}
\label{prop:convergence_LpLq}
Let $p>1$, $q>1$ and $f$ in $C_c^{1,2}([0,T)\times \R^N)$. Then,
\[
\lim_{\epsilon \to 0} \Vert
I_\epsilon f -f \Vert_{L^{p}_tL^q_x} \, = \, 0.
\]
\end{lemma}

We want now to understand which Cauchy-like problem is solved by our frozen Green kernel $\tilde{G}_\epsilon f(t,x)$. For this reason, we introduce for any function $f$ in $C^{1,2}_0([0,T)\times \R^N,\R)$ the following quantity:
\begin{equation}
\label{eq:def_Mepsilon}
\tilde{M}_\epsilon f(t,x)\, := \, \int_{t+\epsilon}^T \int_{\R^N}  \tilde{L}^{s,y}_t \tilde{p}^{s,y}(t,s,x,y)f(s,y) \, dy ds, \quad (t,x) \in [0,T)\times\R^N,
\end{equation}
for some fixed $\epsilon>0$ that is assumed to be small enough.
Then, we can derive from Equation \eqref{eq:differential-eq} and Proposition \ref{prop:Smoothing_effect}
that the following equality holds:
\begin{equation}\label{eq:differential_eq2}
\partial_t \tilde{G}_\epsilon f(t,x)+ \tilde{M}_\epsilon f(t,x) \, = \, -I_\epsilon f(t,x),\quad (t,x)\in [0,T)\times \R^N,
\end{equation}
where we used the same notation in \eqref{eq:def_f_epsilon} for $I_\epsilon f$. We point out that the localization with respect to $\epsilon$ is precisely needed to exploit directly \eqref{eq:differential-eq} and thus, to derive
\eqref{eq:differential_eq2} for any fixed $\epsilon>0$, by usual dominated convergence arguments. In particular, we point out that in the limit case ($\epsilon \to 0$), the smoothness on $f$ is not a sufficient condition to derive the smoothness of $\tilde{G} f$. This is again due to the dependence of the proxy upon the integration variable.
\setcounter{equation}{0}
\section{Well-Posedness of the Martingale Problem}

This section is devoted to the proof of the well-posedness of the martingale problem for $\partial_s+L_s$ with initial condition $(t,x)$, under the assumptions of Theorem
\ref{thm:main_result}.

Since by definition the paths of any solution $\{X_t\}_{t\ge 0}$ of the martingale problem for $\partial_s+L_s$ are c\`adl\`ag, it will be convenient
afterwards to give an alternative definition. We denote by $\mathcal{D}[0,\infty)$ the family of all the c\`adl\`ag paths from $[0,\infty)$ to $\R^N$, equipped
with the ``standard'' Skorokhod topology. For further details, we suggest the interested reader to see \cite{book:Bass11}, \cite{book:Ethier:Kurtz86} or \cite{book:Jacod:Shiryaev87}.\newline
Fixed a starting point $(t,x)$ in $[0,\infty)\times \R^N$, we will say that \textcolor{black}{a} probability measure $\mathbb{P}$ on $\mathcal{D}[0,\infty)$ is a solution of the martingale problem for
$\partial_t+L_t$ starting at $(t,x)$ if the coordinate process $\{y_t\}_{t\ge 0}$ on $\mathcal{D}[0,\infty)$, defined by
\[y_t(\omega) \, = \, \omega(t), \quad \omega \in  \mathcal{D}[0,\infty)\]
is a solution (in the previous sense) of the martingale problem for $\partial_s+L_s$.\newline
Similarly, we will say that uniqueness holds for the martingale problem for $\partial_s+L_s$ with starting point $(t,x)$ if
\[\mathbb{P}\circ y^{-1} \, = \,\tilde{\mathbb{P}}\circ y^{-1},\]
for any two solutions $\mathbb{P}$, $\tilde{\mathbb{P}}$ of the martingale problem for $\partial_s+L_s$ starting at $(t,x)$.

The existence of a solution $\mathbb{P}$ of the martingale problem for $\partial_s+L_s$ can be obtained adapting the proof of Theorem $2.2$ in \cite{Stroock75} exploiting the sublinear structure of the drift $F$ and localization arguments in order to deal with possibly unbounded coefficients.

\begin{prop}[existence]
\label{prop:exist_MP_in_x}
Under the assumptions of Theorem \ref{thm:main_result}, let $(t,x)$ be in $[0,\infty)\times\R^N$. Then, there exists a solution $\mathbb{P}$ of the martingale problem for
$\partial_s+L_s$ starting at $(t,x)$.
\end{prop}

We move to the question of uniqueness for the martingale problem associated with $\partial_s+L_s$. As shown already in the introduction of
Section $3$, the analytical properties on the frozen process $(\tilde{X}^{s,y,t,x}_{u})_{u\in [t,T]}$ we presented there will be the crucial tools for the reasoning in the following section.\newline
We will start proving directly that the Krylov-type estimates \eqref{eq:Krylov_Estimates} holds but first for $p$, $q$ big enough (but finite). It will imply in particular the existence of a density for the canonical process associated with any solution of the martingale problem. As a consequence, the weak well-posedness of SDE \eqref{eq:SDE} under our assumptions can be shown to hold.\newline
Only in a second moment, we will then show that the Krylov estimates holds for \emph{any} $p$, $q$ satisfying condition ($\mathscr{C}$) through a regularization technique. Namely, we regularize the driving noise $Z_t$ by introducing an additional isotropic $\alpha$-stable process depending from a regularizing parameter. Following the previous arguments for the regularized dynamics, we will then prove that the solution process satisfies again the Krylov-type estimates for any $p$, $q$ in the considered range, \emph{uniformly} with respect to the regularizing parameter.\newline
Letting the regularizing parameter go to zero, we will then conclude the proof of Corollary \ref{coroll:Krylov_Estimates}.

\subsection{Proof of Krylov-type Estimates for large enough $p,q$}
The first step in proving the uniqueness of the Martingale problem for $\partial_s+L_s$ is to show that any solution to the martingale problem satisfies the Krylov-like estimates in Equation \eqref{eq:Krylov_Estimates}. To do so, we prove that the difference operator between the genuine generator $L_t$ and a suitable associated perturbation (associated with the frozen generator $\tilde{L}^{s,y}_t$ given in \eqref{eq:def_frozen_generator}) has small  $L^p_tL^q_x$-norm when considering a sufficiently small final horizon $T$. Namely, we introduce the following remainder:
\begin{equation}\label{eq:def_R}
\tilde{R}_\epsilon f(t,x) \, := \, (L_t\tilde{G}_\epsilon f-\tilde{M}_\epsilon f)(t,x) \, = \, \int_{t+\epsilon}^T  \int_{\R^N}(L_t -\tilde{L}^{s,y}_t)\tilde{p}^{s,y}(t,s,x,y) f(s,y) \, dy ds,
\end{equation}
for some $\epsilon$ to be small enough. We recall that $\tilde{ G}_\epsilon f, \tilde{M}_\epsilon f $ and $\tilde p^{s,y}(t,s,x,y)$ were defined in \eqref{eq:def_Green_Kernel}, \eqref{eq:def_Mepsilon} and \eqref{eq:representation_density}, respectively.
\newline
We firstly present a point-wise control for the remainder term $\tilde{R}_\epsilon f$. Importantly, the constant $C$ below does not depend on $\epsilon$, allowing to pass to the limit in Equation \eqref{eq:def_R}. This will be discussed at the end of the present section.

\begin{prop}
\label{prop:control_tildeR}
There exist  $q_0>1$, $p_0>1$ and $C:=C(T,p_0,q_0)$ such that for any $q\ge q_0$, $p\ge p_0$ and any $f$ in $L^p([0,T];L^q(\R^N))$, it holds that
\begin{equation}
\label{eq:control_infty_R}
\Vert\tilde{R}_\epsilon f\Vert_\infty\, \le \, C \Vert f \Vert_{L^p_tL^q_x}.
\end{equation}
\end{prop}
\begin{proof}
We start recalling from \eqref{eq:def_generator}-\eqref{eq:def_frozen_generator} (exploiting also the change of truncation in \eqref{eq:change_of_truncation})) that we can decompose $\tilde{R}_\epsilon f$ in the following way:
\begin{equation}
\label{proof:remainder_decomposition}
\begin{split}
   \tilde{R}_\epsilon f(t,x) \, &= \, \int_{t+\epsilon}^T  \int_{\R^N}(\mathcal{L}_t -\tilde{\mathcal{L}}^{s,y}_t)\tilde{p}^{s,y}(t,s,x,y) f(s,y) \, dy ds \\
   &\qquad\qquad + \int_{t+\epsilon}^T  \int_{\R^N}\langle F(t,x)-\tilde{F}^{s,y}_t,D_x\tilde{p}^{s,y}(t,s,x,y)\rangle f(s,y) \, dy ds \\
   &=: \,  \tilde{R}^0_\epsilon f(t,x)+ \tilde{R}^1_\epsilon f(t,x)
\end{split}
\end{equation}
where the operators $\mathcal{L}_t$ and $\tilde{\mathcal{L}^{s,y}_t}$ have been defined in \eqref{eq:def_generator} and \eqref{eq:def_frozen_generator}, respectively.
 \newline
Since by assumptions, $x_j\to F_i(t,x)$ is $\beta^j$-H\"older continuous, we can control the second term $\tilde{R}^1_\epsilon f$, associated with the difference of the drifts, using Proposition \ref{prop:Smoothing_effect} with $(\tau,\xi)=(s,y)$:
\[
\begin{split}
\bigl|\langle F(t,x)-\tilde{F}^{s,y}_t,D_x&\tilde{p}^{s,y}(t,s,x,y)\rangle \bigr| \, \le \, \sum_{i=1}^n\left|F_i(t,x)-F_i(t,\theta_{t,s}(y))\right|\,  |D_{x_i}\tilde{p}^{s,y}(t,s,x,y)| \\
&\le \, C\sum_{i=1}^n(s-t)^{-\frac{1+\alpha(i-1)}{\alpha}}\frac{\overline{p}(1,\mathbb{T}^{-1}_{s-t}(y-\tilde{m}^{s,y}_{s,t}(x)))}{\det \mathbb{T}_{s-t}}\sum_{j=i}^n\left|(x-\theta_{t,s}(y))_j\right|^{\beta^j} \\
&\le \, C\sum_{i=1}^n\sum_{j=i}^n(s-t)^{{\zeta^j_i}}\left|\mathbb{T}^{-1}_{s-t}(x-\theta_{t,s}(y))\right|^{\beta^j}\frac{\overline{p}(1,\mathbb{T}^{-1}_{s-t}(y-\tilde{m}^{s,y}_{s,t}(x)))}{{\det \mathbb{T}_{s-t}}},
\end{split}
\]
with the following notation at hand:
\[\zeta^j_i \,:= \, -\frac{1+\alpha(i-1)}{\alpha}+\beta^j\frac{1+\alpha(j-1)}{\alpha}.\]
Then, we write with the notations of \eqref{proof:remainder_decomposition} that
\begin{align}
\notag
\left|\tilde{R}^1_\epsilon f(t,x)\right| \, &\le \, C\sum_{i=1}^n\sum_{j=i}^n \int_t^T \int_{\R^N} |f(s,y)|\frac{\overline{p}(1,\mathbb{T}^{-1}_{s-t}(y-\tilde{m}^{s,y}_{s,t}(x)))}{\det \mathbb{T}_{s-t}}\frac{\left|\mathbb{T}^{-1}_{s-t}(x-\theta_{t,s}(y))\right|^{\beta^j}}{(s-t)^{-\zeta^j_i}} \, dyds \\
 &=: \, C\sum_{i=1}^n\sum_{j=i}^n \int_t^T \int_{\R^N} |f(s,y)|\mathcal{I}_{ij}(t,s,x,y) \, dyds.\label{proof:remainder:decomposition_R1}
\end{align}
Then, from the H\"older inequality,
\begin{equation}
\label{proof:control_remainder}
\left|\tilde{R}^1_\epsilon f(t,x)\right| \, \le \, C\Vert f \Vert_{L^p_tL^q_x}\sum_{i=1}^n\sum_{j=i}^n \left(\int_t^T \left(\int_{\R^N} \left[\mathcal{I}_{ij}(t,s,x,y) \right]^{q'} dy\right)^{\frac{p'}{q'}}ds\right)^{\frac{1}{p'}},
\end{equation}
where $q'$ and $p'$ are the conjugate exponents of $q$ and $p$, respectively.\newline
Now, the integrals with respect to $y$ can be easily controlled by  Corollary \ref{coroll:control_power_q_density}. Indeed,
\begin{align}
\label{proof:remainder_control_Jij}
\int_{\R^N} \bigl[\mathcal{I}_{ij}(t,s,x,\,&y) \bigr]^{q'} dy  \\\notag
&\le \, C \left(\frac{(s-t)^{\zeta^j_i}}{\det \mathbb{T}_{s-t}}\right)^{q'} \int_{\R^N}
\left|\mathbb{T}^{-1}_{s-t}(x-\theta_{t,s}(y))\right|^{\beta^jq'}\overline{p}(1,\mathbb{T}^{-1}_{s-t}(y-\tilde{m}^{s,y}_{s,t}(x)))
\, dy.
\end{align}
Choosing $q_0>1$ big enough so that $\beta^j q'<\alpha$ for any $ j$ in $\llbracket 1,n\rrbracket$ and any $q\ge q_0$, we can use Corollary \ref{coroll:Smoothing_effect} to show that
\begin{equation}
\label{proof:remainder_control_Jij2}
\int_{\R^N} \left[\mathcal{I}_{ij}(t,s,x,y) \right]^{q'} dy\, \le \, C  (s-t)^{\zeta^j_iq'}\left(\det \mathbb{T}_{s-t}\right)^{1-q'}.
\end{equation}
Going back to Equation \eqref{proof:control_remainder}, we can thus write that
\[
\left|\tilde{R}^1_\epsilon f(t,x)\right| \, \le \, C\Vert f \Vert_{L^p_tL^q_x} \sum_{i=1}^n\sum_{j=i}^n\left(\int_t^T (s-t)^{\zeta^j_i p'}\left(\det \mathbb{T}_{s-t}\right)^{\frac{p'}{q'}-p'} ds\right)^{\frac{1}{p'}}.
\]
Noticing now that for any $i\le j$ in $\llbracket 1, n\rrbracket$
\begin{equation}
\label{proof:eq:control_zeta}
\zeta^j_i >-1 \, \Leftrightarrow \, -\frac{1+\alpha(i-1)}{\alpha}+\beta^j\frac{1+\alpha(j-1)}{\alpha} >-1 \, \Leftrightarrow \, \beta^j >\frac{1+\alpha(i-2)}{1+\alpha(j-1)},
\end{equation}
we can choose $q_0>1$, $p_0>1$ large enough so that $p'$, $q'$ are sufficiently close to $1$ in order to conclude that
\begin{equation}
\label{proof:remainder:final_control_R^1}
\left|\tilde{R}^1_\epsilon f(t,x)\right| \, \le \, C\Vert f \Vert_{L^p_tL^q_x}.
\end{equation}
We can now focus on the control for the first remainder term $\tilde{R}^0_\epsilon f$. Since clearly $\tilde{R}^0_\epsilon f=0$ if $\sigma(t,x)$ is constant in space, we can assume without loss of generality that $\nu$ is absolutely continuous with respect to the Lebesgue measure on $\R^d$ (cf. assumption [\textbf{AC}]). In particular, we know that it can be decomposed as in \eqref{eq:decomposition_nu_AC}:
\[\nu(dz) \, = \ Q(z)\frac{g(\frac{z}{|z|})}{|z|^{d+\alpha}}dz.\]
Given now a smooth enough function $\phi\colon \R^N\to \R$, we start noticing that
\[
\begin{split}
\mathcal{L}_t\phi(x) \, &= \, \int_{\R^d_0}\left[\phi(x+B\sigma(t,x)z)-\phi(x)\right]\, \nu(dz) \\
 &= \, \int_{\R^d_0}\left[\phi(x+B\sigma(t,x)z)-\phi(x) \right]\, Q(z)g\left(\frac{z}{|z|}\right)\frac{dz}{|z|^{d+\alpha}} \\
 &= \, \int_{\R^d_0}\left[\phi(x+B\tilde{z})-\phi(x) \right] Q(\sigma^{-1}(t,x)\tilde{z})\frac{g\left(\frac{\sigma^{-1}(t,x)\tilde{z}}{|\sigma^{-1}(t,x)\tilde{z}|}\right)}{\det \sigma(t,x)}\frac{d\tilde{z}}{|\sigma^{-1}(t,x)\tilde{z}|^{d+\alpha}},
\end{split}
\]
where we assumed, without loss of generality from \textbf{[UE]}, that $ \det \sigma(t,x)>0$. A similar representation holds for $\tilde{\mathcal{L}}^{s,y}_t\phi(x)$, too.
Now, let us introduce for any $z$ in $\R^d$, the following quantity:
\begin{multline*}
\tilde{H}^{s,y}_{t,x}(z) \\
:= \,  Q(\sigma^{-1}(t,x)z)\frac{g\left(\frac{\sigma^{-1}(t,x)z}{|\sigma^{-1}(t,x)z|}\right)}{\det \sigma(t,x)|\sigma^{-1}(t,x)\frac{z}{|z|}|^{d+\alpha}}- Q((\tilde{\sigma}^{s,y}_t)^{-1}z)\frac{g\left(\frac{(\tilde{\sigma}^{s,y}_t)^{-1}z}{|(\tilde{\sigma}^{s,y}_t)^{-1}z|}\right)}{\det \tilde{\sigma}^{s,y}_t|(\tilde{\sigma}^{s,y}_t)^{-1}\frac{z}{|z|}|^{d+\alpha}},
\end{multline*}
where we have normalized $z$ above in order to make the usual isotropic stable L\'evy measure appear. \newline
Fixed $\eta>0$, local to this section, meant to be small and to be chosen later (and not to be confused with the ellipticity constant in assumption \textbf{[UE]}), we then define
\begin{equation}\label{DEF_ALPHA_ETA}
\alpha_\eta= \alpha/(1-\eta),
\end{equation}
and we decompose the integral in the difference of the generators in the following way:
\[
\begin{split}
 \left(\mathcal{L}_t-\tilde{\mathcal{L}}^{s,y}_t \right)\phi(x) \, &= \, \int_{\R^d_0}\left[\phi(x+Bz)-\phi(x)\right]\tilde{H}^{s,y}_{t,x}(z)\,s \frac{dz}{|z|^{d+\alpha}} \\
 &= \, \int_{\Delta_\eta}\left[\phi(x+Bz)-\phi(x) -\langle D_x\phi(x),Bz\rangle\right]\tilde{H}^{s,y}_{t,x}(z)\, \frac{dz}{|z|^{d+\alpha}} \\
 &\qquad \qquad + \int_{\Delta^c_\eta}\left[\phi(x+Bz)-\phi(x)\right]\tilde{H}^{s,y}_{t,x}(z)\, \frac{dz}{|z|^{d+\alpha}} \\
 &=:\, \sum_{i=1}^2 \left[\Delta_{i}\phi(t,s,\cdot,y)\right](x),
\end{split}
\]
where we have denoted, for simplicity,
\[
\begin{split}
\Delta_\eta \, &:= \, B(0,(s-t)^{\frac{1}{\alpha_\eta}});\\
\Delta^c_\eta \, &:= \, B^c(0,(s-t)^{\frac{1}{\alpha_\eta}}).
\end{split}
\]
We highlight in particular that it is precisely the symmetry of $\nu$ that ensures that the function $\tilde{H}^{s,y}_{t,x}$ is even and that allow us to introduce the odd first order term $\langle D_x\phi(x),Bz\rangle$ in the first integral above on the simmetric space $\Delta_\eta$.\newline
Noticing from Proposition \ref{prop:Decomposition_Process_X} that the frozen ``density'' $\tilde{p}^{s,y}$ is regular enough in $x$, we can now replace $\phi$ in the above decomposition with $\tilde{p}^{s,y}(t,s,\cdot,y)$. Going back to $\tilde{R}^0_\epsilon f$ given in \eqref{proof:remainder_decomposition}, we start rewriting it as
\begin{equation}
\label{proof:remainder:decomposition_R0}
\begin{split}
|\tilde{R}^0_\epsilon f(t,x)| \, &\le \, C\sum_{i=1}^2\int_t^T \int_{\R^N}|f(s,y)|\,\left|\left[\Delta_{i}\tilde{p}^{s,y}(t,s,\cdot,y)\right](x)\right| \, dy ds \\
&=: \, \sum_{i=1}^2\int_t^T \int_{\R^N}|f(s,y)|\mathcal{I}_{0i}(t,s,x,y) \, dy ds.
\end{split}
\end{equation}
As before, we can then apply H\"older inequality to show that
\begin{equation}
\label{proof:control_remainder1}
\left|\tilde{R}^0_\epsilon f(t,x)\right| \, \le \, C\Vert f \Vert_{L^p_tL^q_x}\sum_{i=1}^2 \left(\int_t^T \left(\int_{\R^N} \left[\mathcal{I}_{0i}(t,s,x,y) \right]^{q'} dy\right)^{\frac{p'}{q'}}ds\right)^{\frac{1}{p'}},
\end{equation}
where $q'$ and $p'$ are again the conjugate exponents of $q$ and $p$, respectively.\newline
To control the second term involving $\mathcal{I}_{02}$, we start noticing that
\begin{equation}
\label{proof:control_remainder_H}
    |\tilde{H}^{s,y}_{t,x}(z)| \, \le \, C
\end{equation}
for some constant $C$ independent from the parameters, thanks to assumption [\textbf{UE}] for $\sigma$ and the boundedness of $g$ and $Q$.\newline
Then, we can use Control \eqref{proof:control_remainder_H}, Corollary \ref{coroll:control_power_q_density} and the H\"older inequality to write that
\[
\begin{split}
  |\mathcal{I}_{02}&(t,s,x,y)|^{q'} \, \le \, C\Bigl[\int_{\Delta^c_\eta}|\tilde{p}^{s,y}(t,s,x+Bz,y)-\tilde{p}^{s,y}(t,s,x,y)|\frac{dz}{|z|^{d+\alpha}}\Bigr]^{q'} \\
&\le \, C\left(\int_{\Delta^c_\eta} \frac{dz}{|z|^{d+\alpha}}\right)^{\frac{q'}{q}}\int_{\Delta^c_\eta}|\tilde{p}^{s,y}(t,s,x+Bz,y)-\tilde{p}^{s,y}(t,s,x,y)|^{q'}\frac{dz}{|z|^{d+\alpha}}\\
&\le \, C\frac{(s-t)^{ (\eta-1)\frac{q'}{q}}}{(\det \mathbb{T}_{s-t})^{q'}}\int_{\Delta^c_\eta}\left[\overline{p}(1,\mathbb{T}^{-1}_{s-t}(y-\theta_{s,t}(x+Bz)))+\overline{p}(1,\mathbb{T}^{-1}_{s-t}(y-\theta_{s,t}(x)))\right]\frac{dz}{|z|^{d+\alpha}},
\end{split}
\]
recalling from \eqref{DEF_ALPHA_ETA} that $\alpha_\eta=\alpha/(1-\eta) $ for the last inequality.
The Fubini theorem and the change of variables $\tilde{y}=y-\theta_{s,t}(x+Bz)$ now show that
\begin{equation}
\label{proof:remainder:control_I0j}
\begin{split}
\int_{\R^N} |\mathcal{I}_{02}(t,s,x,y)|^{q'} \, dy
\, &\le \, 2C\frac{(s-t)^{ (\eta-1)\frac{q'}{q}}}{(\det \mathbb{T}_{s-t})^{q'-1}}\int_{B^c(0,(s-t)^{\frac{1}{\alpha_\eta}})}\int_{\R^N}\overline{p}(1,\tilde{y})\, d\tilde{y}\frac{dz}{|z|^{d+\alpha}}
\\
&\le \, C(\det \mathbb{T}_{s-t})^{1-q'}(s-t)^{  (\eta-1)\frac{q'}{q}}\int_{\Delta^c_\eta}\frac{dz}{|z|^{d+\alpha}}
\\
&\le \, C(\det \mathbb{T}_{s-t})^{1-q'}(s-t)^{q' (\eta-1)}.
\end{split}
\end{equation}
Going back to Equation \eqref{proof:control_remainder1}, we can then conclude that
\[
\begin{split}
\int_t^T\left(\int_{\R^N} |\mathcal{I}_{02}(t,s,x,y)|^{q'} \, dy\right)^{\frac{p'}{q'}} ds \,
&\le \, C \int_t^T (\det \mathbb{T}_{s-t})^{-\frac{p'}{q}}(s-t)^{p' (\eta-1)} \, ds  \\
&\le \, C \int_t^T (s-t)^{-p'( 1-\eta+\frac{1}{q}\sum_{i=1}^nd_i\frac{1+\alpha(i-1)}{\alpha})} \, ds,
\end{split}
\]
where in the last step we also exploited \eqref{eq:control_det_T1}.\newline
Assuming now that $\eta<1$ and $p$, $q$ are big enough so that
\[p'( 1-\eta+\frac{1}{q}\sum_{i=1}^nd_i\frac{1+\alpha(i-1)}{\alpha}) \, < \,1,\]
we immediately obtain that
\begin{equation}
\label{proof:control_for_I_02}
\left(\int_t^T\left(\int_{\R^N} |\mathcal{I}_{02}(t,s,x,y)|^{q'} \, dy\right)^{\frac{p'}{q'}} ds\right)^{\frac{1}{p'}} \,
\le \, C_T.
\end{equation}
We can now focus on the integral with respect to $y$ of the first term $\mathcal{I}_{01}$ in Equation \eqref{proof:control_remainder1}. Using the Lipschitz continuity of $Q$ in a neighborhood of zero and the H\"older regularity of the diffusion matrix $\sigma$, it is not difficult to check that
\[ |\tilde{H}^{s,y}_{t,x}(z)| \, \le \, C\sum_{j=1}^n\left|(x-\theta_{t,s}(y))_j\right|^{ \beta^1}.\]
Thanks to the above estimate, we exploit a Taylor expansion on the density $\tilde{p}^{s,y}$ and Proposition \ref{prop:Smoothing_effect} with $k=2$ and $(\tau,\xi)=(s,y)$ to show that
\begin{align}
\Theta(t,s,x,y,z)\, &:= \,  \Bigl|\bigl[\tilde{p}^{s,y}(t,s,x+ Bz,y)-\tilde{p}^{s,y}(t,s,x,y)
-\langle D_x\tilde{p}^{s,y}(t,s,x,y),Bz\rangle\bigr]\tilde{H}^{s
,y}_{t,x}(z) \Bigr|\notag \\
&\le \, \sum_{j=1}^n\int_0^1\left|(x-\theta_{t,s}(y))_j\right|^{ \beta^1}\,  |D^2_{x_1}\tilde{p}^{s,y}(t,s,x+\lambda Bz,y)| |z|^2 \,
d\lambda \notag\\
&\le \, \frac{C}{\det
\mathbb{T}_{s-t}}\int_0^1|z|^2\frac{\overline{p}(1,\mathbb{T}^{-1}_{s-t}(y-\tilde{m}^{s,y}_{s,t}(x+\lambda Bz)))}{(s-t)^{\frac{2}{\alpha}}}\notag\\
&\qquad\qquad\qquad\qquad\times\Bigl[\sum_{j=1}^n\left|(x+\lambda Bz-\theta_{t,s}(y))_j\right|^{\beta^1}+|\lambda B z|^{\beta^1}\Bigr] \, d\lambda \notag \\
\label{proof:remainder:control_H2}
&\le \, \frac{C}{\det
\mathbb{T}_{s-t}}\int_0^1|z|^2\overline{p}(1,\mathbb{T}^{-1}_{s-t}(y-\tilde{m}^{s,y}_{s,t}(x+\lambda Bz)))\\
&\qquad\qquad\qquad\qquad\times\Bigl[\sum_{j=1}^n\frac{\left|\mathbb{T}^{-1
}_{s-t}(x+\lambda Bz-\theta_{t,s}(y))\right|^{\beta^1}}{(s-t)^{{\zeta^j_0}}}+\frac{|z_1|^{\beta^1}}{(s-t)^{\frac{2}{\alpha}}}\Bigr]
 \, d\lambda, \notag
\end{align}
where, similarly to above, we have denoted:
\begin{equation}
\label{proof:def_eta_0}
\zeta^j_0 \,:= \, \frac{2}{\alpha}-\beta^1\frac{1+\alpha(j-1)}{\alpha}.
\end{equation}
It then follows from the H\"older inequality and Corollary \ref{coroll:control_power_q_density} that
\[
\begin{split}
|\mathcal{I}_{01}&(t,s,x,y)|^{q'} \,\le \, \left[\int_{\Delta_\eta}\Theta(t,s, x,y,z)\,\frac{dz}{|z|^{d+\alpha}}\right]^{q'} \\
&\le \, \frac{C}{(\det
\mathbb{T}_{s-t})^{q'}}\left(\int_{\Delta_\eta}1\, dz\right)^{\frac{q'}{q}}\int_0^1\int_{\Delta_\eta}\overline{p}(1,\mathbb{T}^{-1}_{s-t}(y-\tilde{m}^{s,y}_{s,t}(x+\lambda Bz)))\\
&\qquad\qquad\quad\quad \qquad \times \Bigl[\sum_{j=1}^n\frac{\left|\mathbb{T}^{-1
}_{s-t}(x+\lambda Bz-\theta_{t,s}(y))\right|^{\beta^1}}{(s-t)^{{\zeta^j_0}}}+\frac{|z_1|^{\beta^1}}{(s-t)^{\frac{2}{\alpha}}}\Bigr]^{q'}
 \, \frac{dz}{|z|^{q'(d+\alpha-2)}}d\lambda.
\end{split}
\]
If we now add the integral with respect to $y$, Fubini Theorem readily implies that
\[
\begin{split}
\int_{\R^N}|\mathcal{I}_{01}&(t,s,x,y)|^{q'} \, dy \\
&\le \,C\frac{(s-t)^{\frac{d}{\alpha_\eta}(q'-1)}}{(\det
\mathbb{T}_{s-t})^{q'}} \int_0^1\int_{\Delta_\eta}\int_{\R^N}\overline{p}(1,\mathbb{T}^{-1}_{s-t}(y-\tilde{m}^{s,y}_{s,t}(x+\lambda Bz)))\\
&\qquad \qquad \times \Bigl[\sum_{j=1}^n\frac{\left|\mathbb{T}^{-1
}_{s-t}(x+\lambda Bz-\theta_{t,s}(y))\right|^{\beta^1q'}}{(s-t)^{{\zeta^j_0q'}}}+\frac{|z_1|^{\beta^1q'}}{(s-t)^{\frac{2}{\alpha}q'}}\Bigr]
 \, dy\frac{dz}{|z|^{q'(d+\alpha-2)}}d\lambda.
\end{split}
\]
If we assume to have taken $q'$ close enough to $1$ so that $\beta^1q'<\alpha$, we can use Corollary \ref{coroll:Smoothing_effect} to show that
\[
\begin{split}
\int_{\R^N}|\mathcal{I}_{01}(t,s,\,&x,y)|^{q'} \, dy \\
&\le \, C\frac{(s-t)^{\frac{d}{\alpha_\eta}(q'-1)}}{(\det
\mathbb{T}_{s-t})^{q'-1}}\int_{B(0,(s-t)^{\frac{1}{\alpha_\eta}})}\left[\sum_{j=1}^n\frac{1}{(s-t)^{{q'\zeta^j_0}}}+\frac{|z_1|^{q'\beta^1}}{(s-t)^{q'\frac{2}{\alpha}}}\right]
 \,\frac{dz}{|z|^{q'(d+\alpha-2)}}\\
&\le \, C\frac{(s-t)^{\frac{d}{\alpha_\eta}(q'-1)}}{(\det\mathbb{T}_{s-t})^{q'-1}}\int_0^{(s-t)^{\frac{1}{\alpha_\eta}}}\left[\sum_{j=1}^n\frac{r^{d-1-(d+\alpha-2)q'}}{(s-t)^{{q'\zeta^j_0}}}+\frac{r^{d-1-(d+\alpha-2-\beta^1)q'}}{(s-t)^{\frac{2}{\alpha}q'}}\right]
 \,dr.
\end{split}
\]
Similarly, if $q$ is big enough (so that $q'$ is close to $1$), it holds that
\[d-1-q'(d+\alpha-2) > -1 \, \Leftrightarrow \, q' < \frac{d}{d+\alpha-2}\]
 and we can integrate with respect to $r$:
\begin{align}\notag
\int_{\R^N}|\mathcal{I}_{01}(t,\,&s,x,y)|^{q'} \, dy \,
\le \, C\frac{(s-t)^{\frac{d}{\alpha_\eta}(q'-1)}}{(\det\mathbb{T}_{s-t})^{q'-1}}\left[\sum_{j=1}^n\frac{r^{d-q'(d+\alpha-2)}}{(s-t)^{{q'\zeta^j_0}}}+\frac{r^{d-q'(d+\alpha-2-\beta^1)}}{(s-t)^{q'\frac{2}{\alpha}}}\right]\Bigg{|}_0^{(s-t)^{\frac{1}{\alpha_\eta}}} \\
&\le \, C (\det\mathbb{T}_{s-t})^{1-q'}(s-t)^{\frac{q'}{\alpha_\eta}(2-\alpha)}\left[\sum_{j=1}^n(s-t)^{-q'\zeta^j_0}+(s-t)^{q'(\frac{\beta^1}{\alpha_\eta}-\frac{2}{\alpha})}\right].
\label{proof:remainder:control_I0j2}
\end{align}
Hence, it follows from Equation \eqref{eq:control_det_T1} that
\[
\begin{split}
\int_t^T\Bigl(\int_{\R^N}|&\mathcal{I}_{01}(t,s,x,y)|^{q'} \, dy\Bigr)^{\frac{p'}{q'}}ds \\
&\le \, C \int_t^T(\det\mathbb{T}_{s-t})^{-\frac{p'}{q}}(s-t)^{\frac{p'}{\alpha_\eta}(2-\alpha)}\left[\sum_{j=1}^n(s-t)^{-p'\zeta^j_0}+(s-t)^{p'(\frac{\beta^1}{\alpha_\eta}-\frac{2}{\alpha})}\right]\, ds\\
&\le \, C \int_t^T(s-t)^{p'\left( \frac{(2-\alpha)}{\alpha_\eta}-\frac{1}{q}\sum_{i=1}^nd_i\frac{1+\alpha(i-1)}{\alpha}\right)}\left[\sum_{j=1}^n(s-t)^{-p'\zeta^j_0}+(s-t)^{p'(\frac{\beta^1}{\alpha_\eta}-\frac{2}{\alpha})}\right]\, ds
\end{split}
\]
To conclude, we need to show that the two terms above are integrable with respect to $s$. Namely,
\begin{align*}
p'\left(\frac{(2-\alpha)}{\alpha_\eta}-\frac{1}{q}\sum_{i=1}^nd_i\frac{1+\alpha(i-1)}{\alpha}-\zeta^j_0\right) \, &> \, -1, \quad \forall\, j \in \llbracket 1,n \rrbracket;\\
p'\left(\frac{(2-\alpha)}{\alpha_\eta}-\frac{1}{q}\sum_{i=1}^nd_i\frac{1+\alpha(i-1)}{\alpha}+\frac{\beta^1}{\alpha_\eta}-\frac{2}{\alpha}\right) \, &> \, -1.
\end{align*}
Recalling again that we can choose $p$, $q$ big enough as we want, so that Equation \eqref{eq:control_det_T2} holds, it is now sufficient to take $\eta$ in $(0,1)$ in order to have:
\begin{align}
\frac{(2-\alpha)}{\alpha_\eta}-\zeta^j_0 \, = \, \frac{(2-\alpha)}{\alpha_\eta}-\frac{2}{\alpha}+\beta^1\frac{1+\alpha(j-1)}{\alpha} \, &> \, -1,\quad \forall\, j \in \llbracket 1,n \rrbracket; \label{proof:condition1}\\
\frac{(2-\alpha)}{\alpha_\eta}+\frac{\beta^1}{\alpha_\eta}-\frac{2}{\alpha} \, &> \, -1.\label{proof:condition2}
\end{align}
By direct calculations, recalling from \eqref{DEF_ALPHA_ETA} that $\alpha_\eta=\alpha/(1-\eta) $, we now notice that Conditions \eqref{proof:condition1}-\eqref{proof:condition2} can be rewritten as follows
\begin{align*}
\eta \, &< \, \frac{\beta^1(1+\alpha(j-1)}{2-\alpha}, \quad \forall\, j \in \llbracket 1,n \rrbracket;\\
\eta \, &< \, \frac{\beta^1}{2+\beta^1-\alpha}.
\end{align*}
Choosing $\epsilon>0$ so that the above conditions holds, we have that
\begin{equation}
\label{proof:control_for_I_01}
\left(\int_t^T\Bigl(\int_{\R^N}|\mathcal{I}_{01}(t,s,x,y)|^{q'} \, dy\Bigr)^{\frac{p'}{q'}}ds\right)^{\frac{1}{p'}} \, \le \, C_T.
\end{equation}
Going back to Equation \eqref{proof:control_remainder1}, we   use Controls \eqref{proof:control_for_I_02}-\eqref{proof:control_for_I_01} to write that
\begin{equation}
\label{proof:remainder:final_control_R^0}
\left|\tilde{R}^0_\epsilon f(t,x)\right| \, \le \, C\Vert f \Vert_{L^p_tL^q_x}.
\end{equation}
Exploiting Controls \eqref{proof:remainder:final_control_R^1} and \eqref{proof:remainder:final_control_R^0} in Equation \eqref{proof:remainder_decomposition}, we have concluded our proof.
\end{proof}

Let us fix now a function $f$ in $ C_c^{1,2}([0,T)\times \R^{N})$. The first step of our method consists in applying the It\^o formula on the Green kernel $\tilde{G}_\epsilon f$ and the process $\{X^{t,x}_s\}_{s\in[t,T]}$, solution of the martingale problem with starting point $(t,x)$:
\[
\mathbb{E}\left[\tilde{G}_\epsilon f(t,x)+ \int_t^{T} (\partial_s+L_s)\tilde{G}_\epsilon f(s,X^{t,x}_s)   ds\right] \, =\, 0.
\]
We then exploit Equation \eqref{eq:differential_eq2} to write that
\[
\tilde{G}_\epsilon f (t,x)- \mathbb{E}\left[\int_t^{T} I_\epsilon f(s,X^{t,x}_s)\,  ds \right]+\mathbb{E}\left[\int_{t}^T\left[L_s\tilde{G}_\epsilon f- \tilde{M}_{\epsilon}f\right] (s,X^{t,x}_s) \, ds\right] \, = \, 0.
\]
Thus, it holds that
\begin{equation}
\label{eq:NUMBER_PREAL_REG_DENS}
\mathbb{E}\Bigl[\int_t^{T} I_\epsilon f(s,X^{t,x}_s) \, ds \Bigr] \,
= \, \tilde{G}_\epsilon f
(t,x)+\mathbb{E}\Bigl[\int_t^T\tilde{R}_\epsilon f(s,X^{t,x}_s)\, ds\Bigr].
\end{equation}

Thanks to Proposition \ref{prop:pointwise_green_kernel}, we know that there exists $C(T):=C(T) \underset{T\rightarrow 0}{\longrightarrow} 0$ such that
\begin{equation}
\label{eq:control_infty_Green}
\Vert \tilde{G}_\epsilon f\Vert_\infty \, \le \,  C \Vert f\Vert_{L^p_tL^q_x}.
\end{equation}

Let us assume for now that $p,q$ are large enough so that the control \eqref{eq:control_infty_R} of Lemma \ref{prop:control_tildeR} (pointwise control of the remainder) holds. From \textcolor{black}{Equations} \eqref{eq:NUMBER_PREAL_REG_DENS}, \eqref{eq:control_infty_Green} and \eqref{eq:control_infty_R}, we readily get that
\[
\left|\mathbb{E}\left[\int_t^T I_\epsilon f(X_s^{t,x})\,  ds \right]\right|\, \le \,  C\Vert f\Vert_{L^p_tL^q_x}.
\]

Letting $\epsilon $ go to zero,  we  thus derive that any solution $\{X^{t,x}_s\}_{s\in [t,T]}$ of the martingale problem for $\partial_s+L_s$ with initial condition $(t,x)$ satisfies
\[
\left|\mathbb{E}\left[\int_{t}^T f(s,X_s^{t,x})\, ds\right]\right|\,\le \, C\Vert f\Vert_{L^p_tL^q_x},\]
for any $f$ in  $f$ in $ C_c^{1,2}([0,T)\times \R^{N})$. \textcolor{black}{Above, we have exploited Lemma \ref{convergence_dirac} for the integral in space and the bounded convergence Theorem  for that in time.}\newline
To show the result for a general $f$ in $L^p\left(0,T;L^q(\R^N)\right)$, we now use a density argument and the Fatou Lemma. Indeed, let $\{f_n\}_{n\in \N}$ a sequence of functions in $C_c^{1,2}([0,T)\times \R^{N})$ such that $\Vert f_n -f \Vert_{L^p_tL^q_x}\to 0$.
We then have that:
\begin{equation}\label{eq:Krilov_partial_estimates}
\begin{split}
\left|\mathbb{E}\left[\int_{t}^T f(s,X_s^{t,x})\, ds\right]\right|\,&\le \, \left|\mathbb{E}\left[\int_{t}^T \liminf_n f_n(s,X_s^{t,x})\, ds\right]\right| \\
&\le \, \liminf_n\left|\mathbb{E}\left[\int_{t}^T  f_n(s,X_s^{t,x})\, ds\right]\right| \\
&\le \, C\liminf_n \Vert f_n\Vert_{L^p_tL^q_x} \\
&=  \,C \Vert f\Vert_{L^p_tL^q_x}.
\end{split}
\end{equation}
This is precisely the Estimate \eqref{eq:Krylov_Estimates} in Corollary \ref{coroll:Krylov_Estimates}, provided that $p,q$ are large enough.

\subsection{Proof of Uniqueness for martingale problem}

We start by noticing that a control, similar to \eqref{prop:control_tildeR}, can be obtained in $L^p_tL^q_x$-norms, too.
In particular, we point out that Equation \eqref{eq:control_LpLq_R} below implies  that the operator $I-\tilde{R}_\epsilon$ is invertible in $L^p\left(0,T;L^q(\R^N)\right)$, provided $T$ is small enough. From Lemma \ref{prop:convergence_LpLq}, the same holds for $I_\epsilon-\tilde{R}_\epsilon$.

\begin{prop}
\label{prop:control_tildeR_LpLq}
Let $q>1$, $p>1 $ be such that Condition ($\mathscr{C}$) holds. Then, there exists $C:=C(T,p,q)>0$ such that for any $f$ in $L^p\left(0,T;L^q(\R^N)\right)$,
\begin{equation}
\label{eq:control_LpLq_R}
\Vert \tilde{R}_\epsilon f\Vert_{L^p_tL^q_x}\, \le \, C \Vert f \Vert_{L^p_tL^q_x}.
\end{equation}
In particular, it holds that $\lim_{T\to 0}C(T,p,q)=0$.
\end{prop}
\begin{proof}
We are going to keep the same notations used in the previous proof. In particular, we recall the following decomposition
\[\tilde{R}_\epsilon f(t,x) \, = \, \tilde{R}^0_\epsilon f(t,x)+\tilde{R}^1_\epsilon f(t,x), \]
given in Equation \eqref{proof:remainder_decomposition}. \newline
In order to control the second term $\tilde{R}^1_\epsilon f$ in $L^p_tL^q_x$-norm, we start from Equation \eqref{proof:remainder:decomposition_R1} to write that
\[
    \Vert \tilde{R}^1_\epsilon f(t,\cdot) \Vert_{L^q_x} \, \le \, C\sum_{i=1}^n\sum_{j=i}^n \int_t^T \Bigl{\Vert}\int_{\R^N} |f(s,y)|\mathcal{I}_{ij}(t,s,\cdot,y) \, dy\Bigr{\Vert}_{L^q_x}ds.
\]
The Young inequality now implies that
\begin{align*}
\Bigl{\Vert}\int_{\R^N} |f(s,y)|\mathcal{I}_{ij}(t,s,\cdot,y) \, dy\Bigr{\Vert}_{L^q_x}^q&\, = \, \int_{\R^N} \Big|\int_{\R^N} |f(s,y)| \mathcal{I}_{ij}(t,s,x,y)\, dy\Big|^q \, dx\\
&\le\, \int_{\R^n}
\!\left(\int_{\R^N} |f(s,y)|^q \mathcal{I}_{ij}(t,s,x,y)\, dy\right) \|\mathcal{I}_{ij}(t,s,x,\cdot)\|_{L^1}^{q/q'}dx\\
&\le\,  C(s-t)^{\zeta_i^j q/q'} \int_{\R^N} |f(s,y)|^q \left(\int_{\R^N} \mathcal{I}_{ij}(t,s,x,y) \,dx\right)dy
\end{align*}
using Control \eqref{proof:remainder_control_Jij2} and the Fubini Theorem for the last inequality. From \eqref{proof:remainder:decomposition_R1}, \eqref{proof:remainder_control_Jij} and the correspondence \eqref{eq:identification_theta_m1} which gives $y-\tilde m_{s,t}^{s,y}(x)=\theta_{t,s}(y)-x $ it is plain to derive that:
$$\int_{\R^N}dx \mathcal{I}_{ij}(t,s,x,y)\le C (s-t)^{\zeta_i^j }.$$
Thus,
\[\Vert \tilde{R}^1_\epsilon f(t,\cdot) \Vert_{L^q_x} \,
\le \, C\sum_{i=1}^n\sum_{j=i}^n\int_t^T (s-t)^{\zeta^j_i}  \Vert f(t,\cdot) \Vert_{L^q_x} \, ds,\]
where, in the last step, we exploited Equation \eqref{proof:remainder_control_Jij2} with $q'=1$, recalling that $\beta^j<1<\alpha$.\newline
We can then use the above control to write that
\[
\begin{split}
\Vert \tilde{R}^1_\epsilon f\Vert^p_{L^p_tL^q_x} \, &\le \, \sum_{i=1}^n\sum_{j=i}^n \int_0^T \Vert \tilde{R}^1_\epsilon f(t,\cdot) \Vert_{L^q_x}^p \, dt \\
&\le \,C \sum_{i=1}^n\sum_{j=i}^n\int_0^T \Vert f(t,\cdot) \Vert^{p}_{L^q_x} \left( \int_t^T (s-t)^{\zeta^j_i} \, ds  \right)^p dt \\
&\le \, C_T \sum_{i=1}^n\sum_{j=i}^n\int_0^T \Vert f(t,\cdot) \Vert^p_{L^q_x} dt\le C_T\|f\|_{L^p_tL^q_x}^p,
\end{split}
\]
where $C_T:=C(p',q',T)$ denotes a positive constant that tends to zero if $T$ goes to zero (recall indeed from \eqref{proof:eq:control_zeta} that $\zeta^j_i>-1$)
. \newline
The control for $\tilde{R}^0_\epsilon f$ can be obtained following the same arguments above, exploiting Equation \eqref{proof:remainder:decomposition_R0} instead of \eqref{proof:remainder:decomposition_R1} and Equations \eqref{proof:remainder:control_I0j}-\eqref{proof:remainder:control_I0j2} with $q'=1$ for the controls of $\Vert \mathcal{I}_{0j}(t,s,x,\cdot) \Vert_{L^1_x}$.
\end{proof}

Thanks to Estimates \eqref{eq:Krilov_partial_estimates}, we know that the process $\{X^{t,x}_s\}_{s\in[t,T]}$ has a density we will denote by $p(t,s,x,y)$. From Equation
\eqref{eq:NUMBER_PREAL_REG_DENS} it now follows that
\begin{align}
\notag
\tilde{G}_\epsilon f(t,x) \, &= \, \mathbb{E}\left[\int_t^{T} I_\epsilon f(s,X_s^{t,x}) \, ds\right]-\mathbb{E}\left[\int_{t}^T \tilde{R}_\epsilon f(s,X_s^{t,x}) \, ds\right]\\\notag
&= \,  \int_t^{T} \int_{\R^N}I_\epsilon f (s,y) p(t,s,x,y) \, dy  ds- \int_{t}^T\int_{\R^N}\tilde{R}_\epsilon f(s,y)
p(t,s,x,y) \, dy  ds\\
&= \, \int_t^{T} \int_{\R^N}(I_\epsilon-\tilde{R}_\epsilon)f(s,y) p(t,s,x,y) \, dy  ds.
\label{REP_PRESQUE_FINI}
\end{align}
Then, Proposition \ref{prop:pointwise_green_kernel}, Lemma \ref{prop:convergence_LpLq} (with an additional approximation argument) and Control \eqref{eq:control_LpLq_R} imply that both sides of the above control are bounded in the $L^p_tL^q_x$-norm, uniformly in $\epsilon>0$. Thus, we can conclude that Equation \eqref{REP_PRESQUE_FINI} holds for any $f$ in $L^p\left(0,T;L^q(\R^N)\right)$. We then conclude from Lemma \ref{prop:control_tildeR} (pointwise control of the remainder) that letting $\epsilon$ go to zero, it holds that
\[
\mathbb{E}\left[\int_t^{T} f(s,X_s^{t,x}) \, ds\right]\, = \, \tilde{G}\circ (I-\tilde{R})^{-1}f(t,x),
\]
which gives uniqueness if the final time $T$ is small enough. Global well-posedness is again derived from a chaining argument in time.

\subsection{Proof of Krylov-type Estimates under condition ($\mathscr{C}$)}
To complete the proof of Corollary \ref{coroll:Krylov_Estimates}, it remains to derive the Krylov estimates \eqref{eq:Krylov_Estimates} under Condition ($\mathscr{C}$) and not only for $p,q$ large enough.\newline
Fixed a parameter $\delta>0$ meant to be small, we consider a ``mollified'' version of the solution process $X^{t,x}_s$, given by
\begin{equation}
\label{eq:decomposition_final_Krylov}
\overline{X}^{t,x,\delta}_s\, := \, X^{t,x}_s+\delta\mathbb{M}_{s-t} \overline{Z}_{s-t},
\end{equation}
where $\{\overline{Z}_s\}_{s\ge0}$ is an isotropic $\alpha$-stable process on $\R^N$. \newline
Let us denote now by $p^\delta(t,s,x,\cdot)$ the density associated with the random variable $\overline{X}^{t,s,\delta}_{s}$.
We notice that Equation \eqref{eq:decomposition_final_Krylov} implies in particular that
\[p^{\delta}(t,s,x,y) \, = \, \left[p(t,s,x,\cdot)\ast q^\delta(s-t,\cdot)\right](y),\]
where $q^\delta(t,\cdot)$ is the density of the process $\delta \mathbb{M}_t\overline{Z}_t$ and thus, under the integrability condition $(\mathscr{C})$ and thanks to the Young inequality, the quantity $\Vert p^\delta\Vert_{L^{p'}_tL^{q'}_x}$, where $p',q'$ are the conjugate exponents of $p,q$, respectively, is finite (possibly explosive with $\delta $).
The point is now to reproduce the previous perturbative analysis in order to prove that the controls on $\Vert p^\delta\Vert_{L^{p'}_tL^{q'}_x}$ actually do no depend on $\delta$.

For this reason, we introduce the mollified ``frozen'' process $\tilde{X}^{s,y,t,x,\delta}_{s}$ along the flow $\theta_{t,s}(y)$ as
\begin{equation}
\label{eq:decomposition_delta_frozen}
\tilde{X}^{s,y,t,x,\delta}_{s}\, := \, \tilde{X}^{s,y,t,x}_{s}+\delta\mathbb{M}_{s-t} \overline{Z}_{s-t}.
\end{equation}

Following the same arguments presented in Propositions \ref{prop:Decomposition_Process_X} and \ref{prop:Smoothing_effect}, it is now possible to show that the process $\tilde{X}^{s,y,t,x,\delta}_{s}$ admits a density $\tilde{p}^{s,y,\delta}(t,s,x,y)$ and that it  enjoys a  multi-scale bound similar to \eqref{eq:smoothing_effect_frozen_y}. Namely,
\begin{prop}
\label{coroll:Smoothing_effect_delta}
There exists a positive constant $C:=C(N,\alpha)$ such that for any $k$ in $\llbracket 0,2 \rrbracket$, any $i$ in $\llbracket 1,n \rrbracket$, any $t< s$ in $[0,T]$ and any $x,y$ in $\R^N$,
  \begin{equation}
  \label{eq:smoothing_effect_frozen_delta}
      \vert D^k_{x_i} \tilde{p}^{s,y,\delta}(t,s,x,y) \vert \, \le \,C\frac{\left((s-t)(1+\delta)\right)^{-k\frac{1+\alpha(i-1)}{\alpha}}}{\det \mathbb{T}_{(s-t)(1+\delta)}}\overline{p}\left(1,\mathbb{T}^{-1}_{(s-t)(1+\delta)}(y-\theta_{s,t}(x))\right).
  \end{equation}
\end{prop}
A sketch of proof for the above Proposition has been briefly presented in the Appendix section. \textcolor{black}{Importantly, we highlight that the constant $C$ appearing in \eqref{eq:smoothing_effect_frozen_delta} is independent from the ``smoothing'' parameter $\delta$}.\newline
Then, the same arguments leading to \eqref{eq:NUMBER_PREAL_REG_DENS} can be applied here to show that
\begin{equation}
\label{eq:NUMBER_PREAL_REG_DENS1}
\mathbb E\Bigl[\int_t^{T} I_\epsilon f(s,\overline{X}^{t,x,\delta}_s)\, ds\Bigr]  \, = \, \tilde{G}^\delta_{\epsilon} f
(t,x)+\mathbb{E}\Bigl[\int_{t}^T\int_{\R^N}\bigl(L^\delta_t-\tilde{L}^{s,y,\delta}_t\bigr) \tilde{G}^\delta_{\epsilon} f(s,\overline{X}^{t,x,\delta}_s)\, ds\Bigr],
\end{equation}
where $\tilde{G}^\delta_{\epsilon}$ and $\tilde{\mathcal{L}}^{s,y,\delta}$ are the frozen Green kernel and the frozen infinitesimal generator associated with the process $\tilde{X}^{s,y,t,x,\delta}_{s}$, respectively (cf.\ Equations \eqref{eq:def_Green_Kernel} and  \eqref{eq:def_frozen_generator}).
In particular, we point out that the pointwise bound \eqref{eq:control_infty_Green} on the Green kernel and the controls of Proposition \ref{prop:control_tildeR} (pointwise control of the remainder) are uniform with respect to the additional parameter $\delta$, thanks to Proposition \ref{coroll:Smoothing_effect_delta}.\newline
From Equation \eqref{eq:NUMBER_PREAL_REG_DENS1} and Proposition \ref{prop:control_tildeR_LpLq} ($L_t^pL_x^q $ control of the remainder)
we can then deduce that
\[
\left|\int_{t}^{T}\int_{\R^N} I_\epsilon f(s,y) p^\delta(t,s,x,y)\,  dy ds\right| \, \le \, C_T \left(1+\Vert p^\delta\Vert_{L^{p'}_tL^{q'}_x}\right)\Vert f \Vert_{L^p_tL^q_x}.
\]
\newline
From the Riesz representation theorem and the above inequality, we then deduce that  $\Vert p^\delta\Vert_{L^{p'}_tL^{q'}_x}\le C_T$, for $T$ small enough and uniformly in $\delta$. Hence,
\[\left|\int_t^{T} \int_{\R^N} I_\epsilon f (s,y) p^\delta(t,s,x,y)\,   dy ds\right| \, = \, \left|\int_t^{T}\mathbb{E} [I_\epsilon f (s,X_s^{t,x}+\delta\overline{Z}_s)]\,  ds\right| \le C_T\Vert I_\epsilon f\Vert_{L^p_tL^q_x}.\]
The Krylov-type estimate \eqref{eq:Krylov_Estimates} can be then derived exploiting the dominated convergence theorem and Lemma \ref{convergence_dirac} (Dirac Convergence of frozen density), letting firstly $\epsilon$ and then $\delta $ go to zero. We have thus concluded the proof of Corollary \ref{coroll:Krylov_Estimates}.

\setcounter{equation}{0}
\section{A Counter-example to Uniqueness}

In this section, we present a counter-example to the uniqueness in law for the equation \eqref{eq:SDE} when the H\"older regularity in
space of the coefficients is low enough. In particular, we show here the almost sharpness of the thresholds appearing in Theorem
\ref{thm:main_result} for diagonal perturbations, proving also Theorem \ref{thm:counterexample}. In order to test the threshold associated with the critical
H\"older exponent for the $i$-th component of the drift $F$ with respect to the variables $x_j$, we adapt the \emph{ad hoc} Peano
example constructed in \cite{Chaudru:Menozzi17} to our L\'evy framework.\newline
Let us briefly recall it. It is well-known that the following deterministic equation
\begin{equation}
\begin{cases}
\label{eq:deterministic_Peano_example}
dy_t \, = \, \text{sgn}(y_t)\vert y_t\vert^\beta dt, \quad t\ge 0, \\
y_0 \, = \, 0,
\end{cases}
\end{equation}
for some $\beta$ in $(0,1)$, is ill-posed since it admits an infinite number of solutions of the form
\[y_t \, = \, \pm c(t-t_0)^{1/(1-\beta)}\mathds{1}_{[t_0,\infty)}(t), \quad \text{for some $t_0$ in }[0,+\infty).\]
Nevertheless, Bafico and Baldi in \cite{Bafico:Baldi81} proved that the associated SDE, obtained by adding a Brownian Motion
$\{W_t\}_{t\ge 0}$ to the dynamics:
\[\begin{cases}
dX_t \, = \, \text{sgn}(X_t)\vert X_t\vert^\beta dt+\epsilon dW_t, \quad t\ge 0 \\
X_0 \, = \, 0,
\end{cases}\]
is well-posed for any $\epsilon>0$ in a strong (probabilistic) sense. Furthermore, they showed that, letting $\epsilon$ goes to zero, the limit law
concentrates around the two extremal solutions $\pm ct^{1/(1-\beta)}$ of the deterministic equation \eqref{eq:deterministic_Peano_example}, thus providing a \textcolor{black}{selection}
``criterion'' between the infinite deterministic solutions. \newline
In a subsequent article \cite{Delarue:Flandoli14}, Delarue and Flandoli highlighted the hidden dynamical mechanism behind this counter-intuitive behaviour.
Heuristically, this \emph{regularization by noise} happens since, at least in a small time interval, the mean fluctuations of the Brownian noise are stronger than
the irregularity of the deterministic drift. Indeed, they showed that before some transition time $t_\epsilon$, the dominating noise pushes the solution to leave
the drift singularity at $0$, while afterwards, the deterministic part of the system prevails, constraining the (stochastic) solution to fluctuate around one of
the extremal deterministic solutions, given by $\pm ct^{1/(1-\beta)}$. \newline
More quantitatively, we can compare the fluctuations of the noise, say of order $\gamma>0$ with the fluctuations of the deterministic extremal solutions, giving
that
\[t^\gamma \, > \, t^{1/(1-\beta)}.\]
Since it should happen in small times, we then obtain that
\[\beta \, > \, 1-\frac{1}{\gamma},\]
should be the heuristic relation that guarantees the noise dominates in short time.
Clearly, the above inequality holds for any $\beta$ in $(0,1)$ in the Brownian case ($\gamma=1/2$), which would actually give $\beta>-1 $. We can refer to \cite{Delarue:Diel16} which is the closest  work to this threshold since the authors manage to reach $-2/3^+$.

Let us now consider a simplified version of SDE
\eqref{eq:SDE} where $n=N$, $d_k=d=1$ for any $k$ in $\llbracket 1,n\rrbracket$, the matrix $A$
in $\R^n\otimes \R^n$ given by
\[
A \, := \, \begin{pmatrix}
               0& \dots         & \dots         & \dots     & 0 \\
              1       & 0 & \dots         & \dots     & 0\\
               0 & 1       & \ddots & \ddots     & \vdots \\
               \vdots        & \ddots        & \ddots        & \ddots    & \vdots        \\
               0 & \dots         & 0 & 1 & 0
             \end{pmatrix}.
\]
 and $\{Z_t\}_{t\ge 0}$ is a symmetric, \textcolor{black}{scalar} $\alpha$-stable process such that $\mathbb{E}[\vert Z_1 \vert]$ is finite.
\textcolor{black}{Fixed two indexes $i$, $j$ in $\llbracket 1,n\rrbracket$ such that $j\ge i$ and an index $\beta^j_i$ in $(0,1]$ such that
\[
\beta^j_i \, < \, \frac{1+\alpha(i-2)}{1+\alpha(j-1)},
\]
we want to construct a time homogeneous drift $\bar{F}\colon \R^n\to \R^n$ such that its $i$-th component $F_i$ is $\beta^j_i$-H\"older continuous with respect to $x_j$ and uniqueness in law fails for the following SDE:}
\begin{equation}
\label{eq:SDE_for_counter}
    \begin{cases}
dX_t \, = \, \bigl[A\red{X_t}+\bar{F}(X_t)\bigr] \, dt+ BdZ_t, \quad t\ge 0, \\
X_0 \, = \, 0,
\end{cases}
\end{equation}
where the matrix $B$ in \textcolor{black}{$\R^n\otimes \R$} is now given by $B:=(1,0,\dots, 0)^t$.\newline
\textcolor{black}{The heuristic arguments at the beginning of the section suggests us to choose
\[\bar{F}(x) \, := \, e_i\text{sgn}(x_j)\vert x_j\vert^{\beta^j_i},\]
where $\{e_k\colon k \in \llbracket 1, n \rrbracket\}$ is the canonical orthonormal basis for $\R^n$}.

In particular, we are going to focus on the $i$-th component of the above Equation \eqref{eq:SDE_for_counter} that can be rewritten in integral form as:
\begin{equation}
\label{eq:SDE_for_counter_integral}
X^{\textcolor{black}{i}}_t \, = \, \int_0^t\text{sgn}\bigl(I^{j-i}_t(X^{\textcolor{black}{i}})\bigr)\bigl{\vert}I^{j-i}_t(X^j)\bigr{\vert}^{\beta^j_i} dt + I^{i-1}_t(Z), \quad t\ge 0,
\end{equation}
where we have denoted  by $I^k_t(y)$ the $k$-th \emph{iterated integral} of a c\`adl\`ag path $y\colon [0,\infty)\to \R$ at a time $t$. Namely,
\begin{equation}
\label{eq:iterate_int_operator}
I^k_t(y) \, := \, \int_{0}^{t_k=t}\dots \int_{0}^{t_2} y_{t_0} \, dt_0\dots dt_{k-1}, \quad t\ge0.
\end{equation}

In order to improve the readability of the next part, we are going to present our reasoning in a slightly more general way.
\begin{prop}
\label{prop:counter-example-formal}
Let $k$ be in $\N$, $\beta$ in $(0,1)$, $x$ in $\R$ and $\{\mathcal{Z}_t\}_{t\ge0}$ a continuous process on $\R$ such that
\begin{itemize}
\item $\mathbb{E}\bigl[\sup_{s\in [0,1]}|\mathcal{Z}_s|\bigr]<\infty$;
\item it is symmetric and $\gamma$-self-similar in law for some $\gamma>0$. Namely,
\[\bigl(\mathcal{Z}_t\bigr)_{t\ge 0}\, \overset{(\rm{law})}{=} \,\bigl(-\mathcal{Z}_t\bigr)_{t\ge 0}\,  \text{ and } \,\forall \rho>0,\ \bigl(\mathcal{Z}_{\rho t}\bigr)_{t\ge 0}\,
\overset{(\rm{law})}{=} \,\bigl(\mathcal{Z}_t \rho^\gamma\bigr)_{t\ge 0}.\]
\end{itemize}
Then, uniqueness in law fails for the following SDE:
\begin{equation}\label{eq:Peano_SDE}
\begin{cases}
dX_t \, = \, {\rm{sgn}}\bigl(I^k_t(X)\bigr) \bigl{\vert}I^k_t(X)\bigr{\vert}^\beta dt +d\mathcal{Z}_t, \quad t\ge 0 \\
X_0\, = \, x,
\end{cases}
\end{equation}
if $x=0$ and $\beta <\frac{\gamma-1}{\gamma+k}$.
\end{prop}
It is not difficult to check that
Equation \eqref{eq:SDE_for_counter_integral} satisfies the assumptions of the above proposition. Indeed, we can apply Proposition \ref{prop:counter-example-formal} to Equation \eqref{eq:SDE_for_counter_integral} taking $\gamma=i-1+\frac{1}{\alpha}$,
$k=j-i$ and $\beta=\beta^j_i=\beta$. It implies in particular that SDE \eqref{eq:SDE_for_counter} lacks of uniqueness in law if
\[\beta^j_i \, < \,\frac{\gamma-1}{\gamma+k} \, = \, \frac{1+\alpha(i-2)}{1+\alpha(j-1)}.\]
Hence, to complete the proof of Theorem \ref{thm:counterexample}, it suffices to establish Proposition \ref{prop:counter-example-formal}.

Before proving Proposition \ref{prop:counter-example-formal}, we need however an auxiliary result. It roughly states that any solution of SDE \eqref{eq:Peano_SDE}
starting outside zero cannot immediately reach the extremal solutions of the associated deterministic Peano example. Importantly, the constant $\rho$ appearing below does not depend on the starting point $x$.

\begin{lemma}
\label{lemma:Peano_example}
Fixed $x>0$ and $\beta<\frac{\gamma-1}{\gamma+k}$, let $\{X_t\}_{t\ge 0}$ be a solution of Equation \eqref{eq:Peano_SDE} starting from $x$. Then, there exist two
positive constants $\rho:=\rho(k,\beta,\gamma,\mathbb{E}[\sup_{s\in [0,1]}|\mathcal{Z}_s|])$ and $c_0:=c_0(k,\beta)$ such that
\begin{equation}\label{eq:control_on_tau}
\mathbb{P}\bigl(\tau(X)\ge \rho\bigr) \, \ge \, 3/4,
\end{equation}
where $\tau(X)$ is the stopping time on $\Omega$ given by
\begin{equation}\label{eq:Peano_random_time}
\tau(X) \, = \, \inf\{t \ge 0 \colon X_t \, \le \, c_0 t^{\frac{k\beta+1}{1-\beta}}\}.
\end{equation}
\end{lemma}
\begin{proof}
We start noticing that the process $\{X_t\}_{t\ge0}$ is continuous in $0$, since it is c\`adl\`ag. Fixed $c_0>0$ to be chosen later, it implies that $\tau(X)>0$, almost surely. In particular, it makes sense to consider the random interval $(0,\tau(X)]$.\newline
Fixed $t$ in $(0,\tau(X)]$, it holds, by definition of $\tau(X)$, that $X_t>c_0t^{\frac{k\beta+1}{1-\beta}}$. It follows then that
\[\int_{0}^{t}\bigl{\vert} I^k_s(X)\bigr{\vert}^\beta \, ds \, > \, \tilde{C}c^\beta_0t^{\frac{k\beta+1}{1-\beta}} \,\, \text{ where } \,\,
\tilde{C} \,:= \, \bigl(\prod_{i=1}^{k}\frac{k\beta+1}{1-\beta}+(i-1)\bigr)^{-\beta}
.\]
Since $x>0$ by assumption and $X>0$ on $(0,\tau(X)]$, we can now show that
\[X_t \, = \, x + \int_{0}^{t}\text{sgn}\bigl(I^k_s(X)\bigr)\bigl{\vert} I^k_s(X)\bigr{\vert}^\beta \, ds + \mathcal{Z}_t \, > \, \tilde{C}c^\beta_0
t^{\frac{k\beta+1}{1-\beta}}+ \mathcal{Z}_t.\]
The next step is to write $\tilde{C}c^\beta_0=c_0+\hat{C}$ for some constant $\hat{C}>0$. To do so, we need to choose carefully $c_0$. In particular, the
condition above is equivalent to the following
\[\hat C=\tilde{C}c^\beta_0-c_0\, > \, 0 \Leftrightarrow c_0 \, < \, \tilde{C}^{\frac{1}{1-\beta}}.\]
Fixed $c_0=\tilde{C}^{\frac{1}{1-\beta}}/2$, it then holds that
\[X_t \, > \, c_0t^{\frac{k\beta+1}{1-\beta}}+\hat{C}t^{\frac{k\beta+1}{1-\beta}}+ \mathcal{Z}_t\]
for any $t$ in $(0,\tau(X)]$. Fixed $\rho>0$ to be chosen later, we can now define the event $\mathcal A$ in $\Omega$ as
\[\mathcal A \, := \, \{\omega \in \Omega \colon \hat{C}t^{\frac{k\beta+1}{1-\beta}}+ \mathcal{Z}_t>0, \,\, \forall \, t \in (0,\rho]\}.\]
On $\mathcal A$ and for any $t$ in $(0,\tau(X)]$, it then holds that
\[X_t \, > \, c_0t^{\frac{k\beta+1}{1-\beta}}.\]
In particular, we have that $\ \tau(X)\ge \rho$  on $\mathcal A$  and thus, $\mathcal A\subseteq\{\tau(X)\ge \rho\}$ on $\Omega$. It immediately implies that
\[\mathbb{P}\bigl(\tau(X)\ge \rho\bigr) \, \ge \, \mathbb{P}(\mathcal A).\]
It remains to choose $\rho>0$ such that $\mathbb{P}(\mathcal A)\ge 3/4$. Write:
\begin{align*}
\mathbb{P}(\mathcal A)=&\mathbb P[\forall t\in (0,\rho],\  \hat{C}t^{\frac{k\beta+1}{1-\beta}}+\mathcal{Z}_t>0]=\mathbb P[\forall t\in (0,1],\  \hat{C}(\rho t)^{\frac{k\beta+1}{1-\beta}}+\mathcal{Z}_{\rho t}>0]\\
=&\mathbb P[\forall t\in (0,1],\  \hat{C}(\rho t)^{\frac{k\beta+1}{1-\beta}}+\rho^\gamma \mathcal{Z}_{t}>0]=\mathbb P[\forall t\in (0,1],\  \hat{C}\rho^{\frac{k\beta+1}{1-\beta}-\gamma}+t^{-\frac{k\beta+1}{1-\beta}} \mathcal{Z}_{t}>0],
\end{align*}
from the self-similarity assumption on $\mathcal Z$. Since by assumption $\beta<\frac{\gamma-1}{\gamma+k}\iff\frac{k\beta+1}{1-\beta}-\gamma<0 $, the statement will follow taking $\rho $ small enough as soon as we prove the process $\mathcal R_t:=t^{-\frac{k\beta+1}{1-\beta}} \mathcal{Z}_{t},\ t\in (0,1] $, which is continuous on the open set $(0,1]$, can be extended by continuity in $0$ with $\mathcal R_0=0$. Observe that $\mathbb E[|\mathcal R_t|]=t^{\gamma-\frac{k\beta+1}{1-\beta}} \mathbb E[|\mathcal Z_1|]\underset{t\rightarrow 0}{\longrightarrow} 0$. Setting $\delta:= \gamma-\frac{k\beta+1}{1-\beta}>0$ and introducing $ t_n:=n^{-1/\delta(1+\eta)},\eta>0$, we get that for all $\varepsilon>0 $,
$$\mathbb P[|\mathcal R_{t_n}|\ge \varepsilon]\le \varepsilon^{-1}\mathbb E[|\mathcal R_{t_n}|]=\varepsilon^{-1}t_n^\delta \mathbb E[|\mathcal Z_1|]=\varepsilon^{-1}n^{-(1+\eta)}\mathbb E[|\mathcal Z_1|].$$
We thus get from the Borel-Cantelli lemma that $\mathcal R_{t_n}\underset{n,\ a.s.}{\longrightarrow} 0 $. Namely, we have almost sure convergence along the subsequence $t_n$ going to zero with $n$. It now remains to prove that the process $\mathcal R_{t} $ does not fluctuate much between two successive times $t_n $ and $t_{n+1}$. Write for $t\in [t_{n+1},t_n]$:
\begin{align}
|R_t|:=|t^{-\frac{k\beta+1}{1-\beta}} \mathcal{Z}_{t}|\le& t_{n+1}^{-\frac{k\beta+1}{1-\beta}}\Big(|\mathcal{Z}_{t_{n+1}}|+ \sup_{s\in [t_{n+1},t_n]}|\mathcal{Z}_{s}-\mathcal{Z}_{t_{n+1}}|\Big)\notag\\
\le &  t_{n+1}^{-\frac{k\beta+1}{1-\beta}}\Big(2|\mathcal{Z}_{t_{n+1}}|+ \sup_{s\in [0,t_n]}|\mathcal{Z}_{s}|\Big)\label{CTR_REMAIN_BC}.
\end{align}
The first term of the above left hand side tends almost surely to zero with $n$. Observe as well that, from the scaling properties of $\mathcal{Z}$, for any $\varepsilon>0 $:
\begin{align*}
 \mathbb P[ t_{n+1}^{-\frac{k\beta+1}{1-\beta}} \sup_{s\in [0,t_n]}|\mathcal{Z}_{s}|\ge \varepsilon]&=\mathbb P[ t_{n+1}^{-\frac{k\beta+1}{1-\beta}}t_n^\gamma \sup_{s\in [0,1]}|\mathcal{Z}_{s}|\ge \varepsilon]\le \varepsilon^{-1}t_n^{\delta} (\frac{t_n}{t_{n+1}})^{\frac{k\beta+1}{1-\beta}} \mathbb E[\sup_{s\in [0,1]}|\mathcal{Z}_{s}|]\\
 &\le C\varepsilon^{-1}n^{-(1+\eta)}\mathbb E[\sup_{s\in [0,1]}|\mathcal{Z}_{s}|],
\end{align*}
which again gives from the Borel-Cantelli lemma the a.s. convergence with $n$ of the second term in the r.h.s of \eqref{CTR_REMAIN_BC}. We eventually derive
that $\mathcal R_t \underset{t\rightarrow 0, a.s.}{\longrightarrow} 0$.
Again, the key point is that we normalize  the process $\mathcal Z$ at a rate, $t^{\frac{k\beta+1}{1-\beta}} $, which is lower than its own characteristic time scale, $t^{\gamma} $. This is precisely what leaves some margin to establish continuity.

\end{proof}

Exploiting the lower bound for the random time $\tau(X)$ given in Lemma \ref{lemma:Peano_example}, we are now ready to show uniqueness in law fails for SDE \eqref{eq:Peano_SDE} when $x=0$
and $\beta<\frac{\gamma-1}{\gamma+k}$.

\emph{Proof of Proposition \ref{prop:counter-example-formal}.}
By contradiction, we start assuming that uniqueness in law holds for SDE \eqref{eq:Peano_SDE} starting at $x=0$. Fixed any solution $\{X_t\}_{t\ge 0}$ of Equation \eqref{eq:Peano_SDE}
starting at zero, it follows by symmetry that $\{-X_t\}_{t\ge 0}$ is also a solution of the same dynamics.
Since by hypothesis, $-\mathcal{Z}_t\overset{(\text{law})}{=}\mathcal{Z}_t$, uniqueness in law for SDE \eqref{eq:Peano_SDE} implies that the laws of $X$ and $-X$ are identical.\newline
Assuming for the moment that Lemma \ref{lemma:Peano_example} is applicable for $x=0$, we easily find a contradiction. Indeed, it follows from Lemma \ref{lemma:Peano_example}
that
\[\mathbb{P}\bigl(\tau(X)\ge \rho\bigr) \, \ge \, 3/4\]
but on the same time, thanks to the uniqueness in law, we have that
\[\mathbb{P}^0\bigl(\tau(-X)\ge \rho\bigr) \, \ge \, 3/4,\]
which is clearly impossible. To show the validity of Lemma \ref{lemma:Peano_example} in $x=0$, we consider a a sequence $\{\{X^n_t\}_{t\ge 0}\colon n \in \N\}$ of solutions of SDE \eqref{eq:Peano_SDE} starting at $1/n$. It is then easy to check that such a sequence satisfies the Aldous criterion:
\[\mathbb{E}[\vert X^n_t-X^n_0\vert^p] \, \le \, ct^{p\gamma}, \quad t\ge 0\]
for some $p>0$ and $c>0$ independent from $t$ and $n$. It follows (Proposition $34.8$ in \cite{book:Bass11}) that the sequence $\{\mathbb{P}^n\}_{n\in \N}$
of the laws of $\{X^n_t\}_{t\ge0}$ is tight.
Prohorov Theorem (cf. Theorem $30.4$ in \cite{book:Bass11}) ensures now the existence of a converging sub-sequence
$\{\mathbb{P}^{n_k}\}_{k\in \N}$. The uniqueness in law then implies that the sequence $\{\mathbb{P}^{n_k}\}_{k\in \N}$ converges, as
expected, to $\mathbb{P}^0$ the law of the solution starting at $0$. Noticing that inequality \eqref{eq:control_on_tau} holds for any solution $\{X^n_t\}_{t\ge 0}$ and moreover, the constant $\rho$ is independent from the starting points $1/n$, we find that
\[\mathbb{P}\bigl(\tau(X)\ge \rho\bigr) \, \ge \, 3/4.\]
The proof of Proposition \ref{prop:counter-example-formal} is thus concluded.

\setcounter{equation}{0}
\appendix
\section{Appendix}

\subsection{Controls on the density of the proxy process}
We present here two useful lemmas needed to complete the proof of Proposition \ref{prop:Smoothing_effect}. We will analyze the behavior of the laws of the independent random variables $ \tilde{M}^{\tau,\xi,t,s}$ and $\tilde{N}^{\tau,\xi,t,s}$ obtained in \eqref{eq:decomposition_S} by truncation of the process $\tilde{S}^{\tau,\xi,t,s}$ at the associated stable time scale $u^{1/\alpha}$.

\begin{lemma}
\label{lemma:Control_p_M}
Let $m$ be in $\N$. Then, there exists a positive constant $C:=C(m,T)$ such that for any $k$ in $\llbracket 0, m \rrbracket$,
\[
\left| D^k_{z} p_{\tilde{M}^{\tau,\xi,t,s}}(u,z) \right| \, \le \, Cu^{-(N+k)/\alpha}\left(1+\frac{|z|}{u^{1/\alpha}}\right)^{-m} \, =: \, Cu^{-k/\alpha}p_{\overline{M}}(u,z),
\]
for any $u>0$, any $z$ in $\R^N$, any $t\le s$ in $[0,T]$ and any $(\tau,\xi)$ in $[0,T]\times \R^N$.
\end{lemma}
\begin{proof}
Similarly to the proof of Proposition \ref{prop:Decomposition_Process_X} (see in particular Equation \eqref{eq:definition_density_q}), we start writing
\[
p_{\tilde{M}^{\tau,\xi,t,s}}(u,z) \, = \,\frac{1}{(2\pi)^N}\int_{\R^N}e^{-i\langle z,y\rangle}\text{exp}\left(u\int_{|p|\le u^{1/\alpha}}\left[\cos(\langle y,p \rangle)-1\right]\, \nu_{\tilde{S}^{\tau,\xi,t,s}}(dp)\right)\, dy,
\]
\textcolor{black}{where, we recall, $\nu_{\tilde{S}}^{\tau,\xi,t,s}$ is the L\'evy measure associated with the process $\{\tilde{S}^{\tau,\xi,t,s}_{u}\}_{u\ge0}$ in Proposition \ref{prop:Decomposition_Process_X}}. Setting $u^{1/\alpha}y =\tilde{y}$ then yields
\begin{align}\notag
p_{\tilde{M}^{\tau,\xi,t,s}}(u,z) \, &= \, \frac{u^{-N/\alpha}}{(2\pi)^N}\int_{\R^N}  e^{-i\langle z, \frac{\tilde{y}}{u^{1/\alpha}}\rangle}\text{exp}\left(u\int_{|p|\le u^{1/\alpha}}\left[\cos(\langle \tilde{y},\frac{p}{u^{1/\alpha}} \rangle)-1\right] \nu_{\tilde{S}^{\tau,\xi,t,s}}(dp)\right) d\tilde{y} \\
&=: \, \frac{u^{-N/\alpha}}{(2\pi)^N}\int_{\R^N}  e^{-i\langle \frac{z}{u^{1/\alpha}},\tilde{y}\rangle}\hat{f}^{\tau,\xi,t,s}_u(\tilde{y})\, d\tilde{y}
\label{EXP_DENS_M}
\end{align}
Since the L\'evy measure $\nu_{\tilde{S}}^{\tau,\xi,t,s}$ in the expression above has finite support, Theorem $3.7.13$ in Jacob \cite{book:Jacob05} implies that $\hat{f}^{\tau,\xi,t,s}_u$ is infinitely differentiable in $\tilde{y}$. We can thus calculate
\[
\begin{split}
|\partial_{\tilde{y}} \hat{f}^{\tau,\xi,t,s}_u(\tilde{y})| \, &\le \, u\int_{|p|\le u^{1/\alpha}}\frac{|p|}{u^{1/\alpha}}\left|\sin\left(\bigl{\langle} \tilde{y},\frac{p}{u^{1/\alpha}} \bigr{\rangle}\right)\right|\, \nu_{\tilde{S}^{\tau,\xi,t,s}}(dp) \\
&\qquad\qquad\qquad\qquad \times \text{exp}\left(u\int_{|p|\le u^{1/\alpha}}\left[\cos\left(\bigl{\langle} \frac{\tilde{y}}{u^{1/\alpha}},p \bigr{\rangle}\right)-1\right]\, \nu_{\tilde{S}^{\tau,\xi,t,s}}(dp)\right).
\end{split}
\]
Recalling that $\alpha>1$, we can now write that
\[
\begin{split}
 u\int_{|p|\le u^{1/\alpha}}\frac{|p|}{u^{1/\alpha}}\left|\sin\left(\bigl{\langle} \tilde{y},\frac{p}{u^{1/\alpha}} \bigr{\rangle}\right)\right|\, \nu_{\tilde{S}^{\tau,\xi,t,s}}(dp)\, &\le \, Cu\int_{r\le u^{1/\alpha}}  \frac{r}{u^{1/\alpha}}    \frac{|\tilde{y}|r}{u^{1/\alpha}} \frac{dr}{r^{1+\alpha}} \\
&\le \, Cu\int_{r\le u^{1/\alpha}} |\tilde{y}| \frac{r^{1-\alpha}}{u^{2/\alpha}} \, dr \\
&\le \,  C(1+|\tilde{y}|).
\end{split}
\]
It then follows that
\[\begin{split}
&|\partial_{\tilde{y}} \hat{f}^{\tau,\xi,t,s}_u(\tilde{y})|
\\
&\,\,\qquad\le \, C(1+|\tilde{y}|) \text{exp}\left(u\int_{\R^N}\left[\cos\left(\bigl{\langle} \frac{\tilde{y}}{u^{1/\alpha}},p \bigr{\rangle}\right)-1\right]\, \nu_{\tilde{S}^{\tau,\xi,t,s}}(dp)\right)e^{2u\nu_{\tilde{S}^{\tau,\xi,t,s}}(B^c(0,u^{1/\alpha}))}\\
&\,\,\qquad\le \,
C (1+|\tilde{y}|)\exp(-C^{-1}|\tilde{y}|^\alpha),
\end{split}\]
where in second inequality we exploited Control \eqref{eq:control_Levy_symbol_S} and
\begin{equation}
\label{FINITE_MEAS}
\nu_{\tilde{S}^{\tau,\xi,t,s}}(B^c(0,u^{1/\alpha})) \le C/u.
\end{equation}
Iterating the above reasoning, we can then show that for any $l$ in $\N$,
\[
|\partial^l_{\tilde{y}}\hat{f}^{\tau,\xi,t,s}_u(\tilde{y})| \, \le\,
C_l (1+|\tilde{y}|^l)\exp(-C^{-1}|{\tilde{y}}|^\alpha),
\]
for some positive constant $C:=C(l)$.
It implies in particular that $\hat{f}^{\tau,\xi,t,s}_u(\tilde{y})$ is a Schwartz test function.
Denoting by $f^{\tau,\xi,t,s}_u$ its inverse Fourier transform, we thus have
that for any $m$ in $\N$, there exists a positive constant $C:=C(m)$ such that
\[
 |f^{\tau,\xi,t,s}_u(y)| \le C_m (1+|y|)^{-m}, \quad y \in \R^N.
\]
The result for $k=0$ now follows immediately noticing that
\[p_{\overline{M}}(t-s,y)\,  = \, (t-s)^{-\frac d\alpha} f_{s,t}(y/(t-s)^{\frac{1}{\alpha}}).\]
The controls on the derivatives can be derived analogously.
\end{proof}

We can now show a similar control on the law of the process $\tilde{N}^{\tau,\xi,t,s}$.

\begin{lemma}
\label{lemma:Control_P_N}
There exists a family $\{\overline{P}_u\}_{u\ge 0}$ of Poisson measures and a positive constant $C:=C(T,N)$ such that for any $\mathcal A$ in $\mathcal{B}(\R^N)$ and $\tilde{N}^{\tau,\xi,t,s}$ as in \eqref{DEF_TILDE_N},
\begin{equation}\label{DEF_POISSON_QUI_DOMINE}
P_{\tilde{N}^{\tau,\xi,t,s}_u}(\mathcal A) \, \le \, C\overline{P}_u(\mathcal A).
\end{equation}
\end{lemma}
\begin{proof}
For notational simplicity, we start introducing \textcolor{black}{the truncated L\'evy measure  associated with the big jumps of the process} $\{\tilde{S}^{\tau,\xi,t,s}_u\}_{u\ge 0}$:
\[\nu_{\text{tr}}^{\tau,\xi,t,s}(dp)\, = \, \mathds{1}_{|p|\ge u^{1/\alpha}}(p)\nu_{\tilde{S}}^{\tau,\xi,t,s}(dp). \]
It follows immediately that $\nu_{\text{tr}}^{\tau,\xi,t,s}$ is a finite measure (see \eqref{FINITE_MEAS} above). With this notation at hand, we can write:
\[
\begin{split}
\widehat {P_{\tilde{N}^{\tau,\xi,t,s}_u}}(y) \, &= \, \exp\left(u \int_{|p|>u^{\frac{1}{\alpha}}}\left[\cos(\langle y,p\rangle)-1\right] \, \nu_{\tilde{S}}^{\tau,\xi,t,s}(dp)\right) \\
&= \, \exp\left( u \widehat{\nu_{\text{tr}}^{\tau,\xi,t,s}}(y)- u \nu_{\text{tr}}^{\tau,\xi,t,s}(\R^N)\right),
\end{split}
\]
where $\widehat{\nu}$ denotes the Fourier-Stieltjes transform of the considered measure $\nu$. Let us introduce then the following measure:
\[\zeta^{\tau,\xi,t,s}\, := \, u\nu_{\text{tr}}^{\tau,\xi,t,s}.\]
Expanding the previous exponential and by termwise Fourier inversion, we now find that
\begin{equation}
\begin{split}
\label{proof:appendix_decomposition_P_N}
P_{\tilde{N}^{\tau,\xi,t,s}_u} (\mathcal{A}) \, &= \, \exp\left(\zeta^{\tau,\xi,t,s} (\mathcal{A}) - u \nu_{\text{tr}}^{\tau,\xi,t,s}(\R^N)\right) \\
&\textcolor{black}{:=} \, \exp\left(-u\nu_{\text{tr}}^{\tau,\xi,t,s}(\R^N)\right)\sum_{n \in \N}\frac{\left( \zeta^{\tau,\xi,t,s}
\right)^{\star n}(\mathcal{A})}{n!},
\end{split}
\end{equation}
where, for a finite measure $\rho$ on $\R^N$, $(\rho)^{\star n}:= \rho\star\cdots \star \rho$ denotes its $n^{{\rm th}}$ fold convolution.\newline
For now, let us assume that $\sigma(t,x)$ is non-constant in space, so that
\[B\tilde{\sigma}^{\tau,\xi}_{u(v)}\, = \, B\sigma\left(u(v),\theta_{u(v),\tau}(\xi)\right)\]
\textcolor{black}{appearing in the definition of $\nu_{\tilde{S}}^{\tau,\xi,t,s}$}, truly depends on the parameters $\tau,\xi$. Assumption [\textbf{AC}] then ensures the existence of a bounded function $g\colon \mathbb{S}^{d-1}\to \R$ such that
\[\nu(dp) \, = \, Q(p)\frac{g(\frac{p}{|p|})}{|p|^{d+\alpha}}dp.\]
From Equation \eqref{proof:appendix_decomposition_P_N}, it is clear that we need to control the measure $\zeta^{\tau,\xi,t,s}$, uniformly in the parameters $\tau,\xi,t,s$. Namely, for any $\mathcal A$ in $\mathcal{B}(\R^N)$, we write from \eqref{proof:eq:def_Levy_symbol_S} that
\[
\begin{split}
\zeta^{\tau,\xi,t,s}(\mathcal A) \, &= \, u\int_{|p|>u^{\frac{1}{\alpha}}}\mathds{1}_{\mathcal A}(p)\, \nu_{\tilde{S}}^{\tau,\xi,t,s}(dp)\, = \, u\int_0^1\int_{|\widehat{\mathcal{R}}_vB\tilde{\sigma}^{\tau,\xi}_{u(v)}p|>u^{\frac{1}{\alpha}}}\mathds{1}_{\mathcal A}(\widehat{\mathcal{R}}_vB\tilde{\sigma}^{\tau,\xi}_{u(v)}p)\, \nu(dp)dv\\
&= u\int_0^1\int_{|\widehat{\mathcal{R}}_vB\tilde{\sigma}^{\tau,\xi}_{u(v)}p|>u^{\frac{1}{\alpha}}}\mathds{1}_{\mathcal A}(\widehat{\mathcal{R}}_vB\tilde{\sigma}^{\tau,\xi}_{u(v)}p)\frac{g(\frac{p}{|p|})}{|p|^{d+\alpha}}Q(p)\, dpdv\\
&\le \, u\int_0^1 \int_{|\widehat{\mathcal{R}}_vB\tilde{\sigma}^{\tau,\xi}_{u(v)}p|>u^{\frac{1}{\alpha}}}\mathds{1}_{\mathcal A}(\widehat{\mathcal{R}}_vB\tilde{\sigma}^{\tau,\xi}_{u(v)}p)\frac{dp}{|p|^{d+\alpha}}dv
\end{split}
\]
We can then exploit assumption [\textbf{UE}] on $\sigma$ to conclude that
\[
\begin{split}
\zeta^{\tau,\xi,t,s}(\mathcal A) \, &\le \,  u\int_0^1\int_{|\widehat{\mathcal{R}}_vBq|>u^{\frac{1}{\alpha}}}\mathds{1}_{\mathcal A}(\widehat{\mathcal{R}}_vB q)\frac{1}{\det (\tilde{\sigma}^{\tau,\xi}_{u(v)})}\frac{dq}{|(\tilde{\sigma}^{\tau,\xi}_{u(v)})^{-1}q|^{d+\alpha}}dv\\
&\le \, Cu\int_0^1\int_{|\widehat{\mathcal{R}}_vBq|>u^{\frac{1}{\alpha}}}\mathds{1}_{\mathcal A}(\widehat{\mathcal{R}}_vBq)\frac{dq}{|q|^{d+\alpha}}dv.
\end{split}
\]
Denoting now by $\Lambda_{\text{tr}}:=c\mathds{1}_{p>u^{1/\alpha}}\frac{dp}{p^{d+\alpha}}$ the truncated L\'evy measure of the isotropic $\alpha$-stable process and by $\overline{\nu}_{\text{tr}}$ the following push-forward measure
\[\overline{\nu}_{\text{tr}}(\mathcal A) \, := \, \int_0^1\Lambda_{\text{tr}}\left((\widehat{\mathcal{R}}_vB)^{-1}\mathcal A\right)dv, \quad \mathcal A \in \mathcal{B}(\R^N)\]
we derive that there exists a constant $C$ such that for any $(\tau,\xi)$ in $[0,T]\times \R^N$, $t\le s$ in $[0,T]$,
\begin{equation}
\label{proof:appendix_control_N}
\zeta^{\tau,\xi,t,s}(\mathcal A) \, \le \, C u \int_0^1\Lambda_{\text{tr}}\left((\widehat{\mathcal{R}}_vB)^{-1}\mathcal A\right)dv \, = \, u\overline{\nu}_{\text{tr}}(\mathcal A) \, =: \, \overline{\zeta}(\mathcal A).
\end{equation}

Equation \eqref{DEF_POISSON_QUI_DOMINE} now follows from
the above control, \eqref{FINITE_MEAS} and \eqref{proof:appendix_decomposition_P_N}, denoting
\[\overline{P}_u \, := \, \exp\left(-u\overline{\nu}_{\text{tr}}(\R^N)\right)\sum_{n \in \N}\frac{(\overline{\zeta})^{\star n}}{n!},\]
up to a modification of the constant $C$ in \eqref{proof:appendix_control_N}.
Following backwards the same reasoning presented at the beginning of the proof, we then notice that
\[
\begin{split}
\widehat {\overline{P}_u}(y) \, &= \, \exp\left(u \int_0^1\int_{\R^N}\left[\cos(\langle y,p\rangle)-1\right] \, \overline{\nu}_{\text{tr}}(dp)dv\right) \\
&= \, \exp \left(u\int_0^1 \int_{\R^d}\mathds{1}_{\{|\widehat{\mathcal{R}}_vBp|>u^{\frac{1}{\alpha}}\}}\left[\cos(\langle y,\widehat{\mathcal{R}}_vBp\rangle)-1\right] \, \Lambda(dp)dv\right)\\
&= \, \exp \left(u\int_0^1 \int_0^\infty \int_{\mathbb{S}^{d-1}} \mathds{1}_{\{|\widehat{\mathcal{R}}_vB\theta r|>u^{\frac{1}{\alpha}}\}}\left[\cos(\langle y,\widehat{\mathcal{R}}_vB\theta r\rangle)-1\right] \, \mu_{\text{leb}}(d\theta) \frac{dr}{r^{1+\alpha}}dv\right),
\end{split}
\]
where we used the spherical decomposition \textcolor{black}{for the L\'evy measure $\Lambda$ of an isotropic $\alpha$-stable process:}
\begin{equation}
\Lambda(dp) := \frac{dp}{p^{d+\alpha}} \, = \, C\mu_{\text{leb}}(d\theta)\frac{dr}{r^{1+\alpha}},
\end{equation}
with $p=r \theta$ and $\mu_{\text{leb}}$ Lebesgue measure on the sphere $\mathbb{S}^{d-1}$.\newline
We exploit now the non-degeneracy of $\widehat{\mathcal{R}}_v$ to to define two functions $k\colon [0,1]\times \mathbb{S}^{d-1}\to \R$ and $l\colon [0,1]\times
\mathbb{S}^{d-1}\to \mathbb{S}^{N-1}$, given by
\[k(v,\theta) \,:= \,
\vert \widehat{\mathcal{R}}_{v}B\theta\vert \,\, \text{ and } \,\,  l(v,\theta) \,:= \, \frac{\widehat{\mathcal{R}}_{v}B\theta}{\vert\widehat{\mathcal{R}}_{v}B\theta\vert}.\]
Using the Fubini theorem, we can now write that
\[
\begin{split}
&\widehat{\overline{P}_u}(y) \\
&\,\,= \, \exp\left(u\int_0^1\int_{0}^{\infty}\int_{\mathbb{S}^{d-1}} \mathds{1}_{\{|l(v,\theta)k(v,\theta)r|>u^{\frac{1}{\alpha}}\}}
\left[\cos\left(\langle z,l(v,\theta)k(v,\theta)r\rangle\right)-1\right]\, \mu_{\text{leb}}(d\theta)\frac{dr}{r^{1+\alpha}}dv\right) \\
&\,\,= \, \exp\left(u\int_0^1\int_{0}^{\infty}\int_{\mathbb{S}^{d-1}}\mathds{1}_{\{|l(v,\theta)\tilde{r}|>u^{\frac{1}{\alpha}}\}}
\left[\cos\left(\langle z,l(v,\theta)\tilde{r}\rangle\right)-1\right]\, [k(v,\theta)]^\alpha \mu_{\text{leb}}(d\theta)\frac{d\tilde{r}}{\tilde{r}^{1+\alpha}}dv\right).
\end{split}
\]
Denoting now by $\tilde{k}(dv,d\theta)$ the measure on $[0,1]\times \mathbb{S}^{d-1}$ given by
\[\tilde{k}(dv,d\theta) \, := \, [k(v,\theta)]^{\alpha}\mu_{\text{leb}}(d\theta)dv\]
and by $\tilde{\mu}_{\text{sym}}:=\text{Sym}(l)_{\ast}\tilde{k}$ the symmetrization of the measure $\tilde{k}(dv,d\theta)$ push-forwarded through $l$ on $\mathbb{S}^{N-1}$, we can finally conclude that
\begin{align}
\notag \widehat {\overline{P}_u}(y) \, &=\, \exp\left(u\int_0^\infty
\int_{[0,1]\times \mathbb{S}^{d-1}}\mathds{1}_{\{|l(v,\theta)\tilde{r}|>u^{\frac{1}{\alpha}}\}}
\left[\cos\left(\langle z,l(v,\theta)\tilde{r}\rangle\right)-1\right]\, \tilde{k}(dv,d\theta)\frac{d\tilde{r}}{\tilde{r}^{1+\alpha}}\right) \\
&= \, \exp\left(u\int_{|u|^{\frac{1}{\alpha}}}^{\infty}\int_{\mathbb{S}^{N-1}}\left[\cos\left(\langle z, \tilde{\theta} \tilde{r} \rangle\right)-1\right]\tilde{\mu}_{\text{sym}}(d\tilde{\theta})\frac{d\tilde{r}}{\tilde{r}^{1+\alpha}}\right).\label{eq:representation_P_N_segnato}
\end{align}
It is easy to check now that the measure $\tilde{\mu}_{\text{sym}}$ is
finite and non-degenerate in the sense of \eqref{eq:non_deg_measure}. This \textcolor{black}{concludes} the proof of our result under the additional assumption that $\nu$ is absolutely continuous with respect to \textcolor{black}{the} Lebesgue measure. \newline
If this is not the case, assumption [\textbf{AC}] implies immediately that $\sigma(t,x)=:\sigma_t$ does not depends on $x$. Thus,  the ``frozen'' diffusion $\tilde{\sigma}^{\tau,\xi}_t$ does not depends on the parameters $\tau,\xi$ as well. The same arguments \textcolor{black}{as} above then allow to conclude. 
\end{proof}

\paragraph{Sketch of proof for Proposition \ref{coroll:Smoothing_effect_delta}}
We briefly present here the proof of Proposition \ref{coroll:Smoothing_effect_delta} concerning the existence and the associated controls for the density of the mollified frozen process $\tilde{X}^{\tau,\xi,t,x,\delta}_{s}$.\newline
We start noticing that the reasoning in the proof of Proposition \ref{prop:Decomposition_Process_X} can be similarly applied. Indeed, from the definition in \eqref{eq:decomposition_delta_frozen}, it follows immediately that
\[\tilde{X}^{\tau,\xi,t,x,\delta}_{s} \, = \, \tilde{m}^{\tau,\xi}_{s,t}(x)+\mathbb{M}_{s-t}\left(\tilde{S}^{\tau,\xi,t,s}_{s-t}+\delta\overline{Z}_{s-t}\right),\]
and thus, that there exists a density $\tilde{p}^{\tau,\xi,\delta}(t,s,x,y)$ associated with the frozen process $\tilde{X}^{\tau,\xi,t,x,\delta}_{s}$. Moreover, the representation in \eqref{eq:representation_density} holds again if we change there the L\'evy measure $\nu_{\tilde{S}}^{\tau,\xi,t,s}$ with the one associated with the following L\'evy symbol:
\[\Phi_{\tilde{S}^{\tau,\xi,t,s,\delta}}(z) \, := \,
\Phi_{\tilde{S}^{\tau,\xi,t,s}}(z) + c_\alpha\delta |z|^\alpha\, = \,
\int_{0}^{1}\Phi\bigl((\widehat{\mathcal{R}}_{v}B\tilde{\sigma}^{\tau,\xi}_{u(v)})^*z\bigr)\,dv+c_\alpha\delta |z|^\alpha.\]
Namely, it holds that
\begin{multline*}
\tilde{p}^{\tau,\xi,\delta}(t,s,x,y) \, = \, \frac{\det \mathbb{M}^{-1}_{s-t}}{(2\pi)^N}\int_{\R^N}e^{-i\langle \mathbb{M}^{-1}_{s-t}(y-\tilde{m}^{\tau,\xi}_{s,t}(x)),z\rangle}\\
\times\exp\left((s-t)\int_{\R^N}\left[\cos(\langle z, p \rangle )-1\right]\nu_{\tilde{S}^{\tau,\xi,t,s,\delta}}(dp)\right)\, dz,
\end{multline*}
where the L\'evy measure $\nu_{\tilde{S}^{\tau,\xi,t,s,\delta}}$ is given by
\begin{equation}
\label{eq:decomposition_delta}
\nu_{\tilde{S}^{\tau,\xi,t,s,\delta}}(\mathcal{A}) \, = \, \nu_{\tilde{S}^{\tau,\xi,t,s}}(A)+\delta^\alpha \nu_{\overline{Z}}(\mathcal{A}), \quad \mathcal{A} \in \mathcal{B}(\R^N),
\end{equation}
with $\nu_{\overline{Z}}$ L\'evy measure of the isotropic $\alpha$-stable process $\overline{Z}_t$.
In particular, the L\'evy symbol $\Phi_{\tilde{S}^\delta}$ satisfies Control \eqref{eq:control_Levy_symbol_S} for a constant $C$ independent from $\delta$.\newline
We can now move to show the controls on the derivatives
of the mollified frozen density. It is not difficult to check that the arguments presented in the proofs of Proposition \ref{prop:Smoothing_effect}, Lemmas \ref{lemma:Control_p_M} and \ref{lemma:Control_P_N} can be applied again if we substitute there the L\'evy measure $\nu_{\tilde{S}^{\tau,\xi,t,s}}$ with the mollified one $\nu_{\tilde{S}^{\tau,\xi,t,s,\delta}}$. Indeed, taking into account the decomposition in \eqref{eq:decomposition_delta}, we notice that the L\'evy measure $\nu_{\tilde{S}^{\tau,\xi,t,s,\delta}}$ only considers an additional term ($\delta \nu_{\overline{Z}}$) that has the same $\alpha$-scaling nature considered before (but is however much less singular).\newline
To show instead that the estimates \eqref{eq:smoothing_effect_frozen_delta} are indeed uniform in the parameter $\delta$, it is sufficient to notice from \eqref{eq:decomposition_delta} that we have that
\[\nu_{\tilde{S}^{\tau,\xi,t,s,\delta}}(\mathcal{A}) \, \le \, \nu_{\tilde{S}^{\tau,\xi,t,s}}(\mathcal{A})+ \nu_{\overline{Z}}(\mathcal{A}), \quad \mathcal{A} \in \mathcal{B}(\R^N).\]
To conclude the proof of Proposition \ref{coroll:Smoothing_effect_delta}, it is then enough to take $\xi=y$, $\tau=s$ and to follow the same arguments introduced in the proof of Corollary \ref{coroll:Smoothing_effect}.

\subsection{Proof of the Technical Lemmas}
\label{SEC_TEC_LEMMA_APP}
\paragraph{Proof of Lemma \ref{lemma:bilip_control_flow} (Approximate Lipschitz condition of the flows)}
We start considering two measurable flows $\theta,\check{\theta}$ satisfying dynamics \eqref{eq:measurability_flow}. Recalling the decomposition $G(t,x)=A_tx +F(t,x)$, it follows immediately that:
\begin{equation}
\begin{split}
\label{eq:proof_bilip_control_flow1}
\mathbb{T}_{s-t}^{-1}(x-\theta_{t,s}(y))\, &= \, \mathbb{T}_{s-t}^{-1}\Bigl[\check{\theta}_{s,t}(x)-y-\int_t^s \Bigl(G(u,\check{\theta}_{u,t}(x))-G(u,\theta_{u,s}(y))\Bigr) \, du\Bigr] \\
&= \, \mathbb{T}_{s-t}^{-1}\bigl(\check{\theta}_{s,t}(x)-y\bigr)+ \mathcal{I}_{s,t}(x,y),
\end{split}
\end{equation}
where in the last step, we denoted
\[\mathcal{I}_{s,t}(x,y) \, = \, \mathbb{T}_{s-t}^{-1}\int_{t}^s \left[A_u\left(\theta_{u,s}(y)-\check{\theta}_{u,t}(x)\right)+\left(F(u,\theta_{u,s}(y))-F(u,\check{\theta}_{u,t}(x))\right)\right] \, du.\]
To conclude, we need to show the following bound for $\mathcal{I}_{s,t}(x,y)$:
\begin{equation}
\label{eq:proof_bilip_control_flow2}
|\mathcal{I}_{s,t}(x,y)| \, \le \,  C\left[1+(s-t)^{-1}\int_t^s |\mathbb{T}_{s-t}^{-1}(\check{\theta}_{u,t}(x)-\theta_{u,s}(y))|\,    du\right].
\end{equation}
Indeed, Control \eqref{eq:proof_bilip_control_flow2} together with \eqref{eq:proof_bilip_control_flow1} and the Gronwall lemma imply the right-hand side of Control \eqref{eq:bilip_control_flow}. The left-hand side one can be obtained analogously and we will not show it here.\newline
We start decomposing $\mathcal{I}_{s,t}$ into $\mathcal{I}^1_{s,t}+\mathcal{I}^2_{s,t}$, where we denote
\begin{align*}
\mathcal{I}^1_{s,t}(x,y) \, := \, \mathbb{T}_{s-t}^{-1}\int_{t}^s A_u\left(\theta_{u,s}(y)-\check{\theta}_{u,t}(x)\right) \, du; \\
\mathcal{I}^2_{s,t}(x,y) \, := \, \mathbb{T}_{s-t}^{-1}\int_{t}^s \left[F(u,\theta_{u,s}(y))-F(u,\check{\theta}_{u,t}(x))\right] \, du.
\end{align*}
The first remainder $\mathcal{I}^1_{s,t}$ can be controlled easily, exploiting the linearity of $z\to A_uz$. Indeed, for any $z$,$z'$ in $\R^N$ and any $u$ in $[s,t]$, we have that
\begin{equation}
\label{eq:proof_bilip_control_flow6}
\begin{split}
    \left|\mathbb{T}_{s-t}^{-1} A_u(z-z')\right| \, &\le \, \sum_{i=1}^n \sum_{j=(i-1)\vee 1}^n(s-t)^{-\frac{1+\alpha(i-1)}{\alpha}}|A^{i,j}_u|\ |(z-z')_j| \\
    &\le \, C(s-t)^{-1}|\mathbb{T}_{s-t}^{-1}(z-z')|.
\end{split}
\end{equation}
To control instead the second term $\mathcal{I}^2_{s,t}$, we will need to thoroughly exploit an appropriate smoothing method, due to the low regularity in space of the drift $F$. To overcome this problem, we are going to mollify the function $F$ in the following way. We start fixing a family  $\{\rho_i\colon i \in \llbracket 1,n\rrbracket\}$ of mollifiers on $\R^{D_i}$ where $D_i=N-\sum_{j=1}^{i-1}d_j$, i.e. for any $i$ in $\llbracket 1,n\rrbracket$, $\rho_i$ is a
compactly supported, non-negative, smooth function on $\R^{D_i}$ such that $\Vert \rho_i \Vert_{L^1}=1$, and a family $\{\delta_{ij}\colon i \le j\}$ of positive constants to be
chosen later. Then, the mollified version of the drift is defined by $F^\delta:=(F_1,F^{\delta}_2,\dots,F^{\delta}_n)$
where
\begin{equation}
\label{eq:multi_scale_mollification}
\begin{split}
    F^{\delta}_i(t,z) \, &:= \,  F_i \ast_x \rho^{\delta}_i(t,z) \\
    &:= \, \int_{\R^{D_i}}
 F_i(t, z_i-\omega_i,\dots,z_n-\omega_n)\frac{1}{\prod_{j=i}^n\delta_{ij}^{d_i}} \rho_i(\frac{\omega_i}{\delta_{ii}},\dots,\frac{\omega_n}{\delta_{in}}) \, d\omega.
\end{split}
\end{equation}
\textcolor{black}{Roughly speaking, we have mollified any component $F_i$ by convolution in space with a mollifier with multi-scaled dilations}.
Then, standard results on mollifier theory and our current assumptions on $F$ show us that the following controls hold
\begin{align}
\label{Proof:Controls_on_flows_mollifier} \vert  F_i(u,z) - F^{\delta}_i(u,z)\vert \, &\le \,C\sum_{j=i}^n\delta_{ij}^{\beta^j}, \\
\label{Proof:Controls_on_flows_mollifier1} | F^{\delta}_i(u,z) - F^{\delta}_i(u,z')| \, &\le \,C \sum_{j=i}^{n}\delta_{ij}^{\beta^j-1}|(z-z')_{j}\vert.
\end{align}
We can now pick $\delta_{ij}$ for any $i\le j$ in $\llbracket 2,n\rrbracket$ in order to have any contribution associated with the mollification
appearing in \eqref{Proof:Controls_on_flows_mollifier} at a good current scale time. Namely, we would like $\delta_{ij}$ to satisfy
\begin{equation}\label{proof:bilip_control}
    \left| \mathbb{T}^{-1}_{s-t}\left(F(u,z)-F^\delta(u,z)\right) \right| \,
\le \, C(s-t)^{-1},
\end{equation}
for any $u$ in $[t,s]$ and any $z$ in $\R^N$.
Using the mollifier controls \eqref{Proof:Controls_on_flows_mollifier}, it is enough to ask for
\begin{equation}
\sum_{i=2}^n(s-t)^{-\frac{1+\alpha(i-1)}{\alpha}}\sum_{j=i}^n\delta_{ij}^{\beta^j} \, \le \, C(s-t)^{-1}.\label{proof:SUM_COND_DELTA}
\end{equation}
This is true if we fix for example,
\begin{equation}\label{Proof:Controls_on_Flows_Choice_delta}
\delta_{ij} \, = \, (s-t)^{\frac{1+\alpha(i-2)}{\alpha\beta^j}} \quad \text{for $i\le j$ in $\llbracket 2,n\rrbracket$.}
\end{equation}

Next, we would like to show that, for our choice of the regularization parameter $\delta_{ij}$, the mollified drift $F^\delta$ satisfies an \emph{approximate} Lipschitz condition with a constant that, once the drift is integrated, does not yield any additional singularity. Namely, we want to derive the following control:
\begin{equation}
\label{approximateLippourF}
\left|\mathbb{T}_{s-t}^{-1}\left(F^\delta(u,z)-F^\delta(u,z')\right)\right|
\, \le \, C\left[(s-t)^{-\frac 1\alpha}+ (s-t)^{-1}|\mathbb{T}_{s-t}^{-1}(z-z')| \right] .
\end{equation}
To show it, we start noticing that $F_1 $ is H\"older continuous with H\"older index $\beta^1>0$. By Young inequality, it then yields that there exists a positive constant $C$ possibly depending on $\beta^1$ such that
$|z|^{\beta^1} \le C(1+|z|)$ for any $z$ in $\R^N$. It then follows from Equation \eqref{Proof:Controls_on_flows_mollifier1} that
\[\begin{split}
|\mathbb{T}_{s-t}^{-1}\bigl(F^\delta(u,z)-&F^\delta(u,z
')\bigr)| \\
&\le \,
C\Bigl[(s-t)^{-\frac{1}{\alpha}}(1+|(z-z')
|) + \sum_{i=2}^n\sum_{j=i}^n(s-t)^{-\frac{1+\alpha(i-1)}{\alpha}}\delta_{ij}^{\beta^j-1}|(z-z
')_j| \Bigr]\\
&\le\,  C\biggl[(s-t)^{-\frac{1}{\alpha}}+
|\mathbb{T}_{s-t}^{-1}(z-z')|\bigl(1+\sum_{i=2}^
n\sum_{j=i}^n \frac{(s-t)^{j-i}}{\delta_{ij}^{1-\beta^j}}
\bigr)\biggr].
\end{split}\]
Hence, Control \eqref{approximateLippourF} follows from the fact that, from our previous choice of $\delta_{ij}$, one gets
\begin{equation}\label{equilibri}
\frac{(s-t)^{j-i}}{\delta_{ij}^{1-\beta^j}} \, = \, (s-t)^{(j-i)-\frac{1+\alpha(i-2)}{\alpha\beta^j}(1-\beta^j)}\, \le \,  C (s-t)^{-1},
\end{equation}
\textcolor{black}{recalling that we assumed $s-t$ to be small enough and} since from the assumption \eqref{eq:thresholds_beta} on the indexes of H\"older continuity $\beta^j$ for $F$:
\[\beta^j>\frac{1+\alpha(i-2)}{1+\alpha(j-1)} \Leftrightarrow (j-i)-\frac{1+\alpha(i-2)}{\alpha\beta^j}(1-\beta^j) \, > \, -1.\]
We recall that the above inequality should precisely give the natural threshold, namely an exponent $\beta_i^j $ satisfying this condition. The current choice for $\beta^j $
is sufficient to ensure this bound holds for any $i\le j $ and is \textit{sharp} for $i=j$.
We can finally show the bound for the second remainder $\mathcal{I}^2_{s,t}(x,y)$ as given in \eqref{eq:proof_bilip_control_flow2}. It holds that:
\[\begin{split}
|\mathcal{I}^2_{s,t}(x,y)| \, &\le \, \int_{t}^{s}  \left|\mathbb{T}_{s-t}^{-1} \left(F(u,\theta_{u,s}(y))-F(u,\check{\theta}_{u,t}(x))\right)\right| \, du\\
&\le \, \int_t^s \left|\mathbb{T}_{s-t}^{-1}(F(u,\check{\theta}_{u,t}(x))-F^\delta(u,\check{\theta}_{u,t}(x)))\right|\, du \\
&\qquad\qquad\qquad + \int_t^s \left|\mathbb{T}_{s-t}^{-1}(F^\delta(u,\check{\theta}_{u,t}(x))-F^\delta(u,\theta_{u,s}(y)))\right| \, du \\
& \qquad\qquad\qquad\qquad\qquad + \int_t^s \left|\mathbb{T}_{s-t}^{-1}\left(F^\delta(u,\theta_{u,s}(y))-F(u,\theta_{u,s}(y))\right)\right| \, du\\
&=: \,  \mathcal{I}^{21}_{s,t}(x,y)+\mathcal{I}^{22}_{s,t}(x,y)+\mathcal{I}^{23}_{s,t}(x,y).
\end{split}\]
From Control \eqref{Proof:Controls_on_flows_mollifier} with our choice of $\delta_{ij}$,  we easily obtain \textcolor{black}{from Control \eqref{proof:bilip_control} that} there exists a positive constant $C:=C(T)$ such that
\begin{equation}
\label{eq:proof_bilip_control_flow5}
|\mathcal{I}^{21}_{s,t}(x,y)|+|\mathcal{I}^{23}_{s,t}(x,y)| \, \le \, C,
\end{equation}
for any $t\le s$ in $[0,T]$ and $x,y$ in $\R^N$. On the other hand, we exploit \eqref{approximateLippourF} to derive that
\[
|\mathcal{I}^{22}_{s,t}(x,y)| \, \le \, C\left[1+\int_t^s (s-t)^{-1}|\mathbb{T}_{s-t}^{-1}(\check{\theta}_{u,t}(x)-\theta_{u,s}(y))| \,    du\right]\]
for any $t\le s$ in $[0,T]$ and $x,y$ in $\R^N$.
To conclude, we finally derive \eqref{eq:proof_bilip_control_flow2} from the last inequality together with Controls \eqref{eq:proof_bilip_control_flow6}-\eqref{eq:proof_bilip_control_flow5}.

\paragraph{Proof of Lemma \ref{convergence_dirac} (Dirac Convergence of frozen density).}
Fixed $(t,x)$ in $[0,T]\times \R^N$ and a bounded, continuous function $f\colon \R^N\to \R$, we want to show that the following limit
\[
\lim_{\epsilon \to 0}\left| \int_{\R^N} f(y) \tilde{p}^{t+\epsilon,y}(t,t+\epsilon,x,y)\, dy
-f(x) \right| \, = \, 0
\]
holds, uniformly in $t \in [0,T]$.\newline
We start rewriting the argument of the limit in the following way:
\begin{align}
\label{proof:control_deviation}
\int_{\R^N} f(y) \tilde{p}^{t+\epsilon,y}(t,t+\epsilon,\,&x,y)\, dy
-f(x) \\\notag
&= \, \int_{\R^N} f(y) \left[\tilde{p}^{t+\epsilon,y}(t,t+\epsilon,x,y)-\tilde{p}^{t,x}(t,t+\epsilon,x,y)\right] \, dy\\\notag
&\qquad\quad\qquad\qquad\qquad+ \int_{\R^N} f(y) \tilde{p}^{t,x}(t,t+\epsilon,x,y)\, dy-f(x).
\end{align}
By Proposition \ref{prop:Decomposition_Process_X}, we know that the second term in \eqref{proof:control_deviation} tends to zero, uniformly in $t$ in $[0,T]$ (scaling property of the upper bound for the density), when $\epsilon$ goes to zero. We can then focus on the first one. We start splitting the space $\R^N$ in the diagonal/off-diagonal regime associated with our anisotropic dynamics. Namely, we fix $\beta>0$ to be chosen later and we consider the following subsets:
\begin{align*}
   D_1 \, &:= \, \{y \in \R^N\colon \left|\mathbb{T}^{-1}_{\epsilon}(y-\theta_{t+\epsilon,t}(x))\right| \, \le \, \epsilon^{-\beta}\}; \\
    D_2 \, &:= \, \{y \in \R^N\colon \left|\mathbb{T}^{-1}_{\epsilon}(y-\theta_{t+\epsilon,t}(x))\right| \, > \, \epsilon^{-\beta}\},
\end{align*}
where $\mathbb{T}_\epsilon$ was defined in \eqref{eq:def_matrix_T}.
We can then decompose the first term in \eqref{proof:control_deviation} in the following way:
\begin{align}\notag
\Bigr{|}\int_{\R^N} f(y) \bigl[\tilde{p}^{t+\epsilon,y}(t,t+\epsilon,\,&x,y)-\tilde{p}^{t,x}(t,t+\epsilon,x,y)\bigr] \, dy\Bigr{|} \\
&\le \, \Vert f \Vert_\infty \int_{D_1}\left|\tilde{p}^{t+\epsilon,y}(t,t+\epsilon,x,y)-\tilde{p}^{t,x}(t,t+\epsilon,x,y)\right| dy \notag\\
&\qquad \qquad +\Vert f \Vert_\infty\int_{D_2}\left|\tilde{p}^{t+\epsilon,y}(t,t+\epsilon,x,y)-\tilde{p}^{t,x}(t,t+\epsilon,x,y)\right| dy \notag\\
&=:\, \Vert f \Vert_\infty\left(\mathcal{D}_1+\mathcal{D}_2\right)(t,t+\epsilon,x).\label{proof:control_deviation1}
\end{align}
We will follow different approaches to control the two terms $\mathcal{D}_1$, $\mathcal{D}_2$. In the off-diagonal regime $D_2$, the idea is to exploit tail estimates of the single densities while in the diagonal one $D_1$, a more thorough sensibility analysis between the spectral measures and the Fourier transform is needed. Let us consider first the off-diagonal term $\mathcal{D}_2$. We can write that
\begin{align*}
\mathcal{D}_2(t,t+\epsilon,x,y) \, &\le \, \int_{D_2}\left|\tilde{p}^{t+\epsilon,y}(t,t+\epsilon,x,y)\right| + \left|\tilde{p}^{t,x}(t,t+\epsilon,x,y)\right| \, dy\\
&\le \int_{D_2}\frac{1}{\det\mathbb{T}_{\epsilon}}\Big( {\bar p}(1,\mathbb{T}_{\epsilon}^{-1}(x-\theta_{t,t+\epsilon}(y))) + {\bar p}(1,\mathbb{T}_{\epsilon}^{-1}(\theta_{t+\epsilon,t}(x)-y))\Big)dy
\end{align*}
using Proposition \ref{prop:Smoothing_effect} together with Lemma \ref{lemma:identification_theta_m} for the last inequality.
From Lemma \ref{lemma:bilip_control_flow} (to use the \textit{approximate} Lipschitz property of the flows) and introducing
\[\bar D_2 \, := \, \{y \in \R^N\colon \left|\mathbb{T}^{-1}_{\epsilon}(\theta_{t,t+\epsilon}(y)-x)\right| \, > \, \frac 12 \epsilon^{-\beta}\},\]
we thus deduce that for $\epsilon $ small enough we get:
\begin{multline*}
\mathcal{D}_2(t,t+\epsilon,x,y) \, \le \,  \int_{\bar D_2}\frac{1}{\det\mathbb{T}_{\epsilon}} {\bar p}(1,\mathbb{T}_{\epsilon}^{-1}(x-\theta_{t,t+\epsilon}(y)))\, dy\\
+\int_{D_2} \frac{1}{\det\mathbb{T}_{\epsilon}}{\bar p}(1,\mathbb{T}_{\epsilon}^{-1}(\theta_{t+\epsilon,t}(x)-y))\,dy.
\end{multline*}
Using now Equation \eqref{eq:smoothing_effect_frozen_y_GENERIC_FUNCTION} from Corollary \ref{coroll:Smoothing_effect}  for the first integral and the direct change of variable $z=\mathbb{T}^{-1}_\epsilon (y-\theta_{t+\epsilon,t}(x))$ for the second,  we can conclude that
\[\mathcal{D}_2(t,t+\epsilon,x) \,
\le \, C\int_{\R^N}\mathds{1}_{B^c(0,\frac 12\epsilon^{-\beta})}(z)(\check p+\overline{p})(1,z) \, dz,\]
where  $\check p $ is a density enjoying the same integrability properties as $\overline{p}$.
\newline
By dominated convergence theorem, it is easy to notice that $\mathcal{D}_2(t,t+\epsilon,x)$ tends to zero if $\epsilon$ goes to zero, \textcolor{black}{uniformly in the time variable $t$ in } $[0,T]$.\newline
We can now focus on the diagonal term $\mathcal{D}_1$ appearing in \eqref{proof:control_deviation1}. We start recalling from Equation \eqref{eq:definition_density_q} that the density $\tilde{p}^{\omega}(t,s,x,y)$ (for $\omega \in\{(t,x),(t+\epsilon,y)\}$) can be written as
\[
\tilde{p}^\omega(t,t+\epsilon,x,y) \, = \, \frac{\det \mathbb{M}^{-1}_\epsilon}{(2\pi)^N}\int_{\R^N}e^{\mathcal{F}_{t,t+\epsilon}(z,\omega)} \text{exp}\left(-i\langle \mathbb{M}^{-1}_\epsilon(y-\tilde{m}^{\omega}_{t+\epsilon,t}(x)),z\rangle\right) \, dz,
\]
where we have denoted:
\[\mathcal{F}_\epsilon(t,z,\omega) \, := \, \epsilon \int_{0}^{1}\int_{\R^d}\left[\cos\left(\langle z, \widehat{\mathcal{R}}_{v}B\tilde{\sigma}^{\omega}_{u(v)} p\rangle \right)-1\right] \,\nu(dp)dv,\]
with $u(v)=t+\epsilon v$ \textcolor{black}{(cf. notations in \eqref{proof:ref_notations} of Proposition \ref{prop:Decomposition_Process_X})} and $\Phi(p)$ the L\'evy symbol of the process $\{Z_t\}_{t\ge0}$.
We can now consider the two following terms
\begin{align*}
&\mathcal{P}_1(t,t+\epsilon,x,y) \, := \, \frac{\det \mathbb{M}^{-1}_\epsilon}{(2\pi)^N}\int_{\R^N}\left[e^{\mathcal{F}_\epsilon(t,z,t,x)}-e^{\mathcal{F}_\epsilon(t,z,t+\epsilon,y)}\right]e^{-i\langle \mathbb{M}^{-1}_\epsilon(y-\tilde{m}^{t,x}_{t+\epsilon,t}(x)),z\rangle} \, dz\\
&\mathcal{P}_2(t,t+\epsilon,x,y) \\
&\qquad\quad:= \,
\frac{\det \mathbb{M}^{-1}_\epsilon}{(2\pi)^N}\int_{\R^N}e^{\mathcal{F}_\epsilon(t,z,t+\epsilon,y)}\left[e^{-i\langle \mathbb{M}^{-1}_\epsilon(y-\tilde{m}^{t,x}_{t+\epsilon,t}(x)),z\rangle}-e^{-i\langle \mathbb{M}^{-1}_\epsilon(y-\tilde{m}^{t+\epsilon,y}_{t+\epsilon,t}(x)),z\rangle}\right] \, dz \notag
\end{align*}
and decompose $\mathcal{D}_1$ as follows:
\[\mathcal{D}_1 \, = \, \int_{D_1}\left|\mathcal{P}_1(t,t+\epsilon,x,y)\right|+\left|\mathcal{P}_2(t,t+\epsilon,x,y)\right| \, dy.\]
To control the first term $\mathcal{P}_1$, we can exploit a Taylor expansion. Indeed,
\begin{multline*}
   |\mathcal{P}_1(t,t+\epsilon,x,y)| \\
   \le \, \frac{C}{\det \mathbb{M}_\epsilon}\int_{\R^N}\int_0^1\left|\mathcal{F}_\epsilon(t,z,t+\epsilon,y)-\mathcal{F}_\epsilon(t,z,t,x)\right| e^{\lambda\mathcal{F}_\epsilon(t,z,t+\epsilon,y)+(1-\lambda)\mathcal{F}_\epsilon(t,z,t,x)} \, d\lambda dz.
\end{multline*}
We then notice from \eqref{eq:control_Levy_symbol_S} that
\[\mathcal{F}_\epsilon(t,z,\omega) \, \le \, C\epsilon [1-|z|^\alpha],\]
and thus, we obtain that
\[e^{\lambda\mathcal{F}_\epsilon(t,z,t+\epsilon,y)+(1-\lambda)\mathcal{F}_\epsilon(t,z,t,x)} \, \le \, e^{C\epsilon(1-|z|^\alpha)},\]
for some constant $C$ independent from $\lambda$ in $[0,1]$.
\textcolor{black}{From our non-degenerate structure, any linear combination of the symbols remains homogeneous to a non-degenerate symbol.} Thus, we have that
\begin{equation}
\label{proof:control_deviation2}
    |\mathcal{P}_1(t,t+\epsilon,x,y)| \, \le \, \frac{C}{\det \mathbb{M}_\epsilon}\int_{\R^N}\left|\mathcal{F}_\epsilon(t,z,t+\epsilon,y)-\mathcal{F}_\epsilon(t,z,t,x)\right| e^{C\epsilon(1-|z|^\alpha)} \, dz.
\end{equation}
On the other hand,
we can decompose the difference in absolute value in the following way:
\begin{align}
\notag
|\mathcal{F}_\epsilon(t,z,t+\epsilon,y)-&\mathcal{F}_\epsilon(t,z,t,x)| \\ \notag
&\le \, \epsilon \int_{0}^{1}\Bigl{|}
\int_{\R^d}\left[\cos\left(\langle z, \widehat{\mathcal{R}}_{v}B\tilde{\sigma}^{t+\epsilon,y}_{u(v)} p\rangle\right)-\cos\left(\langle z,  \widehat{\mathcal{R}}_{v}B\tilde{\sigma}^{t,x}_{u(v)} p\rangle\right)\right] \,\nu(dp)\Bigr{|}dv \\
&\le \,  \epsilon \int_{0}^{1}\Bigl{|} \left(\Delta^{t,\epsilon,x,y}_s+\Delta^{t,\epsilon,x,y}_l\right)(v,z)\Bigl{|}\, dv,
\label{proof:A6}
\end{align}
where we denoted
\begin{align*}
\Delta^{t,\epsilon,x,y}_s(v,z)\, &= \, \int_{B(0,r_0)}\left[\cos\left(\langle z, \widehat{\mathcal{R}}_{v}B\tilde{\sigma}^{t+\epsilon,y}_{u(v)} p\rangle\right)-\cos\left(\langle z,  \widehat{\mathcal{R}}_{v}B\tilde{\sigma}^{t,x}_{u(v)} p\rangle\right)\right]
Q(p)\, \nu_\alpha(dy); \\
\Delta^{t,\epsilon,x,y}_l(v,z) \, &= \, \int_{B^c(0,r_0)}\left[\cos\left(\langle z, \widehat{\mathcal{R}}_{v}B\tilde{\sigma}^{t+\epsilon,y}_{u(v)} p\rangle\right)-\cos\left(\langle z,  \widehat{\mathcal{R}}_{v}B\tilde{\sigma}^{t,x}_{u(v)} p\rangle\right)\right]
Q(p)\, \nu_\alpha(dp),
\end{align*}
with $r_0$ defined in assumption [\textbf{ND}]. The term $\Delta^{t,\epsilon,x,y}_l$ involving the large jumps can be easily controlled using that $\sup_{p\in \R^d}Q(p)<\infty$:
\begin{align} \notag
|\Delta^{t,\epsilon,x,y}_l(v,z)| \, &\le \, \int_{B^c(0,r_0)}\left|\cos\left(\langle z, \widehat{\mathcal{R}}_{v}B\tilde{\sigma}^{t+\epsilon,y}_{u(v)} p\rangle \right)-\cos\left(\langle z,  \widehat{\mathcal{R}}_{v}B\tilde{\sigma}^{t,x}_{u(v)} p\rangle \right)\right|
\,\nu_\alpha(dp) \\
&\le \, C. \label{proof:A5}
\end{align}
To bound the term $\Delta^{t,\epsilon,x,y}_s$ associated with the small jumps, we want to exploit instead that $Q$ is Lipschitz continuous on $B(0,r_0)$. For this reason, we write that
\begin{align}
\notag
|\Delta^{t,\epsilon,x,y}_s&(v,z)| \\ \notag
&\le \, \left| \int_{B(0,r_0)}\left[\cos\left(\langle z, \widehat{\mathcal{R}}_{v}B\tilde{\sigma}^{t+\epsilon,y}_{u(v)} p\rangle \right)-\cos\left(\langle z,  \widehat{\mathcal{R}}_{v}B\tilde{\sigma}^{t,x}_{u(v)} p\rangle \right)\right][Q(p)-Q(0)]
\,\nu_\alpha(dp) \right|\\\notag
&\qquad\qquad\,\,+ \left| \int_{B(0,r_0)}\left[\cos\left(\langle z, \widehat{\mathcal{R}}_{v}B\tilde{\sigma}^{t+\epsilon,y}_{u(v)} p\rangle\right)-\cos\left(\langle z,  \widehat{\mathcal{R}}_{v}B\tilde{\sigma}^{t,x}_{u(v)} p\rangle \right)\right]Q(0)
\,\nu_\alpha(dp) \right| \\\label{proof:A1}
&=:\, \left(\Delta^{t,\epsilon,x,y}_{s,1}+\Delta^{t,\epsilon,x,y}_{s,2}\right)(v,z).
\end{align}
Since $Q$ and the cosine function are Lipschitz continuous in a neighborhood of $0$, we have that
\begin{equation}
\label{proof:A2}
\begin{split}
 \Delta^{t,\epsilon,x,y}_{s,1}(v,z)\, &\le \, C \int_{B(0,r_0)}|p||z|\left|\widehat{\mathcal{R}}_{v}B\tilde{\sigma}^{t+\epsilon,y}_{u(v)}p -\widehat{\mathcal{R}}_{v}B\tilde{\sigma}^{t,x}_{u(v)}p\right|
\,\nu_\alpha(dp) \\
&\le \, C \int_{B(0,r_0)}|p||z|\left|\sigma(u(v),\theta_{u(v),t+\epsilon}(y))p - \sigma(u(v),\theta_{u(v),t}(x))p\right|
\,\nu_\alpha(dp)\\
&\le \, C|z| \int_{B(0,r_0)}|p|^2 \,\nu_\alpha(dp) \, \le \, C|z|,
\end{split}
\end{equation}
where in the last step, we used that the diffusion coefficient $\sigma$ is bounded (cf. assumption [\textbf{UE}]).\newline
The control of the other term
$\Delta^{t,\epsilon,x,y}_{s,2}$ now follows from the classical characterization of the L\'evy symbol of a non-degenerate $\alpha$-stable process (see e.g. \cite{book:Sato99}). Indeed,
\[
\begin{split}
\Delta^{t,\epsilon,x,y}_{s,2}(v,z) \, &= \,\left| \int_{\R^d}\left[\cos\left(\langle z,\widehat{\mathcal{R}}_{v}B\tilde{\sigma}^{t+\epsilon,y}_{u(v)}p\rangle\right)-1\right]-\left[\cos\left(\langle z,\widehat{\mathcal{R}}_{v}B\tilde{\sigma}^{t,x}_{u(v)}p\rangle\right)-1\right] \, \nu(dp)\right| \\
&\le \, C\int_{\mathbb{S}^{d-1}}\left|\left|\langle z,\widehat{\mathcal{R}}_{v}B\tilde{\sigma}^{t+\epsilon,y}_{u(v)}s\rangle\right|^\alpha-\left|\langle z,\widehat{\mathcal{R}}_{v}B\tilde{\sigma}^{t,x}_{u(v)}s\rangle\right|^\alpha\right| \, \mu(ds).
\end{split}
\]
We now exploit the $\beta^1$-H\"older regularity in space of the diffusion coefficient $\sigma$ to show that
\begin{equation}
\label{proof:A3}
\begin{split}
    \Delta^{t,\epsilon,x,y}_{s,2}(v,z) \,&\le \, C
|z|^\alpha \left|\theta_{u(v),t+\epsilon}(y)-\theta_{u(v),t}(x)\right|^{\beta^1(\alpha\wedge 1)}  \\
&\le \, C |z|^\alpha \left[|y-\theta_{t+\epsilon,t}(x)|^{\beta^1}+\epsilon^{\beta^1}\right],
\end{split}
\end{equation}
where in the last step we used that $\alpha>1$ and the approximate Lipschitz property of the flow (cf. Lemma \ref{lemma:bilip_control_flow} \textcolor{black}{up to a normalization}, see also Lemma $1.1$ in \cite{Menozzi:Pesce:Zhang21}).\newline
We can now use Controls \eqref{proof:A2}-\eqref{proof:A3} in Equation \eqref{proof:A1} to show that
\begin{equation}
\label{proof:A4}
|\Delta^{t,\epsilon,x,y}_s(v,z)| \, \le \, C\left(|z|+   \epsilon^{\beta^1} |z|^\alpha+ |y-\theta_{t+\epsilon,t}(x)|^{\beta^1}|z|^\alpha\right).
\end{equation}
Similarly, Controls \eqref{proof:A4}-\eqref{proof:A5} with Equation \eqref{proof:A6} allow us to conclude that
\begin{equation}
|\mathcal{F}_\epsilon(t,z,t+\epsilon,y)-\mathcal{F}_\epsilon(t,z,t,x)| \, \le \, C \epsilon \left(1+|z|+   \epsilon^{\beta^1} |z|^\alpha+ |y-\theta_{t+\epsilon,t}(x)|^{\beta^1}|z|^\alpha\right).
\end{equation}
We can now go back to Equation \eqref{proof:control_deviation2}. Changing variable and integrating over $z$, we find that
\[
\begin{split}
 |\mathcal{P}_1(t,t+\epsilon,x,y)| &\, \le \, \frac{C\epsilon}{\det \mathbb{M}_\epsilon}\int_{\R^N}\left(1+|z|+    \epsilon^{\beta^1}|z|^\alpha+ |y-\theta_{t+\epsilon,t}(x)|^{\beta^1}|z|^\alpha\right) e^{C\epsilon(1-|z|^\alpha)} \, dz \\
 &\le \, \frac{C}{\det \mathbb{T}_\epsilon}\int_{\R^N}\left(\epsilon+\epsilon^{\frac{\alpha-1}{\alpha}}|\tilde{z}|+  \epsilon^{\beta^1}  |\tilde{z}|^\alpha+ |y-\theta_{t+\epsilon,t}(x)|^{\beta^1}|\tilde{z}|^\alpha\right) e^{C(1-|\tilde{z}|^\alpha)} \, d\tilde{z}\\
&\le \, \frac{C}{\det \mathbb{T}_\epsilon}\left(\epsilon^{(1-\frac 1\alpha)\wedge \beta^1} +|y-\theta_{t+\epsilon,t}(x)|^{\beta^1}\right)
\end{split}
\]
where we recall that $\mathbb{T}_t=t^{1/\alpha}\mathbb{M}_t$.\newline
To conclude, we apply the change of variable $\tilde{y}=y-\theta_{t+\epsilon,t}(x)$:
\[
\begin{split}
\int_{D_1}|\mathcal{P}_1(t,t+\epsilon,x,y)| \, dy \, &\le \,\frac{C}{\det \mathbb{T}_\epsilon} \int_{D_1}   \left[|y-\theta_{t+\epsilon,t}(x)|^{\beta^1}+\epsilon^{(1-\frac 1\alpha)\wedge \beta^1}\right] \, dy \\
&= \, C\int_{|\tilde{y}|\le \epsilon^{-\beta}}   \left[|\mathbb{T}_\epsilon \tilde{y}|^{\beta^1}+\epsilon^{(1-\frac 1\alpha)\wedge \beta^1}\right] \, d\tilde{y} \\
&\le \, C[\epsilon^{\beta^1/\alpha-\beta(N+\beta^1)}+\epsilon^{(1-\frac 1\alpha)\wedge \beta^1-\beta N}].
\end{split}
\]
The above control then tends to zero letting $\epsilon$ go to zero, if we choose $\beta$ such that
\[0\,<\, \beta \,<\, \frac{\beta^1}{\alpha(N+\beta^1)}\wedge \frac{(1-\frac 1\alpha)\wedge \beta^1}{N}. \]
To control the second term $\mathcal{P}_2$, we use again Control \eqref{eq:control_Levy_symbol_S} and a Taylor expansion to write, similarly to above, that
\begin{align}\notag
|\mathcal{P}_2(t,&t+\epsilon,x,y) | \\ \notag
&\le \,
\frac{C}{\det \mathbb{M}_\epsilon}\int_{\R^N}e^{C\epsilon (1-|z|^\alpha)}\left|\langle \mathbb{M}^{-1}_\epsilon(y-\tilde{m}^{t,x}_{t+\epsilon,t}(x)),z\rangle-\langle \mathbb{M}^{-1}_\epsilon(y-\tilde{m}^{t+\epsilon,y}_{t+\epsilon,t}(x)),z\rangle \right| \, dz   \\
&\le \, \frac{C}{\det \mathbb{T}_\epsilon}\left|\mathbb{T}^{-1}_\epsilon\left(\theta_{t+\epsilon,t}(x)-\tilde{m}^{t+\epsilon,y}_{t+\epsilon,t}(x)\right)\right|,\label{proof:control_P2}
\end{align}
where in the last passage we used Lemma \ref{lemma:identification_theta_m}. To bound the above right-hand side, we now exploit Corollary \ref{coroll:control_error_flow} to show that
\[
\begin{split}
|\mathcal{P}_2(t,t+\epsilon,x,y) | \,
\le \, C \epsilon^{\frac{1}{\alpha} \wedge \zeta}\frac{1}{\det \mathbb{T}_\epsilon}\left(1+|\mathbb{T}_{\epsilon}^{-1}(\theta_{t+\epsilon,t}(x)-y)|\right).
\end{split}
\]
Similarly to above, we can then apply a change of variables:
\[
\int_{D_1}|\mathcal{P}_2(t,t+\epsilon,x,y) | \, dy \, \le \, C \epsilon^{\frac{1}{\alpha} \wedge \zeta}\int_{|z|\le\epsilon^{-\beta}}\left(1+|z| \right) \, dz.
\]
We can then notice again that the above control tends to zero letting $\epsilon$ goes to zero, if we choose $\beta$ small enough.

\paragraph{Proof of Lemma \ref{prop:convergence_LpLq}.}
As in the previous Lemma \ref{convergence_dirac}, we want to show the following limit:
\[
\lim_{\epsilon \to 0} \Vert
I_\epsilon f -f \Vert_{L^{p}_tL^q_x} \, = \, 0,
\]
for some $p\in (1,+\infty)$, $q\in (1,+\infty)$ and $f$ in $C_c^{1,2}([0,T)\times \R^N)$.
We start writing that
\[
\|I_\epsilon f-f\|^p_{L^{p}_tL^q_x} \, = \, \int_0^T \Vert I_\epsilon f(t,\cdot)-f(t,\cdot)\Vert_{L^q}^{p}\, dt.\]
We then notice that, up to a middle point-type argument, the indicator function in the definition \eqref{eq:def_f_epsilon} of $I_\epsilon f$ can be easily controlled. We can now write that
\begin{align}
\notag
\Vert I_\epsilon f(t,\cdot)-f(&t,\cdot)\Vert_{L^p}^{p}\, = \, \int_{\R^N} \left|\int_{\R^N}f(t+\epsilon,y) \tilde{p}^{t+\epsilon,y}(t,t+\epsilon,x,y) \, dy - f(t,x)\right|^{p}  dx \\ \notag
&\le C \Bigl(\int_{\R^N} \left|\int_{\R^N}f(t+\epsilon,y) \tilde{p}^{t+\epsilon,y}(t,t+\epsilon,x,y) \, dy - f(t+\epsilon,\theta_{t+\epsilon,t}(x))\right|^{p} dx\\ \notag
&\,\,\qquad\qquad\qquad\qquad\qquad\qquad\qquad + \int_{\R^N} |f(t+\epsilon,\theta_{t+\epsilon,t}(x))-f(t,x)|^p \, dx\Bigr)\\
&=:\, C\left(\mathcal{I}+\mathcal{I}'\right)(\epsilon,t).\label{proof:eq:LpLq_convergence}
\end{align}
Since $f$ is smooth  and with compact support in time and space, it follows immediately that $\mathcal{I}'(\epsilon,t)$ tends to zero if $\epsilon$ goes to zero, \textcolor{black}{thanks to the bounded convergence Theorem}.\newline
We can then focus on the first term $\mathcal{I}(\epsilon,t) $. We start splitting it in the following way:
\[
\begin{split}
\mathcal{I}(\epsilon,t) \, &\le \,  C\Bigl(\int_{\R^N}\left|\int_{\R^N}\left[f(t+\epsilon,y)-f(t+\epsilon,\theta_{t+\epsilon,t}(x))\right] \tilde{p}^{t+\epsilon,y}(t,t+\epsilon,x,y) \,  dy \right|^p dx\\
&\quad + \int_{\R^N}\left|f(t+\epsilon,\theta_{t+\epsilon,t}(x)) \int_{\R^N} \left[\tilde{p}^{t+\epsilon,y}(t,t+\epsilon,x,y)- \tilde{p}^{t,x}(t,t+\epsilon,x,y)\right]dy\right|^p dx\Bigr)\\
&=: \, C\left(\mathcal{I}_1+\mathcal{I}_2\right)(\epsilon,t),
\end{split}
\]
where we used that $\tilde{p}^{t,x}(t,s,x,y)$ is indeed a \textit{true} density with respect to $y$. The second term $\mathcal{I}_2(\epsilon,t)$ already appeared in the proof of Lemma \ref{convergence_dirac} (Dirac Convergence of frozen density) \textcolor{black}{(cf. term $\mathcal{D}_2$ in \eqref{proof:control_deviation1})} and a similar  analysis readily gives that $\mathcal{I}_2(\epsilon, t)\overset{\epsilon\rightarrow 0}{\longrightarrow}0 $.\newline
To control instead the first term $\mathcal{I}_1(\epsilon,t)$, we decompose the whole space $\R^N$ into $\Delta_1\cup \Delta_2$ given by
\begin{align*}
    \Delta_1 \, &:= \, \{x \in \R^N \colon |\theta_{t+\epsilon, t}(x)-\text{supp}[f(t+\epsilon,\cdot)]|\le 1\}; \\
    \Delta_2 \, &:= \,  \{x \in \R^N \colon |\theta_{t+\epsilon, t}(x)-\text{supp}[f(t+\epsilon,\cdot)]| > 1\}.
\end{align*}
Using Proposition \ref{prop:Smoothing_effect} with $(\tau,\xi)=(t+\epsilon,y)$, we write that
\[\begin{split}
\mathcal{I}_1(\epsilon,t) \, &\le \, \int_{\R^N}\left(\int_{\R^N}|f(t+\epsilon,y)-f(t+\epsilon,\theta_{t+\epsilon,t}(x))| \frac{\overline{p}\left(1,\mathbb{T}^{-1}_\epsilon(y-\tilde{m}^{t+\epsilon,y}_{t+\epsilon,t}(x))\right)}{\det \mathbb{T}_\epsilon} \,  dy  \right)^p dx \\
&\le \,\int_{\Delta_1}\left(\int_{\R^N}|f(t+\epsilon,y)-f(t+\epsilon,\theta_{t+\epsilon,t}(x))| \frac{\overline{p}\left(1,\mathbb{T}^{-1}_\epsilon(y-\tilde{m}^{t+\epsilon,y}_{t+\epsilon,t}(x))\right)}{\det \mathbb{T}_\epsilon} \,  dy  \right)^p dx \\
&\quad + \int_{\Delta_2}\left(\int_{\R^N}|f(t+\epsilon,y)-f(t+\epsilon,\theta_{t+\epsilon,t}(x))| \frac{\overline{p}\left(1,\mathbb{T}^{-1}_\epsilon(y-\tilde{m}^{t+\epsilon,y}_{t+\epsilon,t}(x))\right)}{\det \mathbb{T}_\epsilon} \,  dy  \right)^p dx\\
&=: \, \left(\mathcal{I}_{11}+\mathcal{I}_{12}\right)(\epsilon,t).
\end{split}
\]
To control $\mathcal{I}_{11}$, we start noticing that $f$ is H\"older continuous with a H\"older exponent $\gamma<\alpha$ in $(0,1]$, since it has a compact support. Moreover, $\Delta_1$ is a bounded set (uniformly in $\epsilon$). Then, from Lemma \ref{lemma:identification_theta_m} (cf.\ Equation \eqref{eq:identification_theta_m1}), Lemma \ref{lemma:bilip_control_flow} and Corollary \ref{coroll:Smoothing_effect},
\[
\begin{split}
\mathcal{I}_{11}(\epsilon,t) \, &\le \,C \int_{\Delta_1}\left(\int_{\R^N}|y-\theta_{t+\epsilon,t}(x)|^\gamma \frac{\overline{p}\left(1,\mathbb{T}^{-1}_\epsilon(y-\tilde{m}^{t+\epsilon,y}_{t+\epsilon,t}(x))\right)}{\det \mathbb{T}_\epsilon} \,  dy  \right)^p dx \\
&\le \,C \epsilon^{p\gamma/\alpha}\int_{\Delta_1}\left(\int_{\R^N}|\mathbb{T}^{-1}_\epsilon(y-\theta_{t+\epsilon,t}(x))|^\gamma \frac{\overline{p}\left(1,\mathbb{T}^{-1}_\epsilon(\theta_{t,t+\epsilon}(y)-x)\right)}{\det \mathbb{T}_\epsilon} \,  dy  \right)^p dx \\
&\le \, C \epsilon^{p\gamma/\alpha}\int_{\Delta_1}\left(\int_{\R^N}\left[|\mathbb{T}^{-1}_\epsilon(\theta_{t,t+\epsilon}(y)-x)|^\gamma+1\right] \frac{\overline{p}\left(1,\mathbb{T}^{-1}_\epsilon(\theta_{t,t+\epsilon}(y)-x)\right)}{\det \mathbb{T}_\epsilon} \,  dy  \right)^p dx \\
&\le \, C\epsilon^{p\gamma/\alpha}.
\end{split}
\]
To control instead $\mathcal{I}_{12}$ we firstly notice that if $x$ is in $\Delta_2$, then, $\theta_{t+\epsilon,t}(x)$ is not in the support of $f$. Thus,
\[\begin{split}
\mathcal{I}_{12}(\epsilon,t) \, &= \, \int_{\Delta_2}\left(\int_{\R^N}|f(t+\epsilon,y)-f(t+\epsilon,\theta_{t+\epsilon,t}(x))| \frac{\overline{p}\left(1,\mathbb{T}^{-1}_\epsilon(y-\tilde{m}^{t+\epsilon,y}_{t+\epsilon,t}(x))\right)}{\det \mathbb{T}_\epsilon} \,  dy  \right)^p dx\\
 &\le \, \int_{\Delta_2}\left(\int_{\text{supp}f}|f(t+\epsilon,y)| \frac{\overline{p}\left(1,\mathbb{T}^{-1}_\epsilon(\theta_{t,t+\epsilon}(y)-x)\right)}{\det \mathbb{T}_\epsilon} \,  dy  \right)^p dx\\
 &\le \,   \Vert f\Vert^p_\infty\int_{\Delta_2}\left(\int_{\text{supp}f} \frac{\overline{p}\left(1,\mathbb{T}^{-1}_\epsilon(\theta_{t,t+\epsilon}(y)-x)\right)}{\det \mathbb{T}_\epsilon} \,  dy\right)^{p-1+1} dx \\
  &\le \, C\int_{\text{supp}f}\int_{\Delta_2} \frac{\overline{p}\left(1,\mathbb{T}^{-1}_\epsilon(\theta_{t,t+\epsilon}(y)-x)\right)}{\det \mathbb{T}_\epsilon} \, dxdy,
\end{split}
\]
where in the last step we used that, from Corollary \ref{coroll:Smoothing_effect}
\[\left(\int_{\text{supp}f} \frac{\overline{p}\left(1,\mathbb{T}^{-1}_\epsilon(\theta_{t,t+\epsilon}(y)-x)\right)}{\det \mathbb{T}_\epsilon} \,  dy\right)^{p-1} \, \le \, \left(\int_{\R^N} \frac{\overline{p}\left(1,\mathbb{T}^{-1}_\epsilon(\theta_{t,t+\epsilon}(y)-x)\right)}{\det \mathbb{T}_\epsilon} \,  dy\right)^{p-1} \, \le  \, C_p.\]
We notice now that for any $y$ in $\text{supp}f$ and any $x$ in $\Delta_2$, we have that $|y-\theta_{t+\epsilon,t}(x))|\ge 1$. Exploiting Corollary \ref{coroll:Smoothing_effect} and Lemma \ref{lemma:bilip_control_flow}, we write that
\[\begin{split}
\mathcal{I}_{12}(\epsilon,t) \, &\le \,    \int_{\text{supp}f}\int_{\Delta_2} |y-\theta_{t+\epsilon,t}(x)| \frac{\overline{p}\left(1,\mathbb{T}^{-1}_\epsilon(\theta_{t,t+\epsilon}(y)-x)\right)}{\det \mathbb{T}_\epsilon} \, dxdy \\
&\le \, C\epsilon^{\frac{1}{\alpha}}\int_{\text{supp}f}\int_{\Delta_2}|\mathbb{T}^{-1}_\epsilon(y-\theta_{t+\epsilon,t}(x))| \frac{\overline{p}\left(1,\mathbb{T}^{-1}_\epsilon(\theta_{t,t+\epsilon}(y)-x)\right)}{\det \mathbb{T}_\epsilon} \, dxdy \\
&\le \, C\epsilon^{\frac{1}{\alpha}}\int_{\text{supp}f}\int_{\R^N}\left[|\mathbb{T}^{-1}_\epsilon(\theta_{t,t+\epsilon}(y)-x)|+1\right] \frac{\overline{p}\left(1,\mathbb{T}^{-1}_\epsilon(\theta_{t,t+\epsilon}(y)-x)\right)}{\det \mathbb{T}_\epsilon} \, dxdy \\
&\le \, C\epsilon^{\frac{1}{\alpha}}\int_{\text{supp}f}\int_{\R^N}\left[|z|+1\right] \overline{p}\left(1,z\right) \, dzdy \\
&\le \, C\epsilon^{\frac{1}{\alpha}}.
\end{split}
\]
Knowing the convergence of $\mathcal{I}(\epsilon,t)$ and $\mathcal{I}'(\epsilon,t)$ to zero, we can finally conclude the proof using the dominated convergence theorem in \eqref{proof:eq:LpLq_convergence}.

\subsection{Controls associated with the change of variable}

\subsubsection{Proof of Corollary \ref{coroll:Smoothing_effect}}
We first concentrate on the proof of Control \eqref{eq:smoothing_effect_frozen_y}. We start exploiting the decomposition of $\overline{p}(t,z)$ in terms of small and large jumps, as in \eqref{proof_decomposition_p_segnato}, to rewrite the left-hand side of Equation \eqref{eq:smoothing_effect_frozen_y} in the following way:
\begin{align*}
I(s,t,x)\, &:= \,
\int_{\R^N} \frac{|\mathbb{T}_{s-t}^{-1}(\theta_{t,s}(y)-x)|^\gamma}{\det\mathbb T_{s-t}}\bar p(1, \mathbb{ T}_{s-t}^{-1}(\theta_{t,s}(y)-x))\, dy  \\
&=\, \int_{\R^N} \frac{|\mathbb{T}_{s-t}^ {-1}(\theta_{t,s}(y)-x)|^\gamma}{\det \mathbb T_{s-t}}\int_{\R^N}p_{\overline{M}}(1, \mathbb{ T}_{s-t}^{-1}(\theta_{t,s}(y)-x)-w)\overline{P}_1(dw)dy,
\end{align*}
Then, the Fubini Theorem and the definition of $p_{\overline{M}}$ in \eqref{proof:control_pM_segnato} immediately imply that
\begin{align*}
I(s,t,x)\, &= \, \int_{\R^N} \int_{\R^N} \frac{|\mathbb{T}_{s-t}^ {-1}(\theta_{t,s}(y)-x)|^\gamma}{\det \mathbb T_{s-t}}p_{\overline{M}}(1, \mathbb{ T}_{s-t}^{-1}(\theta_{t,s}(y)-x)-w)\, dy\overline{P}_1(dw)\\
&\le \, C\int_{\R^N} \int_{\R^N}\det \mathbb T_{s-t}^{-1}\frac{ \left[|\mathbb{T}_{s-t}^ {-1}(\theta_{t,s}(y))-x-w|^\gamma+|w|^\gamma\right]}{\left[1+|\mathbb{ T}_{s-t}^{-1}(\theta_{t,s}(y)-x)-w|\right]^{N+2}}\, dy\overline{P}_1(dw).
\end{align*}
To conclude, it is now enough to show that for any $M> N+1$, there exists $C:=C(M)$ such that
\begin{equation}
\label{proof:control_smoothing2}
\int_{\R^N}  \frac{\det\mathbb{T}^{-1}_{s-t}}{\left[1+|\mathbb{T}_{s-t}^{-1}(\theta_{t,s}(y)-x) -w|\right]^M}\, dy\, \le \, C.
\end{equation}
Indeed, it would follow from Control \eqref{proof:control_smoothing2} that
\begin{align*}
    I(t,s,x) \, &\le \, C\int_{\R^N} \int_{\R^N} \frac{\det \mathbb T_{s-t}^{-1}}{\left[1+|\mathbb{ T}_{s-t}^{-1}(\theta_{t,s}(y)-x)-w|\right]^{N+2-\gamma}}\, dy\overline{P}_1(dw) \\
    &\qquad\qquad +\int_{\R^N} \int_{\R^N} \left[1+|w|^\gamma\right]\frac{\det \mathbb T_{s-t}^{-1}}{\left[1+|\mathbb{ T}_{s-t}^{-1}(\theta_{t,s}(y)-x)-w|\right]^{N+2}}\, dy\overline{P}_1(dw) \\
    &\le \,C\int_{\R^N} \left[1+|w|^\gamma\right]\overline{P}_1(dw) \, \le \, C.
\end{align*}
In order to show Control \eqref{proof:control_smoothing2}, we start noticing that it would be enough to apply the change of variable $\tilde{y}=\mathbb{T}_{s-t}^{-1}(x-\theta_{t,s}(y)) -w$ and then, to control the Jacobian matrix of the transformation. Unfortunately, our coefficients are not smooth enough in order to follow this kind of reasoning. Indeed, the drift $F$ is only H\"older continuous.\newline
As done already in in the proof of Lemma \ref{lemma:bilip_control_flow}, we firstly need to regularize $F$ through a multi-scale mollification procedure. Namely, we reintroduce the mollified drift $F^\delta:=(F^\delta_1,\dots,F^\delta_n)$ similarly to what we did in Equation \eqref{eq:multi_scale_mollification}. However we modify a bit the mollifying parameters and set
\begin{equation}\label{Proof:Controls_on_Flows_Choice_delta_LOC_LEMMA}
\delta_{ij} \, = \, \bar C(s-t)^{\frac{1+\alpha(j-2)}{\alpha\beta^j}} \quad \text{for $2\le i\le j\le n$,}
\end{equation}
for a constant $\bar C$ meant to be large enough. We also mollify the first component $ F_1$ at a macro scale, i.e. $\delta_{1j}=C_1$, with $C_1$ large enough as well.

In particular, this choice of parameters gives that the controls
\eqref{eq:proof_bilip_control_flow6}, \eqref{proof:bilip_control} and \eqref{approximateLippourF} hold again.
\newline
We can now define the mollified flow $\theta^\delta_{t,s}(y)$ associated with the drift $F^\delta$ given by
\begin{equation}\label{proof:definition_theta_delta}
\theta_{t,s}^\delta(y) \, =\, y-\int_t^s \left[A_u\theta_{u,s}^\delta+ F^\delta(u,\theta_{u,s}^\delta(y))\right] \,  du.
\end{equation}
Denoting now, for brevity, \[\Delta^\delta\theta_{u,s}(y) \, :=\, \theta_{u,s}(y)-\theta_{u,s}^\delta(y),\]
it is not difficult to check from the Gr\"onwall Lemma and Controls \eqref{eq:proof_bilip_control_flow6}, \eqref{proof:bilip_control} and \eqref{approximateLippourF} that
\begin{align} \label{SENSI_DELTA_THETA}
|\mathbb T_{s-t}^{-1}(&\theta_{t,s}(y)-\theta_{t,s}^\delta(y))| \, \le \, \left|\int_t^s \mathbb{T}_{s-t}^{-1}\left[A_u(\Delta^\delta\theta_{u,s}(y))+ F(u,\theta_{u,s}(y))-F^\delta(u,\theta_{u,s}^\delta(y))\right] \,  du\right| \notag\\
&\le \, \int_t^s \left|\mathbb{T}_{s-t}^{-1}A_u(\Delta^\delta\theta_{u,s}(y))\right|\, du+\int_t^s\left|\mathbb{T}_{s-t}^{-1}\left(F(u,\theta_{u,s}(y))-F^\delta(u,\theta_{u,s}(y))\right)\right| \,  du\notag\\
&\qquad \qquad\qquad \qquad\qquad \qquad\qquad + \int_t^s\left|\mathbb{T}_{s-t}^{-1}\left(F^\delta(u,\theta_{u,s}(y))-F^\delta(u,\theta_{u,s}^\delta(y))\right)\right| \,  du \notag\\
&\le C_0,
\end{align}
for some positive constant $C_0$. \newline
Exploiting now Control \eqref{SENSI_DELTA_THETA}, we firstly notice that for $C\ge 2C_0$,
\[
\begin{split}
C+|\mathbb{T}_{s-t}^{-1}(x-\theta_{t,s}(y)) -w|\, &\ge \, C+|\mathbb{T}_{s-t}^{-1}(x-\theta^\delta_{t,s}(y)) -w| - |\mathbb{T}_{s-t}^{-1}(\theta_{t,s}(y)-\theta_{t,s}^\delta(y))| \\
&\ge \, C_0+|\mathbb{T}_{s-t}^{-1}(x-\theta^\delta_{s,t}(y)) -w|
\end{split}\]
and we then use it to write that
\begin{align}
\int_{\R^N} \notag \frac{\det\mathbb{T}^{-1}_{s-t}}{\left(1+|\mathbb{T}_{s-t}^{-1}(x-\theta_{t,s}(y)) -w|\right)^M}\, dy\, &\le \,C \int_{\R^N}  \frac{\det\mathbb{T}^{-1}_{s-t}}{\left(1+|\mathbb{T}_{s-t}^{-1}(x-\theta^\delta_{t,s}(y)) -w|\right)^M}\, dy  \\
&= \, C\int_{\R^N}  \frac{1}{\left(1+|\tilde{y}|\right)^M}\frac{1}{\det  J_{t,s}^\delta(\tilde{y})}\, dy
\label{proof:sensibility1}
\end{align}
where in the last step we used the change of variables $\tilde{y}=\mathbb{T}_{s-t}^{-1}(x-\theta^\delta_{t,s}(y)) -w$ and denoted by $\textcolor{black}{J_{t,s}^\delta}(\tilde{y})$ the Jacobian matrix of $y\to \theta^\delta_{t,s}(y)$.\newline
It is now clear that the last term  in  \eqref{proof:sensibility1} is indeed controlled by a constant $C$, if we show the existence of a positive constant $c$, independent from $y$ in $\R^N$, $t<s$ in $[0,T]$ and $\delta$ , such that
\begin{equation}
\label{proof:sensibility3}
    |\det J_{t,s}^\delta(y)| \, \ge \, c >0.
\end{equation}
This is precisely the result provided by Lemma \ref{LEMMA_FOR_DET} below.
From the previous computations it is clear that \eqref{eq:smoothing_effect_frozen_y} holds.

Let us now turn to the proof of Control \eqref{eq:smoothing_effect_frozen_y_GENERIC_FUNCTION}. Following the previous approach, we can write
\begin{align*}
&\int_{\{|\mathbb{T}^{-1}_{s-t}(\theta_{t,s}(y)-x)|\ge K\}} \frac 1{\det \mathbb{T}_{s-t}}\overline{p}(1,\mathbb{T}^{-1}_{s-t}(\theta_{t,s}(y)-x)) \, dy \\
 &\le \, C \int_{\{|\mathbb{T}^{-1}_{s-t}(\theta_{t,s}^\delta(y)-x)|\ge K-|\mathbb{T}^{-1}_{s-t}(\Delta^\delta\theta_{u,s}(y))|\}}  \int_{\R^N}  \frac{\det\mathbb{T}^{-1}_{s-t}}{\left(1+|\mathbb{T}_{s-t}^{-1}(x-\theta^\delta_{s,t}(y)) -w|\right)^M}\overline{P}_1(dw) dy\\
 \le& \, C \int_{\{|\mathbb{T}^{-1}_{s-t}(\theta_{t,s}^\delta(y)-x)|\ge K-C_0\}}  \int_{\R^N}  \frac{\det\mathbb{T}^{-1}_{s-t}}{\left(1+|\mathbb{T}_{s-t}^{-1}(x-\theta^\delta_{s,t}(y)) -w|\right)^M}\overline{P}_1(dw) dy,
\end{align*}
exploiting also \eqref{SENSI_DELTA_THETA} for the last inequality. Using now the Fubini Theorem and the change of variables  $z=\mathbb{T}_{s-t}^{-1}(x-\theta^\delta_{s,t}(y))$, we derive from \eqref{proof:sensibility3} that
\begin{align*}
\int_{\{|{\mathbb{T}^{-1}_{s-t}(\theta_{t,s}(y)-x)}|\ge K\}} \frac 1{\det \mathbb{T}_{s-t}}&\overline{p}(1,\mathbb{T}^{-1}_{s-t}(\theta_{t,s}(y)-x)) \, dy \\
&\le \, C \int_{\R^N}   \int_{\{|z|\ge K-C_0\}} \frac{1}{\left(1+|z -w|\right)^M} \, dz \overline{P}_1(dw)\\
&=:\, C\int_{\{|z|\ge \frac K2|\}}\check p(1,z)\, dz,
\end{align*}
where $\check p $ is a density satisfying the same integrability properties as $\bar p $ assuming as well $K $ large enough.
Thus \eqref{eq:smoothing_effect_frozen_y_GENERIC_FUNCTION} holds and the proof of Corollary \ref{coroll:Smoothing_effect} is now complete.

\subsubsection{Jacobian of the mollified system}
This is a technical part dedicated to the proof of control \eqref{proof:sensibility3} appearing in the proof of key Corollary \ref{coroll:Smoothing_effect} which precisely gives the expected smoothing effect of the frozen density where the freezing parameters also correspond to the integration variable.
\begin{lemma}[Control of the determinant for the change of variable]\label{LEMMA_FOR_DET}
Let  $\theta^\delta_{t,s}(y)$ denote the mollified flow associated with the drift $F^\delta$ where the mollifying parameter $\delta$ has the form \eqref{Proof:Controls_on_Flows_Choice_delta_LOC_LEMMA}. Its dynamics writes:
\begin{equation*}
\theta_{t,s}^\delta(y) \, =\, y-\int_t^s \left[A_u\theta_{u,s}^\delta+ F^\delta(u,\theta_{u,s}^\delta(y))\right] \,  du.
\end{equation*}
Then, there exists a constants $c_0:=c_0(T)>0$ s.t., denoting for $0\le t\le s\le T  $ by $J_{t,s}^\delta (y)$ the  Jacobian matrix associated with the mapping $y\mapsto \theta_{s,t}^\delta (y) $
 $$ {\rm det}(J_{t,s}^{\delta}(y))\ge c_0 .$$
 Importantly, $c_0$ does not depend on $\delta$.
\end{lemma}
\begin{proof}
Let us first mention that the even though the coefficients $F^\delta$ are smooth, the above control is not direct because there is a subtle balance between the mollifying, matrix valued, parameter $\delta$ and the length of the considered time interval $[t,s] $. We recall indeed that the entries $\delta_{ij} $ given in \eqref{Proof:Controls_on_Flows_Choice_delta_LOC_LEMMA} do depend on $s-t$.

We also recall that, similarly to \eqref{Proof:Controls_on_flows_mollifier1}, it holds that
\begin{equation}
\label{proof:control_Derivative_F}
|D_{x_j}F_1^\delta(t,z)| \, \le \,  C (\delta_{1j})^{\beta^1-1},\ \forall 2\le i\le j\le n,\ |D_{x_j}F_i^\delta(t,z)| \, \le \,  C (\delta_{ij})^{\beta^j-1}.
\end{equation}
To prove the statement, we have thus to justify, somehow similarly to the control for the flows of Lemma \ref{lemma:bilip_control_flow}, that the explosive behavior of the Lipschitz moduli  is indeed well balanced by the time-integration.

Let us now start from the dynamics of $J^\delta (y)$ which writes:
\begin{align*}
J_{t,s}^\delta(y) = \ D_y\theta^\delta_{t,s}(y) &= \mathds{I} - \int_t^s \left[
\left(A_u +D_zF^\delta(u,z)|_{z=\theta^\delta_{u,s}(y)}\right)D_y\theta^\delta_{u,s}(y)\right] \, du\\
&= \mathds{I} - \int_t^s \left[
\left(A_u +D_zF^\delta(u,z)|_{z=\theta^\delta_{u,s}(y)}\right)J^\delta_{u,s}(y)\right] \, du.
\end{align*}
The above equation can be partially integrated using the resolvent $(R_{u,s})_{u\in [t,s]} $ associated with $A$, i.e. the $\R^N\otimes \R^N $ valued function satisfying
\begin{equation}\label{DYN_RES_STRUCT}
\frac{d}{du}R_{u,s}=A_uR_{u,s},\ R_{s,s}=I_{nd\times nd}.
\end{equation}
This yields:
\begin{equation}
\label{LIN_DYN_PARTIALLY_INTEGRATED}
J_{t,s}^\delta(y) =R_{t,s}-\int_t^s R_{t,u}D_zF^\delta(u,z)|_{z=\theta^\delta_{u,s}(y)}J^\delta_{u,s}(y)du.
\end{equation}

We actually have the following important structure property of the resolvent
$(R_{u,s})_{u\in [t,s]} $. There exists a non-degenerate family of  matrices $(\hat{R}_{\frac{u-t}{s-t}}^{t,s})_{u\in [t,s]}$, which  is bounded uniformly on $u\in [t,s] $ with constants depending on $T$ s.t.
\begin{equation}\label{scaling_relation}
R_{u,s}= \mathbb{T}_{s-t} \hat{R}_{\frac{u-t}{s-t}}^{t,s} (\mathbb{T}_{s-t})^{-1}.
\end{equation}
Indeed, setting for all $v\in [0,1],\ \hat R_v^{t,s}:=(\mathbb{T}_{s-t})^{-1}R_{t+v(s-t),s} \mathbb{T}_{s-t}$ and differentiating yields:
\begin{eqnarray*}
\partial_v \hat R_v^{t,s}&=&(s-t)(\mathbb{T}_{s-t})^{-1}A_{t+v(s-t)}R_{t+v(s-t),s} \mathbb{T}_{s-t}\\
&=&\biggl((s-t)(\mathbb{T}_{s-t})^{-1}A_{t+v(s-t)}\mathbb{T}_{s-t}\biggr)\hat R_v^{t,s}:=A_v^{t,s} \hat R_v^{t,s}.
\end{eqnarray*}
The identity \eqref{scaling_relation} then actually follows from the structure of the matrix $A_t$ (see assumption \textbf{[H]} and \eqref{eq:def_matrix_A}) which ensures that $(A_v^{t,s})_{v\in [0,1]} $ has bounded entries.

As a by-product of \eqref{scaling_relation}, we derive that there exists $C\ge 1 $ s.t. for all $(i,j)\in \llbracket 1,n \rrbracket$,
\begin{equation}\label{THE_BOUND_RES}
|(R_{t,u})_{ij}|\le C(\mathds{1}_{j\ge i}+(s-t)^{i-j}\mathds{1}_{i>j}).
\end{equation}
From \eqref{LIN_DYN_PARTIALLY_INTEGRATED} we thus derive
\begin{align}
|J_{t,s}^\delta(y)|\le& C+\int_t^s \sum_{i,j,k=1}^n \big|R_{t,u}D_zF^\delta(u,z)|_{z=\theta^\delta_{u,s}(y)}\big|_{ik} |J^\delta_{u,s}(y)|_{kj}du \notag\\
\le & C+\int_t^s \sum_{i,j,k=1}^n \big|R_{t,u} D_zF^\delta(u,z)|_{z=\theta^\delta_{u,s}(y)}\big|_{ik} |J^\delta_{u,s}(y)|_{kj}du.\label{SEMI_INT_2}
\end{align}
Remember now that $D_z F^\delta(u,z) $ is upper triangular. Then for fixed $(j,k)\in \llbracket 1,n \rrbracket^2 $, using \eqref{THE_BOUND_RES},
\begin{align}
\big|R_{t,u}D_zF^\delta(u,z)|_{z=\theta^\delta_{u,s}(y)}\big|_{ik}\le& \sum_{\ell=1}^k |R_{t,u}|_{i\ell}|DF^\delta_{\ell k}|_\infty\notag\\
\le &C \sum_{\ell=1}^k (\mathds{1}_{\ell\ge i}+(t-s)^{i-l}\mathds{1}_{\ell<i})|D_kF_\ell^{\delta}|_\infty.
\end{align}
It is now clearly seen that, for a fixed line index $i$ and $\ell\ge i $, there is no time regularity, contrarily to what happened with the control of the renormalized flows. Recall that if we had chosen
$F^\delta$ as in the proof of Lemma \ref{lemma:bilip_control_flow} then, for $\ell\ge 2 $ (recall that we regularize at macro scale $C_1$ for $F_1^\delta $) $|D_k F_\ell^\delta|_\infty \le C(\delta_{\ell k})^{-1+\beta_\ell^k}$. It is then clear that the $ \big(\delta_{lk}^{-1+\beta_\ell^k}\big)_{\ell \in \llbracket 2, k\rrbracket}$ must have the same order, which precisely prevents from the choice in \eqref{Proof:Controls_on_Flows_Choice_delta} which allows to consider minimal H\"older regularity exponents distinguishing the regularity with respect to the $k^{{\rm th}} $ variable in function of the level $\ell $ of the chain. We are here  led to consider $\beta_{\ell}^k =\beta_k^k =\beta^k$ (condition \eqref{Proof:Controls_on_Flows_Choice_delta_LOC_LEMMA}), imposing the strongest integrability threshold, associated with the diagonal perturbation at level $k$ all along the previous levels (up to the second one), which in principle lead to less singularity when the corresponding gradients are considered.

Such a phenomenon naturally appears when investigating the strong uniqueness of the SDE because of the Zvonkin approach, see e.g. \cite{Hao:Wu:Zhang19} for the Kinetic case deriving from our framework or \cite{Chaudru17} for the kinetic Brownian case. It was also the case, still for the Brownian kinetic case, in \cite{Chaudru18} where the parametrix approach freezing the initial coefficients was considered. The author had to impose the same regularity for the drift, in the degenerate variable, on the whole $F$. Hence, adapting the work \cite{Marino20} to derive pointwise bound of the gradients, which could have been another approach would have led to the same constraints. Here, we have slightly more freedom since we manage to have arbitrary smoothness indexes for the non-degenerate component of the drift.

We thus derive from \eqref{SEMI_INT_2} and for $\bar C,  C_1$ large enough there exists $c_0>0$ such that
\[\left[ \sum_{k=2}^n\sum_{\ell=2}^k(\delta_{\ell k})^{-1+\beta^k}+\sum_{k=1}^n(\delta_{1k})^{-1+\beta_1^k}\right](s-t)\le c_0\]
meant to be small that, under the current assumptions, there exists $C\ge 1 $ s.t.
\begin{equation*}
|J_{t,s}^\delta(y)|\le C\exp(c_0),
\end{equation*}
and similarly, $\forall u\in [t,s]$,
\begin{equation}\label{CTR_BD_SUP}
|J_{u,s}^\delta(y)|\le C\exp(c_0).
\end{equation}
Rewriting:
\begin{align*}
J_{t,s}^\delta(y) =R_{t,s}\big(I-\int_t^s R_{s,u}D_zF^\delta(u,z)|_{z=\theta^\delta_{u,s}(y)}J^\delta_{u,s}(y)du\big),
\end{align*}
we derive from \eqref{THE_BOUND_RES}, \eqref{CTR_BD_SUP} that the matrix $\big(I-\int_t^s R_{s,u}D_zF^\delta(u,z)|_{z=\theta^\delta_{u,s}(y)}J^\delta_{u,s}(y)du\big)$
 has diagonal dominant and this gives, from the non degeneracy of $R$, the statement concerning the determinant.
\end{proof}

\setcounter{equation}{0}

\section*{Acknowledgment}
For the first author, the work was supported by a public grant ($2018-0024$H) as part of the FMJH project. The research of the second author was funded by the Russian Science Foundation project (project No. $20-11-20119$).

\bibliography{bibli}
\bibliographystyle{abbrv}
\end{document}

%% file: preamble.tex
\usepackage{lineno} \modulolinenumbers[5]
\usepackage[a4paper]{geometry}
\usepackage{amsfonts}
\usepackage{parskip}
\usepackage{xtab}
\usepackage{setspace}
\usepackage{fancyhdr}
\usepackage[dvips]{graphicx}
\usepackage[T1]{fontenc}
\usepackage[latin1]{inputenc}
\usepackage[main=english,italian]{babel}
\usepackage{type1ec}
\usepackage[final]{microtype}
\usepackage{amsmath,amssymb,amsthm,eucal,dsfont,bm,mathrsfs,stmaryrd}
\usepackage{lmodern}
\usepackage{textcomp}
\usepackage{pict2e,enumitem}
\usepackage{hyperref}
\usepackage[nodayofweek]{datetime} 
\usepackage{comment}
\usepackage{authblk}

\hypersetup{
    pdfpagemode={UseOutlines},
    bookmarksopen,
    pdfstartview={FitH},
    colorlinks,
    linkcolor={black},
    citecolor={black},
    urlcolor={black}
            }

\usepackage{geometry}

\newdateformat{monthyear}{\monthname[\THEMONTH], \THEYEAR}

\DeclareMathOperator{\dom}{dom}

\DeclareMathOperator{\Id}{Id}

\newcommand{\N}{\ensuremath{\mathbb{N}}}
\newcommand{\R}{\ensuremath{\mathbb{R}}}
\newcommand{\red}{\textcolor{red}}

\theoremstyle{definition}
        \newtheorem{definition}{Definition}[section]
        \newtheorem{remark}[definition]{Remark}

\theoremstyle{plain}
        \newtheorem{theorem}[definition]{Theorem}
        \newtheorem{lemma}[definition]{Lemma}
        \newtheorem{corollary}[definition]{Corollary}
        \newtheorem{prop}[definition]{Proposition}
